\documentclass{amsart}
\usepackage{graphicx}
\usepackage{amsmath}
\usepackage{amssymb}
\usepackage{amsthm}

\newtheorem{theorem}{Theorem}[section]
\newtheorem{lem}[theorem]{Lemma}

\theoremstyle{Corollary}

\newtheorem{prop}[theorem]{Proposition}

\numberwithin{equation}{section}

\begin{document}

\title[Webster scalar curvature flow II]{The Webster scalar curvature flow on CR sphere. Part II}

\author{Pak Tung Ho}
\address{Department of Mathematics, Sogang University, Seoul
121-742, Korea}

\email{ptho@sogang.ac.kr, paktungho@yahoo.com.hk}

\subjclass[2000]{Primary 32V20, 53C44; Secondary 53C21, 35R01}

\date{January 25, 2013.}

\keywords{Webster scalar curvature, CR sphere, CR Yamabe problem}

\begin{abstract}
This is the second of two papers,
in which we study the problem of prescribing Webster scalar curvature on the CR sphere
as a given function $f$. Using the Webster scalar curvature flow,
we prove an existence result under suitable assumptions on the Morse indices of $f$.
\end{abstract}

\maketitle

\section{Introduction}

Suppose $(M,g_0)$ is a compact $n$-dimensional Riemannian manifold without boundary, where $n\geq 3$.
Given a function $f$ on $M$, the problem of prescribing scalar curvature
is  to find a metric $g$ conformal to $g_0$ such that $R_g=f$. When $f$ is constant, it is the Yamabe problem,
which was solved by Trudinger \cite{Trudinger}, Aubin \cite{Aubin0}, and Schoen \cite{Schoen}. When
$(M,g_0)$ is the $n$-dimensional sphere $S^n$ with $g_0$ being the standard metric in $S^n$, it is the so-called Nirenberg's problem and was studied in
\cite{Chang&Gursky&Yang,Chang&Yang5,Chang&Yang3,Chang&Yang4,Kazdan&Warner2,Struwe}. Kazdan and Warner \cite{Kazdan&Warner2},
using a clever integration by parts, found a necessary condition, which is now known as Kazdan-Warner condition. More precisely,
they showed that if $f$ can be prescribed as the scalar curvature of a metric $g=u^{\frac{4}{n-2}}g_0$, then
$$\int_{S^n}\langle\nabla_{g_0}f,\nabla_{g_0}x_i\rangle_{g_0}u^{\frac{2n}{n-2}}dV_{g_0}=0\hspace{2mm}\mbox{ for }i=1,2,...,n+1,$$
where $x_i$ is the coordinate function of $\mathbb{R}^{n+1}$ restricted to $S^n$.
Later, Chang and Yang \cite{Chang&Yang5} proved the following (see also \cite{Chang&Gursky&Yang}):

\begin{theorem}[Chang-Yang]\label{thm1.3}
Suppose that $f$ is a smooth positive Morse function with only non-degenerate critical points and satisfies the
degree condition:
\begin{equation*}
\sum_{\nabla_{g_0}f(x),\Delta_{g_0}f(x)<0}(-1)^{ind(f,x)}\neq -1.
\end{equation*}
If $\|f-n(n+1)\|_{C^0(S^n)}$ is sufficiently small,
then there exists a metric $g$ conformal to $g_0$ such that its scalar curvature $R_g=f$.
\end{theorem}

Using the scalar curvature flow, Chen and Xu \cite{Chen&Xu} was able to estimate how small $\|f-n(n+1)\|_{C^0(S^n)}$ should be. More precisely,
they proved the following:
\begin{theorem}[Chen and Xu \cite{Chen&Xu}]
Suppose that $f$ is a smooth positive function on the $n$-dimensional sphere $S^{n}$ with only non-degenerate critical points
with Morse indices $ind(f,x)$ and such
that $\Delta_{g_0}f (x)\neq 0$ at any such point $x$. Let
$$m_i =\#\{x\in S^{n}: \nabla_{g_0}f(x) = 0, \Delta_{g_0}f(x)< 0, ind(f,x) = n-i\}.$$
Furthermore, suppose $\delta_n = 2^{\frac{2}{n}}$
if $3\leq n\leq 4$ or
$=2^{\frac{2}{n-2}}$
 for $n\geq 5$.
If there is no solution with coefficient $k_i\geq 0$ to the system of equations
\begin{equation*}
m_0 = 1 + k_0, m_i = k_{i-1} + k_i\mbox{ for }1\leq i\leq n, k_{n} = 0,
\end{equation*}
and
$f$ satisfies
\begin{equation*}
\max_{S^{n}}f/\min_{S^{n}}f<\delta_n,
\end{equation*}
then $f$ can be realized as the scalar curvature of some metric conformal to the standard metric
$g_0$.
\end{theorem}

In this paper, we are interested in the problem of prescribing Webster scalar curvature. More precisely,
suppose $(M,\theta_0)$ is a compact strongly pseudoconvex CR manifold of real dimensional $2n+1$ with a given contact form $\theta_0$.
We are interested in the following question:
can we find a contact form $\theta$ conformal to $\theta_0$ such that its Webster scalar curvature
$R_{\theta}=f$? This has been studied in \cite{Chtioui&Elmehdi&Gamara,Felli&Uguzzoni,Gamara3,Ho1,Malchiodi&Uguzzoni,Riahia&Gamara,Salem&Gamara}.
When $f$ is constant, this is the CR Yamabe problem, which was solved by Jerison and Lee in \cite{Jerison&Lee1,Jerison&Lee2,Jerison&Lee3},
and by Gamara and Yacoub in \cite{Gamara2,Gamara1}. As an analogy of Nirenberg's problem, we want to study the problem of prescribing  Webster scalar curvature on the CR sphere $(S^{2n+1},\theta_0)$.

From now on, we assume that $f$ is a smooth positive Morse function on $S^{2n+1}$ with only non-degenerate critical points in the sense that
$\Delta_{\theta_0}f(x)\neq 0$ whenever $f'(x)=0$. Here $f'(x)$ denotes the gradient of $f$ with respect to the standard Riemannian metric on
$S^{2n+1}$. In \cite{Malchiodi&Uguzzoni}, Malchiodi and Uguzzoni proved the following:{\footnote {Note that
Theorem \ref{thm1.2} in \cite{Malchiodi&Uguzzoni} was stated in terms of Heisenberg group $\mathbb{H}^n$. But one can easily see that the statement here is
equivalent to theirs.}}
\begin{theorem}[Malchiodi and Uguzzoni \cite{Malchiodi&Uguzzoni}]\label{thm1.2}
If $f$ satisfies
\begin{equation}\label{1.2}
\sum_{f'(x)=0,\,\Delta_{\theta_0}f(x)<0}(-1)^{ind(f,x)}\neq -1,
\end{equation}
where $ind(f,x)$ denotes the Morse index of $f$ at $x$, then $f$ can be realized as the Webster scalar curvature of some contact form conformal to $\theta_0$,
provided that $f$ is sufficiently closed to the Webster scalar curvature of the standard contact form on $S^{2n+1}$ in sup norm.
\end{theorem}

This is CR version of Theorem \ref{thm1.3}.
It is important to know how large the difference in sup norm can possibly be.
To answer this question, we follow the argument of Chen-Xu in \cite{Chen&Xu} and consider the Webster scalar curvature flow.
By using the Webster scalar curvature flow, we prove
the following theorem, which is our main result:

\begin{theorem}\label{thm1.1}
Suppose that $n\geq 2$ and $f$ is a smooth positive function on $S^{2n+1}$ with only non-degenerate critical points
with Morse indices $ind(f,x)$ and such
that $\Delta_{\theta_0}f (x)\neq 0$ at any such point $x$. Let
\begin{equation}\label{1.0}
m_i =\#\{x\in S^{2n+1}: f'(x) = 0, \Delta_{\theta_0}f(x)< 0, ind(f,x) = 2n+1-i\}.
\end{equation}
If there is no solution with coefficient $k_i\geq 0$ to the system of equations
\begin{equation}\label{1.1}
m_0 = 1 + k_0, m_i = k_{i-1} + k_i\mbox{ for }1\leq i\leq 2n+1, k_{2n+1} = 0,
\end{equation}
and
$f$ satisfies the simple bubble condition, namely
\begin{equation}\tag{sbc}\label{sbc}
\max_{S^{2n+1}}f/\min_{S^{2n+1}}f<2^{\frac{1}{n}},
\end{equation}
then $f$ can be realized as the Webster scalar curvature of some contact form conformal to
$\theta_0$.
\end{theorem}

We remark that Theorem \ref{thm1.1} in fact implies Theorem \ref{thm1.2}. See the remark after the proof of
Theorem \ref{thm1.1} in section \ref{section7}.

This project began when I was invited by Prof. Paul Yang to visit Princeton University in the summer of 2011.
I would like to thank Prof. Paul Yang, without his support this paper would not have been possible.
I would like to thank Prof. Xingwang Xu, who answered many of my questions about his paper.
I am also grateful to Prof. Sai-Kee Yeung for his continuous encouragement and support for the last few years.
I would like to thank Prof. Jih-Hsin Cheng for the invitation to the Institute of Mathematics, Academia Sinica in the January of 2013,
where part of this work was done. This work was supported by the National Research Foundation of Korea (NRF) grant funded
by the Korea government (MEST) (No.2012R1A1A1004274).

\section{The Webster scalar curvature flow}

Let  $\theta_0$ be the standard contact form on the sphere
$S^{2n+1}=\{x=(x_1,...,x_{n+1}):|x|^2=1\}\subset\mathbb{C}^{n+1}$, i.e.
$$\theta_{0}=\sqrt{-1}(\overline{\partial}-\partial)|x|^2
=\sqrt{-1}\sum_{j=1}^{n+1}(x_{j}d\overline{x}_{j}-\overline{x}_{j}dx_{j}).$$
Then $(S^{2n+1},\theta_0)$ is a compact strictly pseudoconvex CR
manifold of real dimension $2n+1$.
Suppose $f$ is a smooth positive function on $S^{2n+1}$. Let $u_0\in C^\infty(S^{2n+1})$ such that
\begin{equation}\label{2.01}
\int_{S^{2n+1}}u_0^{2+\frac{2}{n}}dV_{\theta_0}=\int_{S^{2n+1}}dV_{\theta_0}.
\end{equation}
We introduced the Webster scalar curvature flow in part I \cite{Ho3}, which
is defined as the evolution of the contact form
$\theta=\theta(t)$, $t\geq 0$ as follows:
\begin{equation}\label{2.1}
\frac{\partial }{\partial t}\theta=(\alpha f- R_{\theta})\theta,
\hspace{4mm} \theta\big|_{t=0}=u_0^{\frac{2}{n}}\theta_0,
\end{equation}
where $R_\theta$ is the Webster scalar curvature  of the contact form
$\theta$ and $\alpha=\alpha(t)$ is given by
\begin{equation}\label{2.2}
\alpha\int_{S^{2n+1}}f dV_{\theta}=\int_{S^{2n+1}}R_\theta dV_{\theta}.
\end{equation}
If we write $\theta=u^{\frac{2}{n}}\theta_0$ where $u=u(t)$,
then (\ref{2.1}) is equivalent to
the following evolution equation of the conformal factor $u$:
\begin{equation}\label{2.3}
\frac{\partial u}{\partial t}=\frac{n}{2}(\alpha f- R_{\theta})u,
\hspace{4mm} u|_{t=0}=u_0.
\end{equation}
Since $\theta=u^{\frac{2}{n}}\theta_0$, the Webster scalar curvature
$R_\theta$ of $\theta$ satisfies the following CR Yamabe equation
\begin{equation}\label{2.4}
R_\theta=u^{-(1+\frac{2}{n})}\left(-(2+\frac{2}{n})\Delta_{\theta_0}u+R_{\theta_0}u\right),
\end{equation}
where $R_{\theta_0}=n(n+1)/2$ is the Webster scalar curvature of
$\theta_0$.

We recall some of the results we have proved in part I. In \cite{Ho3},
we established the long-time existence of the flow (\ref{2.1}). See section 2.3 in \cite{Ho3}. Define
\begin{equation}\label{2.5}
E(u)=\int_{S^{2n+1}}\left((2+\frac{2}{n})|\nabla_{\theta_0}u|^2_{\theta_0}
+R_{\theta_0}u^2\right)dV_{\theta_0}=\int_{S^{2n+1}}R_\theta dV_\theta
\end{equation}
where the last equality follows from (\ref{2.4}). We also define
\begin{equation}\label{2.7}
E_f(u)=\frac{E(u)}{(\int_{S^{2n+1}}fu^{2+\frac{2}{n}}dV_{\theta_0})^{\frac{n}{n+1}}}.
\end{equation}
We have proved in part I that (see Proposition 2.2 in \cite{Ho3}):
\begin{prop}\label{prop2.2}
The functional $E_f$ is non-increasing along the flow $(\ref{2.3})$. Indeed,
$$\frac{d}{dt}E_f(u)=-n\int_{S^{2n+1}}(\alpha f-R_\theta)^2u^{2+\frac{2}{n}}dV_{\theta_0}\Big/
\left(\int_{S^{2n+1}}fu^{2+\frac{2}{n}}dV_{\theta_0}\right)^{\frac{n}{n+1}}\leq 0.$$
\end{prop}
Recall
the definition of the normalized contact form:
For every smooth positive function $u(t)$, set
$P(t)=\displaystyle\int_{S^{2n+1}}x
u(t)^{2+\frac{2}{n}}dV_{\theta_0}$ where $x=(x_1,...,x_{n+1})\in
S^{2n+1}\subset\mathbb{C}^{n+1}$, and we define
\begin{equation}\label{cm}
\widehat{P(t)}=\displaystyle\frac{P(t)}{\|P(t)\|}\mbox{ if }\|P(t)\|\neq 0,
\mbox{ otherwise }\widehat{P(t)}=P(t).
\end{equation}
Clearly $\widehat{P(t)}\in S^{2n+1}$ smoothly depends
on the time $t$ if $u$ does. There exists a family of conformal CR
diffeomorphisms $\phi(t): S^{2n+1}\rightarrow S^{2n+1}$ such that
(see \cite{Frank&Lieb})
\begin{equation}\label{4.19}
\int_{S^{2n+1}}x\,dV_h=(0,...,0)\in\mathbb{C}^{n+1}\hspace{2mm}\mbox{ for all }t>0,
\end{equation}
where the new contact form
\begin{equation}\label{4.19a}
h=h(t)=\phi(t)^*\big(\theta(t)\big)=v(t)^{2+\frac{2}{n}}\theta_0
\end{equation}
is called the normalized contact form with
$v=v(t)=(u(t)\circ\phi(t))\big|\det(d\phi(t))\big|^{\frac{n}{2n+2}}$
and the volume form $dV_h=v(t)^{2+\frac{2}{n}}dV_{\theta_0}$. In
fact, the conformal CR diffeomorphism may be represented as
$\phi(t)=\phi_{p(t),r(t)}=\Psi\circ T_{p(t)}\circ D_{r(t)}\circ\pi$
for some $p(t) \in\mathbb{H}^n$ and $r(t)>0$. Here the CR
diffeomorphism $\pi:S^{2n+1}\setminus\{(0,...,0,-1)\}\rightarrow\mathbb{H}^n$
is given by
\begin{equation}\label{4.20}
\pi(x)=\left(\frac{x'}{1+x_{n+1}},Re\big(\sqrt{-1}\frac{1-x_{n+1}}{1+x_{n+1}}\big)\right)
, x=(x',x_{n+1})\in S^{2n+1},
\end{equation}
 where $\mathbb{H}^n$ denotes the
Heisenberg group, and
 $D_{\lambda},
T_{(z',\tau')}:\mathbb{H}^n\rightarrow\mathbb{H}^n$ are respectively
the dilation and translation on $\mathbb{H}^n$ given by
\begin{equation}\label{4.21}
D_{\lambda}(z,\tau)=(\lambda z, \lambda^2\tau)\mbox{ and }
T_{(z',\tau')}(z,\tau)=(z+z',\tau+\tau'+2Im(z'\cdot
\overline{z}))\mbox{ for }(z,\tau)\in\mathbb{H}^n.
\end{equation}
And $\Psi=\pi^{-1}$ is the inverse of $\pi$.

Meanwhile, the normalized function $v$ satisfies
\begin{equation}\label{4.22}
-(2+\frac{2}{n})\Delta_{\theta_0}v+R_{\theta_0}v=R_hv^{1+\frac{2}{n}},
\end{equation}
where $R_h=R_\theta\circ\phi(t)$ is the Webster scalar curvature of
the normalized contact form $h=h(t)$ in view of (\ref{4.19a}). Hereafter, we set
$f_\phi=f\circ\phi$.

Let $\delta_n=\max_{S^{2n+1}}f/\min_{S^{2n+1}}f.$ By assumption (\ref{sbc}) in Theorem \ref{thm1.1}, we
have $\delta_n<2^{\frac{1}{n}}.$ Then there exists $\epsilon_0>0$
such that
$$\displaystyle\frac{\delta_n^{\frac{n}{n+1}}}{2^{\frac{1}{n+1}}}=\frac{1-\epsilon_0}{1+\epsilon_0}.$$
In particular,
$(1+\epsilon_0)\delta_n^{\frac{n}{n+1}}<2^{\frac{1}{n+1}}$.
 Set
\begin{equation}\label{4.15}
\beta=(1+\epsilon_0)Y(S^{2n+1},\theta_0)\left(\min_{S^{2n+1}}f\right)^{-\frac{n}{n+1}}.
\end{equation}

The following was proved in part I. See Theorem 4.7 in \cite{Ho3}.

\begin{theorem}\label{thm4.7}
For any given $u_0$ satisfying $(\ref{2.01})$ with
$E_f(u_0)\leq\beta$ with $\beta$ defined as in $(\ref{4.15})$, consider the flow $\theta(t)$ defined in $(\ref{2.1})$ with initial data
$u_0$. Let $\{t_k\}$ be a time sequence of the flow with
$t_k\rightarrow\infty$ as $k\rightarrow\infty$. Let $\{\theta_k\}$
be the corresponding contact forms such that
$\theta_k=u(t_k)^{\frac{2}{n}}\theta_0$. Assume that
$\|R_{\theta_k}-R_\infty\|_{L^{p_1}(S^{2n+1},\theta_k)}\rightarrow
0$ as $k\rightarrow\infty$ for some $p_1>n+1$ and a smooth function
$R_\infty>0$ defined on $S^{2n+1}$ which satisfies the simple bubble
condition (sbc):
$$\frac{\max_{S^{2n+1}}R_\infty}{\min_{S^{2n+1}}R_\infty}<2^{\frac{1}{n}}.$$
Then, up to a subsequence, either\\
\emph{(i)} $\{u_k\}$ is uniformly bounded in $S^p_2(S^{2n+1},\theta_0)$ for some $p\in(n+1,p_1)$.
 Furthermore, $u_k\rightarrow u_\infty$ in
$S^p_2(S^{2n+1},\theta_0)$ as $k\rightarrow\infty$, where $\theta_\infty=u_\infty^{\frac{2}{n}}\theta_0$
has Webster scalar curvature $R_\infty$, or\\
\emph{(ii)} let
$h_k=\phi(t_k)^*(\theta_k)=v_k^{\frac{2}{n}}\theta_0$ be the
associated sequence of the normalized contact forms satisfying
$\displaystyle\int_{S^{2n+1}}x\,dV_{h_k}=(0,...,0)\in\mathbb{C}^{n+1}$. Then, there exists
$Q\in S^{2n+1}$ such that
\begin{equation}\label{4.23}
dV_{\theta_k}\rightharpoonup\mbox{\emph{Vol}}(S^{2n+1},\theta_0)\delta_Q,\hspace{2mm}\mbox{
as }k\rightarrow\infty
\end{equation}
in the weak sense of measures. In addition, for any $\lambda\in (0,1)$, we have
\begin{equation}\label{4.24}
v_k\rightarrow 1\mbox{ in }
C^{1,\lambda}_P(S^{2n+1})\hspace{2mm}\mbox{ as }k\rightarrow\infty.
\end{equation}
Here $C^{1,\lambda}_P(S^{2n+1})$ is the parabolic H\"{o}rmander
H\"{o}lder spaces.
\end{theorem}

It follows from Theorem \ref{thm4.7} that we have the following dichotomy: Either
the flow converges in $S_2^p$ for some $p>2n+2$, and in this case, $f$ can be realized as the
Webster scalar curvature of some contact form conformal to
$\theta_0$ thanks to Lemma 4.8 in \cite{Ho3}, or
the corresponding normalized flow $h(t)$ defined in (\ref{4.19a}) converges.

Starting from now, we will assume that, with the initial data
$u_0\in C^\infty_f$ where
$$u_0\in C^\infty_f:=\{u\in C_*^\infty: u>0\mbox{ and }E_f(u)\leq \beta\}$$
with $\beta$ defined as (\ref{4.15}) and
$$C^\infty_*:=\Big\{0<u\in C^\infty(S^{2n+1}): \theta=u^{\frac{2}{n}}\theta_0
\mbox{ satisfies }\int_{S^{2n+1}}u^{2+\frac{2}{n}}dV_{\theta_0}
=\int_{S^{2n+1}}dV_{\theta_0}\Big\},$$
the flow (\ref{2.3}) does not converge and $f$
cannot be realized as the Webster scalar curvature in the conformal
class of $\theta_0$. So Theorem \ref{thm4.7} can always be applied
without further mention.

The following lemma was proved in part I. See Lemma 4.16 in \cite{Ho3}.

\begin{lem}\label{lem4.14}
Let $f:S^{2n+1}\rightarrow\mathbb{R}$ be a smooth positive
non-degenerate Morse function satisfying the simple bubble condition
(sbc):
$$\frac{\max_{S^{2n+1}}f}{\min_{S^{2n+1}}f}<2^{\frac{1}{n}}.$$
Suppose that $f$ cannot be realized as the the Webster scalar
curvature of any contact form conformal to $\theta_0$. Let $u(t)$ be
a smooth solution of $(\ref{2.3})$ with initial data $u_0\in
C_f^\infty$. Then there exists a family of CR diffeomorphism
$\phi(t)$ on $S^{2n+1}$ with the normailized contact form
$h(t)=v(t)^{\frac{2}{n}}\theta_0=\phi(t)^*\big(\theta(t)\big)$ such
that as $t\rightarrow\infty$
$$v(t)\rightarrow 1,\hspace{2mm}h(t)\rightarrow\theta_0\hspace{2mm}\mbox{ in }C^{1,\gamma}_P(S^{2n+1})$$
for any $\gamma\in(0,1)$, and $\phi(t)-\widehat{P(t)}\rightarrow 0$ in
$L^2(S^{2n+1},\theta_0)$. Moreover, as $t\rightarrow\infty$, we have
$$\|f\circ\phi(t)-f(\widehat{P(t)})\|_{L^2(S^{2n+1},\theta_0)}\rightarrow 0\hspace{2mm}\mbox{ and
}\hspace{2mm}\alpha(t)f(\widehat{P(t)})\rightarrow R_{\theta_0}.$$
Here $\widehat{P(t)}$ is defined as in $(\ref{cm})$.
\end{lem}

\section{Analysis on the vector field $\xi=(d\phi)^{-1}\displaystyle\frac{d\phi}{dt}$}\label{section5}

\subsection{Normalized curvature flow}

Let us start with the normalized flow defined in (\ref{4.19}) and (\ref{4.19a}), which satisfies (\ref{4.22}). J. H. Cheng proved the following
Kazdan-Warner type condition in \cite{Cheng}:
\begin{equation}\label{5.1}
\int_{S^{2n+1}}\langle\nabla_{\theta_0}x,
\nabla_{\theta_0}R_h\rangle_{\theta_0}dV_h=(0,...,0)\mbox{ and
}\int_{S^{2n+1}}\langle\nabla_{\theta_0}\overline{x},
\nabla_{\theta_0}R_h\rangle_{\theta_0}dV_h=(0,...,0),
\end{equation}
where $x=(x_1,...,x_{n+1})\in S^{2n+1}\subset\mathbb{C}^{n+1}$ and
$\overline{x}=(\overline{x}_1,...,\overline{x}_{n+1})$.

For the corresponding conformal CR diffeomorphism of the normalized flow (\ref{4.19}), we let
$\phi(t)=(\phi_1(t),...,\phi_{n+1}(t))\in
S^{2n+1}\subset\mathbb{C}^{n+1}$. We define $\xi=(d\phi)^{-1}\displaystyle\frac{d\phi}{dt}.$
Recall that $v(t)=(u(t)\circ\phi(t))\big|\det(d\phi(t))\big|^{\frac{n}{2n+2}}$. Differentiating it with respect to
$t$ and using (\ref{2.3}), we obtain
\begin{equation}\label{5.a}
v_t=\frac{n}{2}(\alpha f_\phi-R_h)v+\frac{n}{2n+2}v\,\mbox{div}'_{h}(\xi),
\end{equation}
where $\mbox{div}'_h$ is the subdivergence operator of type $(1,0)$ with respect to the contact form $h$ (see \cite{Cheng} for the definition).
Differentiating (\ref{4.19}) with respect to $t$ and using (\ref{5.a}), we get
\begin{equation}\label{5.b}
\begin{split}
(0,...,0)&=\frac{d}{dt}\left(\int_{S^{2n+1}}x\,dV_h\right)=\frac{2n+2}{n}\left(\int_{S^{2n+1}}xv_tv^{1+\frac{2}{n}}\,dV_{\theta_0}\right)\\
&=(n+1)\int_{S^{2n+1}}x(\alpha f_\phi-R_h)dV_h+\int_{S^{2n+1}}x\,\mbox{div}'_{h}(\xi)dV_h\\
&=(n+1)\int_{S^{2n+1}}x(\alpha f_\phi-R_h)dV_h+\int_{S^{2n+1}}x\,\mbox{div}'_{\theta_0}(v^{2+\frac{2}{n}}\xi)dV_{\theta_0}\\
&=(n+1)\int_{S^{2n+1}}x(\alpha f_\phi-R_h)dV_h-\int_{S^{2n+1}}\xi\,dV_{h}.\\
\end{split}
\end{equation}

\subsection{Cayley transform}

The Cayley transform is the CR diffeomorphism
$\pi:S^{2n+1}\setminus\{S\}\rightarrow\mathbb{H}^n$ given in  (\ref{4.20}), i.e.
$$
\pi(x)=\left(\frac{x_1}{1+x_{n+1}},\cdots,\frac{x_n}{1+x_{n+1}},
Re\big(\sqrt{-1}\frac{1-x_{n+1}}{1+x_{n+1}}\big)\right),
$$
where $x=(x_1,\cdots,x_n,x_{n+1})\in S^{2n+1}\setminus\{S\}$, where $S=(0,...,0,-1)$ and
$\mathbb{H}^n$ is the Heisenberg group. Note that
$\Psi=\pi^{-1}:\mathbb{H}^n\rightarrow S^{2n+1}$ is given by
\begin{equation}\label{5.6}
\Psi(z,\tau)=\left(\frac{2 z}{1+|z|^2-\sqrt{-1}\tau},
\frac{1-|z|^2+\sqrt{-1}\tau}{1+|z|^2-\sqrt{-1}\tau}\right)
\end{equation} where
$(z,\tau)\in\mathbb{H}^n\subset\mathbb{C}^n\times\mathbb{R}$. If
we write $\Psi=(\Psi_1,\cdots,\Psi_{n+1})\in
S^{2n+1}\subset\mathbb{C}^{n+1}$, and
$(z,\tau)=(z_1,\cdots,z_n,\tau)
=(a_1+\sqrt{-1}b_1,\cdots,a_n+\sqrt{-1}b_n,\tau)\in\mathbb{H}^n\subset\mathbb{C}^n\times\mathbb{R}$, then
\begin{equation}\label{5.7}
\begin{split}
&\frac{\partial\Psi_i}{\partial
a_j}=\frac{2\delta_{ij}}{1+|z|^2-\sqrt{-1}\tau}-\frac{4(a_i+\sqrt{-1}b_i)a_j}{(1+|z|^2-\sqrt{-1}\tau)^2},\\
&\frac{\partial\Psi_i}{\partial
b_j}=\frac{2\delta_{ij}\sqrt{-1}}{1+|z|^2-\sqrt{-1}\tau}-\frac{4(a_i+\sqrt{-1}b_i)b_j}{(1+|z|^2-\sqrt{-1}\tau)^2},
\frac{\partial\Psi_i}{\partial
\tau}=\frac{2\sqrt{-1}(a_i+\sqrt{-1}b_i)}{(1+|z|^2-\sqrt{-1}\tau)^2},
\end{split}
\end{equation}
for $1\leq i,j\leq n$, and
\begin{equation}\label{5.8}
\begin{split}
\frac{\partial\Psi_{n+1}}{\partial
a_j}&=-\frac{4a_j}{(1+|z|^2-\sqrt{-1}\tau)^2},\\
\frac{\partial\Psi_{n+1}}{\partial
b_j}&=-\frac{4b_j}{(1+|z|^2-\sqrt{-1}\tau)^2},
\frac{\partial\Psi_{n+1}}{\partial
\tau}=\frac{2\sqrt{-1}}{(1+|z|^2-\sqrt{-1}\tau)^2},
\end{split}
\end{equation}
for $1\leq j\leq n$. Note that $\displaystyle
X_j=\frac{\partial}{\partial a_j}+2b_j\frac{\partial}{\partial\tau},
Y_j=\frac{\partial}{\partial b_j}-2a_j\frac{\partial}{\partial\tau},
T=\frac{\partial}{\partial\tau}$, where $1\leq j\leq n$, is a basis
for the tangent space of $\mathbb{H}^n$. By (\ref{5.6}), (\ref{5.7})
and (\ref{5.8}), we have
\begin{equation}\label{5.9}
\begin{split}
X_j(\Psi_i)&=\frac{2\delta_{ij}}{1+|z|^2-\sqrt{-1}\tau}-\frac{4z_i\overline{z}_j}{(1+|z|^2-\sqrt{-1}\tau)^2},\\
Y_j(\Psi_i)&=\frac{2\sqrt{-1}\delta_{ij}}{1+|z|^2-\sqrt{-1}\tau}
-\frac{4\sqrt{-1}z_i\overline{z}_j}{(1+|z|^2-\sqrt{-1}\tau)^2},\\
T(\Psi_i)&=\frac{2\sqrt{-1}z_i}{(1+|z|^2-\sqrt{-1}\tau)^2}
\end{split}
\end{equation}
for $1\leq i,j\leq n$, and
\begin{equation}\label{5.10}
\begin{split}
X_j(\Psi_{n+1})&=-\frac{4\overline{z}_j}{(1+|z|^2-\sqrt{-1}\tau)^2}
,\\
Y_j(\Psi_{n+1})&=-\frac{4\sqrt{-1}\overline{z}_j}{(1+|z|^2-\sqrt{-1}\tau)^2}
,
T(\Psi_{n+1})=\frac{2\sqrt{-1}}{(1+|z|^2-\sqrt{-1}\tau)^2}
\end{split}
\end{equation}
for $1\leq j\leq n$, where $\Psi=(\Psi_1,\cdots,\Psi_{n+1})$.
Recall that for $r>0$ the dilation
$D_{r}:\mathbb{H}^n\rightarrow\mathbb{H}^n$ and for
$q=(z',\tau')\in\mathbb{H}^n$ the translation
$T_{(z',\tau')}:\mathbb{H}^n\rightarrow\mathbb{H}^n$ are
respectively given by
$$
D_{r}(z,\tau)=(r z, r^2\tau)\mbox{ and }
T_{q}(z,\tau)=(z+z',\tau+\tau'+2Im(z'\cdot \overline{z}))\mbox{ for
}(z,\tau)\in\mathbb{H}^n.
$$
If we define $\delta_{q,r}:\mathbb{H}^n\rightarrow\mathbb{H}^n$ as
$\delta_{q,r}=T_{q}\circ D_{r}$, i.e.
\begin{equation}\label{5.11}
\delta_{q,r}(z,\tau)=(rz+z',r^2\tau+\tau'+2rIm(z'\cdot
\overline{z}))\mbox{ for }(z,\tau)\in\mathbb{H}^n,
\end{equation}
then we have
\begin{equation}\label{5.12}
\begin{split}
d\delta_{q,r}(X_j)&=d\delta_{q,r}\left(\frac{\partial}{\partial
a_j}+2b_j\frac{\partial}{\partial\tau}\right)=r\left(\frac{\partial}{\partial
a_j}+2(rb_j+b_j')\frac{\partial}{\partial\tau}\right),\\
d\delta_{q,r}(Y_j)&=d\delta_{q,r}\left(\frac{\partial}{\partial
b_j}-2a_j\frac{\partial}{\partial\tau}\right)=r\left(\frac{\partial}{\partial
b_j}-2(ra_j+a_j')\frac{\partial}{\partial\tau}\right),\\
d\delta_{q,r}(T)&=d\delta_{q,r}\left(\frac{\partial}{\partial\tau}\right)=r^2\frac{\partial}{\partial\tau},
\end{split}
\end{equation}
where $z'=(a_1'+\sqrt{-1}b_1',\cdots,a_n'+\sqrt{-1}b_n')$.

Now recall that the CR diffeomorphism
$\phi=\phi(t):S^{2n+1}\rightarrow S^{2n+1}$ is given by
$\phi=\Psi\circ\delta_{q(t),r(t)}\circ\pi$ for some
$q(t)=(z(t),\tau(t))=(a_1(t)+\sqrt{-1}b_1(t),\cdots,a_n(t)+\sqrt{-1}b_n(t),\tau(t))
\in\mathbb{H}^n$ and $r(t)>0$. Therefore, by (\ref{5.11}), we have
\begin{equation}\label{5.13}
\begin{split}
&\frac{d}{dt}\delta_{q(t),r(t)}(z,\tau)\\
&=\frac{dr(t)}{dt}
\sum_{j=1}^n\left[a_j\left(\frac{\partial}{\partial
a_j}+2(r(t)b_j+b_j(t))\frac{\partial}{\partial\tau}\right)+
b_j\left(\frac{\partial}{\partial
b_j}-2(r(t)a_j+a_j(t))\frac{\partial}{\partial\tau}\right)\right]\\
&\hspace{4mm}+2\frac{dr(t)}{dt}r(t)\tau\frac{\partial}{\partial\tau}
+\sum_{j=1}^n\left[\frac{da_j(t)}{dt}\left(\frac{\partial}{\partial
a_j}+2(r(t)b_j+b_j(t))\frac{\partial}{\partial\tau}\right)\right.\\
&\hspace{4mm} +\frac{db_j(t)}{dt}\left(\frac{\partial}{\partial
b_j}-2(r(t)a_j+a_j(t))\frac{\partial}{\partial\tau}\right) +4r(t)
\left(a_j\frac{db_j(t)}{dt}-b_j\frac{da_j(t)}{dt}\right)\frac{\partial}{\partial\tau}\\
&\hspace{4mm}\left.+2\left(a_j(t)\frac{db_j(t)}{dt}-b_j(t)\frac{da_j(t)}{dt}\right)\frac{\partial}{\partial\tau}
\right]+\frac{d\tau(t)}{dt}\frac{\partial}{\partial\tau}.
\end{split}
\end{equation}
Using (\ref{5.12}) and (\ref{5.13}), we obtain
\begin{equation}\label{5.14}
\begin{split}
&(d\delta_{q(t),r(t)})^{-1}\left(
\frac{d}{dt}\delta_{q(t),r(t)}\right)\\
&=\frac{1}{r(t)}\frac{dr(t)}{dt}\sum_{j=1}^n(a_jX_j+b_jY_j)+\frac{2\tau}{r(t)}\frac{dr(t)}{dt}T
+\frac{1}{r(t)}\sum_{j=1}^n\left(\frac{da_j(t)}{dt}X_j+\frac{db_j(t)}{dt}Y_j\right)\\
&\hspace{4mm} +\frac{4}{r(t)}\sum_{j=1}^n
\left(a_j\frac{db_j(t)}{dt}-b_j\frac{da_j(t)}{dt}\right)T
+\frac{2}{r(t)^2}\sum_{j=1}^n\left(a_j(t)\frac{db_j(t)}{dt}-b_j(t)\frac{da_j(t)}{dt}\right)T
\\
&\hspace{4mm}+\frac{1}{r(t)^2}\frac{d\tau(t)}{dt}T.
\end{split}
\end{equation}
Since $d\phi=d\Psi\circ d\delta_{q(t),r(t)}\circ d\pi$ and
$\displaystyle\frac{d\phi}{dt}=d\Psi\circ
\frac{d}{dt}(\delta_{q(t),r(t)}\circ \pi)$, we have
\begin{equation}\label{5.15}
\begin{split}
\displaystyle\xi=(d\phi)^{-1}\frac{d\phi}{dt}&=(d\pi)^{-1}\circ
(d\delta_{q(t),r(t)})^{-1}\circ (d\Psi)^{-1}\left(d\Psi\circ
\frac{d}{dt}\delta_{q(t),r(t)}\circ \pi\right)\\
&=d\Psi\circ (d\delta_{q(t),r(t)})^{-1}\left(
\frac{d}{dt}\delta_{q(t),r(t)}\circ\pi \right).
\end{split}
\end{equation}
Since $\xi=(\xi_1,\cdots,\xi_{n+1})$, it follows from (\ref{5.9}), (\ref{5.10}), (\ref{5.14}) and (\ref{5.15}) that
\begin{equation}\label{5.16}
\begin{split}
\xi_{i}&=\frac{1}{r(t)}\frac{dr(t)}{dt}\sum_{j=1}^n\big(a_jd\Psi_i(X_j)+b_jd\Psi_i(Y_j)\big)+\frac{2\tau}{r(t)}\frac{dr(t)}{dt}d\Psi_i(T)\\
&\hspace{4mm}
+\frac{1}{r(t)}\sum_{j=1}^n\left(\frac{da_j(t)}{dt}d\Psi_i(X_j)+\frac{db_j(t)}{dt}d\Psi_i(Y_j)\right)\\
&\hspace{4mm}
+\frac{4}{r(t)}\sum_{j=1}^n
\left(a_j\frac{db_j(t)}{dt}-b_j\frac{da_j(t)}{dt}\right)d\Psi_i(T)\\
&\hspace{4mm}+\frac{1}{r(t)^2}\left[2\sum_{j=1}^n\left(a_j(t)\frac{db_j(t)}{dt}-b_j(t)\frac{da_j(t)}{dt}\right)+\frac{d\tau(t)}{dt}\right]d\Psi_i(T)\\
&=\frac{1}{r(t)}\frac{dr(t)}{dt}
\sum_{j=1}^n(a_j+\sqrt{-1}b_j)\left(\frac{2\delta_{ij}(1+|z|^2-\sqrt{-1}\tau)-4z_i\overline{z}_j}{(1+|z|^2-\sqrt{-1}\tau)^2}\right)\\
&\hspace{4mm}
+\frac{2\tau}{r(t)}\frac{dr(t)}{dt}\frac{2\sqrt{-1}z_i}{(1+|z|^2-\sqrt{-1}\tau)^2}\\
&\hspace{4mm}
+\frac{1}{r(t)}\sum_{j=1}^n\left(\frac{da_j(t)}{dt}+\sqrt{-1}\frac{db_j(t)}{dt}\right)
\left(\frac{2\delta_{ij}(1+|z|^2-\sqrt{-1}\tau)-4z_i\overline{z}_j}{(1+|z|^2-\sqrt{-1}\tau)^2}\right)\\
&\hspace{4mm} +\frac{4}{r(t)}\sum_{j=1}^n
\left(a_j\frac{db_j(t)}{dt}-b_j\frac{da_j(t)}{dt}\right)\frac{2\sqrt{-1}z_i}{(1+|z|^2-\sqrt{-1}\tau)^2}\\
&\hspace{4mm}+\frac{1}{r(t)^2}\left[2\sum_{j=1}^n\left(a_j(t)\frac{db_j(t)}{dt}-b_j(t)\frac{da_j(t)}{dt}\right)+\frac{d\tau(t)}{dt}\right]
\frac{2\sqrt{-1}z_i}{(1+|z|^2-\sqrt{-1}\tau)^2}\\
&=\frac{1}{r(t)}\frac{dr(t)}{dt}\frac{2z_i(1-|z|^2)}{(1+|z|^2-\sqrt{-1}\tau)^2}
+\frac{1}{r(t)}\frac{dr(t)}{dt}\frac{2\sqrt{-1}z_i\tau}{(1+|z|^2-\sqrt{-1}\tau)^2}\\
&\hspace{4mm}+\frac{1}{r(t)}\frac{dz_i(t)}{dt}\frac{2}{1+|z|^2-\sqrt{-1}\tau}
-\frac{1}{r(t)}\sum_{j=1}^n\frac{dz_j(t)}{dt}\frac{4z_i\overline{z}_j}{(1+|z|^2-\sqrt{-1}\tau)^2}\\
&\hspace{4mm}-\frac{1}{r(t)}\frac{8z_i\sqrt{-1}}{(1+|z|^2-\sqrt{-1}\tau)^2}Im\Big(\frac{d\overline{z(t)}}{dt}\cdot
z\Big)\\
&\hspace{4mm}+\frac{1}{r(t)^2}\left[2Im\Big(\frac{dz(t)}{dt}\cdot
\overline{z(t)}\Big)+\frac{d\tau(t)}{dt}\right]\frac{2\sqrt{-1}z_i}{(1+|z|^2-\sqrt{-1}\tau)^2}\\
&=\frac{1}{r(t)}\frac{dr(t)}{dt}\Psi_i
\Psi_{n+1}+\frac{1}{r(t)}\frac{dz_i(t)}{dt}(1+\Psi_{n+1})
-\frac{1}{r(t)}
\sum_{j=1}^n\frac{dz_j(t)}{dt}\frac{(1+\Psi_{n+1})\Psi_i\overline{\Psi}_j}{1+\overline{\Psi}_{n+1}}\\
&\hspace{4mm}-\frac{2}{r(t)}
(1+\Psi_{n+1})\Psi_i\sqrt{-1}Im\left(\frac{d\overline{z(t)}}{dt}\cdot
\Big(\frac{\Psi_1}{1+\Psi_{n+1}},\cdots,\frac{\Psi_n}{1+\Psi_{n+1}}\Big)\right)\\
&\hspace{4mm}
+\frac{1}{r(t)^2}\left[2Im\Big(\frac{dz(t)}{dt}\cdot
\overline{z(t)}\Big)+\frac{d\tau(t)}{dt}\right]\frac{\sqrt{-1}}{2}\Psi_i(1+\Psi_{n+1})
\end{split}
\end{equation}
for $1\leq i\leq n$, and
\begin{equation}\label{5.17}
\begin{split}
\xi_{n+1}&=\frac{1}{r(t)}\frac{dr(t)}{dt}\sum_{j=1}^n\big(a_jd\Psi_{n+1}(X_j)+b_jd\Psi_{n+1}(Y_j)\big)
+\frac{2\tau}{r(t)}\frac{dr(t)}{dt}d\Psi_{n+1}(T)\\
&\hspace{4mm}+\frac{1}{r(t)}\sum_{j=1}^n\left(\frac{da_j(t)}{dt}d\Psi_{n+1}(X_j)
+\frac{db_j(t)}{dt}d\Psi_{n+1}(Y_j)\right)\\
&\hspace{4mm}+\frac{4}{r(t)}\sum_{j=1}^n
\left(a_j\frac{db_j(t)}{dt}-b_j\frac{da_j(t)}{dt}\right)d\Psi_{n+1}(T)\\
&\hspace{4mm}
+\frac{1}{r(t)^2}\left[2\sum_{j=1}^n\left(a_j(t)\frac{db_j(t)}{dt}-b_j(t)\frac{da_j(t)}{dt}\right)+\frac{d\tau(t)}{dt}\right]d\Psi_{n+1}(T)\\
&=-\frac{1}{r(t)}\frac{dr(t)}{dt}\sum_{j=1}^n\frac{4(a_j+\sqrt{-1}b_j)\overline{z}_j}{(1+|z|^2-\sqrt{-1}\tau)^2}
+\frac{2\tau}{r(t)}\frac{dr(t)}{dt}\frac{2\sqrt{-1}}{(1+|z|^2-\sqrt{-1}\tau)^2}\\
&\hspace{4mm}
-\frac{1}{r(t)}\sum_{j=1}^n\left(\frac{da_j(t)}{dt}+\sqrt{-1}\frac{db_j(t)}{dt}\right)
\frac{4\overline{z}_j}{(1+|z|^2-\sqrt{-1}\tau)^2}\\
&\hspace{4mm} +\frac{4}{r(t)}\sum_{j=1}^n
\left(a_j\frac{db_j(t)}{dt}-b_j\frac{da_j(t)}{dt}\right)\frac{2\sqrt{-1}}{(1+|z|^2-\sqrt{-1}\tau)^2}\\
&\hspace{4mm}+\frac{1}{r(t)^2}\left[2\sum_{j=1}^n\left(a_j(t)\frac{db_j(t)}{dt}-b_j(t)\frac{da_j(t)}{dt}\right)+\frac{d\tau(t)}{dt}\right]\frac{2\sqrt{-1}}{(1+|z|^2-\sqrt{-1}\tau)^2}\\
&=-\frac{1}{r(t)}\frac{dr(t)}{dt}\frac{4|z|^2}{(1+|z|^2-\sqrt{-1}\tau)^2}
+\frac{1}{r(t)}\frac{dr(t)}{dt}\frac{4\sqrt{-1}\tau}{(1+|z|^2-\sqrt{-1}\tau)^2}\\
&\hspace{4mm}
-\frac{1}{r(t)}\sum_{j=1}^n\frac{dz_j(t)}{dt}
\frac{4\overline{z}_j}{(1+|z|^2-\sqrt{-1}\tau)^2} -\frac{4}{r(t)}\frac{2\sqrt{-1}}{(1+|z|^2-\sqrt{-1}\tau)^2}Im\Big(\frac{d\overline{z(t)}}{dt}\cdot z\Big)\\
&\hspace{4mm}
+\frac{1}{r(t)^2}\left[2Im\Big(\frac{dz(t)}{dt}\cdot
\overline{z(t)}\Big)+\frac{d\tau(t)}{dt}\right]\frac{2\sqrt{-1}}{(1+|z|^2-\sqrt{-1}\tau)^2}\\
&=\frac{1}{r(t)}\frac{dr(t)}{dt}(\Psi_{n+1}^2-1)-\frac{1}{r(t)}\sum_{j=1}^n\frac{dz_j(t)}{dt}
\frac{\overline{\Psi}_j(1+\Psi_{n+1})^2}{1+\overline{\Psi}_{n+1}}\\
&\hspace{4mm}+\frac{2}{r(t)}(1+\Psi_{n+1})^2\sqrt{-1}Im
\left(\frac{d\overline{z(t)}}{dt}\cdot\Big(\frac{\Psi_1}{1+\Psi_{n+1}},\cdots,\frac{\Psi_n}{1+\Psi_{n+1}}\Big)\right)\\
&\hspace{4mm}
+\frac{1}{r(t)^2}\left[2Im\Big(\frac{dz(t)}{dt}\cdot
\overline{z(t)}\Big)+\frac{d\tau(t)}{dt}\right]\frac{\sqrt{-1}}{2}(1+\Psi_{n+1})^2.
\end{split}
\end{equation}
Thus, in our calculation, we may assume at time $t$, $q(t)=0$ which simplify the calculation since otherwise it is
the matter of the choice of the coordinates of $S^{2n+1}$. In doing so, we let
\begin{equation}\label{5.2}
X=(X_1,\cdots,X_{n+1})=\int_{S^{2n+1}}\xi\,dV_{\theta_0},
\end{equation}
and denote $r(t)^{-1}$ by $\epsilon$. Then by symmetry, we obtain from (\ref{5.16}) and (\ref{5.17}) that
\begin{equation}\label{5.18}
\begin{split}
X_i&=\epsilon\int_{S^{2n+1}}\left[\vphantom{\left(\frac{d\overline{z_i(t)}}{dt}\right)}
\frac{dz_i(t)}{dt}-\frac{dz_i(t)}{dt}\frac{1+\Psi_{n+1}}{1+\overline{\Psi}_{n+1}}|\Psi_i|^2\right.\\
&\hspace{8mm}\left.
-2(1+\Psi_{n+1})\Psi_i\sqrt{-1}Im\left(\frac{d\overline{z_i(t)}}{dt}\frac{\Psi_i}{1+\Psi_{n+1}}\right)\right]dV_{\theta_0}\\
&=\epsilon\int_{S^{2n+1}}\left[\vphantom{\left(\frac{d\overline{z_i(t)}}{dt}\right)}
\frac{dz_i(t)}{dt}-\frac{dz_i(t)}{dt}\frac{1+\Psi_{n+1}}{1+\overline{\Psi}_{n+1}}|\Psi_i|^2\right.\\
&\hspace{8mm}\left.
-(1+\Psi_{n+1})\Psi_i\left(\frac{d\overline{z_i(t)}}{dt}\frac{\Psi_i}{1+\Psi_{n+1}}-
\frac{dz_i(t)}{dt}\frac{\overline{\Psi}_i}{1+\overline{\Psi}_{n+1}}\right)\right]dV_{\theta_0}\\
&=\epsilon\frac{dz_i(t)}{dt}\int_{S^{2n+1}}dV_{\theta_0}
-\epsilon\frac{d\overline{z}_i(t)}{dt}\int_{S^{2n+1}}\Psi_i^2dV_{\theta_0}\\
&=\epsilon\frac{dz_i(t)}{dt}\int_{S^{2n+1}}dV_{\theta_0}
-\epsilon\frac{d\overline{z}_i(t)}{dt}\int_{S^{2n+1}}\Big(Re(\Psi_i)^2-Im(\Psi_i)^2+2\sqrt{-1}Re(\Psi_i)Im(\Psi_i)\Big)dV_{\theta_0}\\
&=\epsilon\mbox{Vol}(S^{2n+1},\theta_0)\frac{dz_i(t)}{dt}
\end{split}
\end{equation}
for $1\leq i\leq n$, and
\begin{equation}\label{5.19}
\begin{split}
X_{n+1}&=\epsilon\int_{S^{2n+1}}\frac{dr(t)}{dt}\frac{\Psi_{n+1}^2-1}{2}\,dV_{\theta_0}\\
&\hspace{4mm}
+\epsilon^2\int_{S^{2n+1}}\left[2Im\Big(\frac{dz(t)}{dt}\cdot
\overline{z(t)}\Big)+\frac{d\tau(t)}{dt}\right]\frac{\sqrt{-1}}{2}(1+\Psi_{n+1})^2dV_{\theta_0}\\
&=\frac{\epsilon}{2}\frac{dr(t)}{dt}\int_{S^{2n+1}}(Re(\Psi_{n+1})^2-Im(\Psi_{n+1})^2-1)dV_{\theta_0}\\
&\hspace{4mm}
+\frac{\epsilon^2\sqrt{-1}}{2}\left[2Im\Big(\frac{dz(t)}{dt}\cdot
\overline{z(t)}\Big)+\frac{d\tau(t)}{dt}\right]\int_{S^{2n+1}}(Re(\Psi_{n+1})^2-Im(\Psi_{n+1})^2+1)dV_{\theta_0}\\
&=\frac{\epsilon}{2}\mbox{Vol}(S^{2n+1},\theta_0)\left(-\frac{dr(t)}{dt}
+2\sqrt{-1}\epsilon Im\Big(\frac{dz(t)}{dt}\cdot
\overline{z(t)}\Big)+\sqrt{-1}\epsilon\frac{d\tau(t)}{dt}\right).
\end{split}
\end{equation}

Now we are going to get the estimate on the conformal vector field $\xi$.

\begin{lem}\label{lem5.1}
There exists a constant $C>0$ such that
$$\|\xi\|_{L^\infty}^2\leq C\int_{S^{2n+1}}(\alpha(t) f_\phi-R_h)^2dV_h.$$
\end{lem}
\begin{proof}
Note that $\|\Psi_i\|_{L^\infty}\leq 3$ for $i=1,...,n+1$. Thus by (\ref{5.16}) and (\ref{5.17}),
we have
\begin{equation}\label{5.20}
\|\xi\|_{L^\infty}\leq C\left(\left|\epsilon\frac{dr(t)}{dt}\right|
+\sum_{i=1}^{n}\left|\epsilon\frac{dz_i(t)}{dt}\right|
+\left|2\epsilon^2 Im\Big(\frac{dz(t)}{dt}\cdot
\overline{z(t)}\Big)+\epsilon^2\frac{d\tau(t)}{dt}\right|\right)
\end{equation}
for some constant $C$ independent of $t$. By (\ref{5.18}) and
(\ref{5.19}), we have
\begin{equation}\label{5.21}
\begin{split}
&\epsilon\frac{dz_i(t)}{dt}
=\frac{1}{\mbox{Vol}(S^{2n+1},\theta_0)}X_i,
\epsilon\frac{dr(t)}{dt}=-\frac{1}{\mbox{Vol}(S^{2n+1},\theta_0)}(X_{n+1}+\overline{X}_{n+1}),\\
&\sqrt{-1}\epsilon^2\left[2Im\Big(\frac{dz(t)}{dt}\cdot
\overline{z(t)}\Big)+\frac{d\tau(t)}{dt}\right]=-\frac{1}{\mbox{Vol}(S^{2n+1},\theta_0)}(X_{n+1}-\overline{X}_{n+1}).
\end{split}
\end{equation}
Hence, it follows from (\ref{5.20}), (\ref{5.21}) and Cauchy-Schwartz inequality that
\begin{equation}\label{5.22}
\|\xi\|_{L^\infty}\leq C_0\|X\|
\end{equation}
for some constant $C_0$ independent of $t$. Here $\|X\|$ is the norm
of the vector $X\in\mathbb{C}^{n+1}$, i.e. $\|X\|^2=
\displaystyle\sum_{i=1}^{n+1}|X_i|^2$. Combining
(\ref{5.b}) and (\ref{5.22}), we get the following estimate:
\begin{equation}\label{5.23}
\begin{split}
\|\xi\|_{L^\infty}&\leq C_0\|X\|\\
&\leq C_0\left(
(n+1)\left|\int_{S^{2n+1}}x(\alpha(t)f_\phi-R_h)dV_h\right|
+\left|\int_{S^{2n+1}}\xi(1-v^{2+\frac{2}{n}})dV_{\theta_0}\right|\right)\\
&\leq C_0\left(
(n+1)\left|\int_{S^{2n+1}}x(\alpha(t)f_\phi-R_h)dV_h\right|
+\mbox{Vol}(S^{2n+1},\theta_0)\|\xi\|_{L^\infty}\|v^{2+\frac{2}{n}}-1\|_{C^0}\right).
\end{split}
\end{equation}
Then by Lemma \ref{lem4.14},
$\|v^{2+\frac{2}{n}}-1\|_{C^0}\rightarrow 0$ as
$t\rightarrow\infty$. Hence there exists a $T>0$ such that
$C_0\mbox{Vol}(S^{2n+1},\theta_0)\|v^{2+\frac{2}{n}}-1\|_{C^0}\leq
1/2$ if $t\geq T$. Hence, by (\ref{5.23}), for all $t\geq T$ we have
\begin{equation*}
\begin{split}
\|\xi\|_{L^\infty}&\leq
2C_0(n+1)\left|\int_{S^{2n+1}}x(\alpha(t)f_\phi-R_h)dV_h\right|\\
&\leq
2C_0(n+1)\mbox{Vol}(S^{2n+1},\theta_0)^{\frac{1}{2}}
\left(\int_{S^{2n+1}}(\alpha(t)f_\phi-R_h)^2dV_h\right)^{\frac{1}{2}}
\end{split}
\end{equation*}
by H\"{o}lder's inequality. On the other hand, $\xi$ is continuous
on $S^{2n+1}\times [0,T]$. Setting $C_1=\max_{(x,t)\in
S^{2n+1}\times [0,T]}\|\xi\|^2_{L^\infty}$, we conclude that
$$\|\xi\|^2_{L^\infty}\leq \frac{C_1}{\min_{t\in[0,T]}F_2(t)} F_2(t)$$
for all $t\leq T$. Here we observe that $F_2(t)$ can never be zero
for any finite $t$, otherwise $f$ could be realized as the Webster
scalar curvature of some conformal contact form. Hence this lemma
follows from these two estimates.
\end{proof}

\section{Spectral decomposition}\label{section6}

For convenience, we denote
$$F_p(t)=\int_{S^{2n+1}}|R_\theta-\alpha f|^pdV_\theta\hspace{2mm}\mbox{ and }\hspace{2mm}G_p(t)
=\int_{S^{2n+1}}|\nabla_\theta(R_\theta-\alpha f)|_\theta^pdV_\theta$$
for $p\geq 1$.
The following lemma was proved in part I. See Lemma 3.2 and 3.3 in \cite{Ho3}.

\begin{lem}\label{lem3.2}
For any $p<\infty$, there holds $F_p(t)\rightarrow 0$ as
$t\rightarrow\infty$. There also holds $G_2(t)\rightarrow 0$ as $t\rightarrow\infty$.
\end{lem}

The following lemma was also proved in part I. See Lemma 5.1 in \cite{Ho3}.

\begin{lem}\label{lem6.1}
With error $o(1)\rightarrow 0$ as $t\rightarrow\infty$, there holds
$$\frac{d}{dt}F_2(t)\leq (n+1+o(1))(nF_2(t)-2G_2(t))+o(1)F_2(t).$$
\end{lem}

\subsection{The shadow flow}\label{section6.2}

From now on, we assume that $n\geq 2$, as in the assumption of Theorem \ref{thm1.1}.
Recall Theorem \ref{thm4.7}, the center of mass $\Theta(t)$ of the contact
form $\theta(t)$ is given approximately by
$$\Theta(t)=\int_{S^{2n+1}}\phi(t)dV_{\theta_0}$$
with
$\widehat{\Theta(t)}=\displaystyle\frac{\Theta(t)}{\|\Theta(t)\|}$.
For any given $t\geq 0$, rotate $\widehat{\Theta(t)}$ as the south
pole, then the conformal CR diffeomorphism may be represented as
$\phi(t)=\Psi\circ \delta_{q(t),r(t)}\circ\pi$ for some $q(t)\in\mathbb{H}^n$ and $r(t)>0$. In the following
lemma, we extend $f(\mu y)=f(y)$ for $0<\mu< 1$,
$y\in\mathbb{S}^{2n+1}$.

\begin{lem}\label{lem6.2}
With a uniform constant $C>0$, if one set $\epsilon=1/r(t)$, then
there holds
$$\|f_\phi-f(\widehat{\Theta(t)})\|_{L^2(S^{2n+1},\theta_0)}
+\|\nabla_{\theta_0}f_{\phi}\|_{L^{2}(S^{2n+1},\theta_0)}
\leq C\epsilon.$$
\end{lem}
\begin{proof}
We choose the coordinate at the point $\widehat{\Theta(t)}$ which
can be represented as the north pole so that $S^{2n+1}$ can be
represented by $\Psi$, where $\Psi(z,\tau)=\pi^{-1}(z,\tau)$,
$(z,\tau)\in\mathbb{H}^n$ defined in (\ref{5.6}). For simplicity, we set
$\epsilon(t)=\displaystyle\frac{1}{r(t)}$. Hence, by a calculation similar to (\ref{5.7}) and (\ref{5.8})
we have
\begin{equation*}
\begin{split}
\int_{S^{2n+1}}|\nabla_{\theta_0}\phi|_{\theta_0}^2dV_{\theta_0}
&=\int_{\mathbb{H}^n}|\nabla_{\Psi^*(\theta_0)}(\phi\circ\Psi)|_{\Psi^*(\theta_0)}^2
\left(\frac{4}{\tau^2+(1+|z|^2)^2}\right)^{n+1}dzd\tau\\
&=\int_{\mathbb{H}^n}\left(\frac{4n\epsilon^2}{(1+\epsilon^2|z|^2)^2+\epsilon^4\tau^2}\right)
\left(\frac{4}{\tau^2+(1+|z|^2)^2}\right)^{n}dzd\tau\\
&\leq
C\epsilon^2\int_{B_{\epsilon^{-1}}(0)}\frac{dzd\tau}{(\tau^2+(1+|z|^2)^2)^{n}}\\
&\hspace{2mm}+C\epsilon^{-2}\int_{\mathbb{H}^n\setminus B_{\epsilon^{-1}}(0)}\frac{dzd\tau}{(\tau^2+(1+|z|^2)^2)^{n+1}}\\
&\leq C\epsilon^2+C\epsilon^{2n-2}\leq C\epsilon,
\end{split}
\end{equation*}
where we have used the estimates:
For $\displaystyle0\leq m\leq\frac{n}{2}$, we have
\begin{equation}\label{6.56}
\begin{split}
&\int_{B_{\epsilon^{-1}}(0)}\frac{(\tau^2+|z|^4)^{m}dzd\tau}{(\tau^2+(1+|z|^2)^2)^{n+1}}\\
&\leq
\int_{B_{\epsilon^{-1}}(0)}\frac{dzd\tau}{(\tau^2+(1+|z|^2)^2)^{n+1-m}}
=\int_{\{\sqrt[4]{\tau^2+|z|^4}\leq\epsilon^{-1}\}}\frac{dzd\tau}{(\tau^2+(1+|z|^2)^2)^{n+1-m}}\\
&\leq
\int_{\{|z|\leq\epsilon^{-1}\}}\left(\int_{-\sqrt{\epsilon^{-4}-|z|^4}}^{\sqrt{\epsilon^{-4}-|z|^4}}\frac{d\tau}{1+\tau^2}\right)
\frac{dz}{(1+|z|^2)^{2n-2m}}\\
&=\int_{\{|z|\leq\epsilon^{-1}\}}\Big[\tan^{-1}(\tau)\Big]_{-\sqrt{\epsilon^{-4}-|z|^4}}^{\sqrt{\epsilon^{-4}-|z|^4}}
\frac{dz}{(1+|z|^2)^{2n-2m}}\\
&\leq \pi\int_{\{|z|\leq\epsilon^{-1}\}}\frac{dz}{(1+|z|^2)^{2n-2m}}
=C\int_0^{\epsilon^{-1}}\frac{r^{2n-1}dr}{(1+r^2)^{2n-2m}}\\
&=C\left(\int_0^{1}\frac{r^{2n-1}dr}{(1+r^2)^{2n-2m}}+\int_1^{\epsilon^{-1}}\frac{r^{2n-1}dr}{(1+r^2)^{2n-2m}}\right)\\
&\leq C\left(\int_0^{1}\frac{dr}{(1+r^2)^{1-2m}}+\int_1^{\epsilon^{-1}}\frac{dr}{r^{2n-4m+1}}\right)=\left\{
    \begin{array}{ll}
      C+C\epsilon^{2n-4m}, & \hbox{if $m<\frac{n}{2}$;} \\
     C+C\log \epsilon, & \hbox{if $m=\frac{n}{2}$,}
    \end{array}
  \right.
\end{split}
\end{equation}
and
\begin{equation}\label{6.3}
\begin{split}
&\int_{\mathbb{H}^n\setminus
B_{\epsilon^{-1}}(0)}\frac{dzd\tau}{(\tau^2+(1+|z|^2)^2)^{k}}
=\int_{\{\sqrt[4]{\tau^2+|z|^4}\geq\epsilon^{-1}\}}\frac{dzd\tau}{(\tau^2+(1+|z|^2)^2)^{k}}\\
&\leq
2\int_{\{|z|\geq\epsilon^{-1}\}}\left(\int_{\sqrt{\epsilon^{-4}-|z|^4}}^\infty\frac{d\tau}{1+\tau^2}\right)
\frac{dz}{(1+|z|^2)^{2k-2}}\\
&\leq\pi\int_{\{|z|\geq\epsilon^{-1}\}}\frac{dz}{(1+|z|^2)^{2k-2}}
=C\int_{\epsilon^{-1}}^\infty\frac{r^{2n-1}dr}{(1+r^2)^{2k-2}}\\
&\leq
C\int_{\epsilon^{-1}}^\infty\frac{dr}{r^{4k-2n-3}}=O(\epsilon^{4k-2n-4})\hspace{2mm}\mbox{ if }4k\geq 2n+5.
\end{split}
\end{equation}
Recall that $\Theta(t)$ is the average of $\phi(t)$, from the Poincar\'{e}-type inequality
(see Theorem 3.20 in \cite{Dragomir}),
we have
$$\|\phi(t)-\Theta(t)\|_{L^2(S^{2n+1},\theta_0)}\leq C\|\nabla_{\theta_0}\phi\|_{L^{2}(S^{2n+1},\theta_0)}\leq C\epsilon.$$
Here we need the assumption $n\geq 2$ to conclude that $p=2<n+1$ so that Theorem 3.20 in \cite{Dragomir} can be applied.
Hence, by the inequalities
$$
|f_\phi-f(\widehat{\Theta(t)})|=|f_\phi-f(\Theta(t))|\leq\|\nabla f\|_{L^\infty}|\phi(t)-\Theta(t)|$$
 and $$
|\nabla_{\theta_0} f_\phi|\leq\|\nabla f\|_{L^\infty}|\nabla_{\theta_0}\phi|,
$$
the assertion follows.
\end{proof}

Let $\{\varphi_i\}$ be an $L^2(S^{2n+1},\theta_0)$-orthonormal basis
of eigenfunctions of $-\Delta_{\theta_0}$, satisfying
$-\Delta_{\theta_0}\varphi_i=\lambda_i\varphi_i$ with eigenvalues
$0=\lambda_0<\lambda_1=\cdots=\lambda_{2n+2}=\displaystyle\frac{n}{2}<\lambda_{2n+3}\leq\cdots$. In fact,
we can take with loss of generality
\begin{equation}\label{6.a}
\varphi_i=\frac{1}{\sqrt{n+1}}x_i\hspace{2mm}\mbox{ and }\hspace{2mm}\varphi_{n+1+i}=\frac{1}{\sqrt{n+1}}\overline{x}_i\hspace{2mm}\mbox{ for }i=1,...,n+1,
\end{equation}
where $x=(x_1,\cdots,x_{n+1})$ is the coordinates of $\mathbb{C}^{n+1}$ restricted to $S^{2n+1}$. Now in terms of the orthonormal basis $\{\varphi_i^\theta\}$,
$\{\varphi_i^h\}$ of the eigenfunctions of $-\Delta_\theta$,
$-\Delta_h$ with the corresponding eigenvalues $\lambda_i^\theta$,
$\lambda_i^h$ respectively, we expand
\begin{equation*}
\alpha
f-R_\theta=\sum_{i=0}^\infty\beta_\theta^i\varphi_i^\theta\hspace{2mm}\mbox{
and }\hspace{2mm}\alpha
f_\phi-R_h=\sum_{i=0}^\infty\beta_h^i\varphi_i^h,
\end{equation*}
with coefficients
\begin{equation}\label{6.21}
\beta^i_h=\int_{S^{2n+1}}(\alpha f_\phi-R_h)\varphi_i^h
\,dV_h=\int_{S^{2n+1}}(\alpha f-R_\theta)\varphi_i^\theta
\,dV_\theta=\beta_\theta^i
\end{equation}
for all $i\in\mathbb{N}$. First notice that we always have
$\beta_0^\theta=0$ in view of (\ref{2.2}). It is well known that
$\varphi_i^h=\varphi_i^\theta\circ\phi$, which implies (\ref{6.21})
and $\lambda_i^\theta=\lambda_i^h$ for all $i\in\mathbb{N}$.

\begin{lem}\label{lem6.3}
 As $t\rightarrow\infty$, we have
$\lambda_i^\theta=\lambda_i^h\rightarrow\lambda_i$ and we can choose
$\varphi_i$ such that $\varphi_i^h\rightarrow\varphi_i$ in
$L^2(S^{2n+1},\theta_0)$ for all $i\in\mathbb{N}$.
\end{lem}

Since the proof is essentially the same as the proof of Lemma 5.2 in part I, we
omit the proof and refer the reader to \cite{Ho3}. Now we define
\begin{equation}\label{6.1}
b=(b^1,\cdots,b^{2n+2})=\int_{S^{2n+1}}(x,\overline{x})(\overline{R}_h-R_h)dV_h
\end{equation}
where $x=(x_1,...,x_{n+1})\in S^{2n+1}\subset\mathbb{C}^{n+1}$ and
$\overline{x}=(\overline{x}_1,...,\overline{x}_{n+1})$.
That is,
$$
b^i=\int_{S^{2n+1}}x_i(\alpha f_\phi-R_h)dV_h\mbox{ and }b^{n+1+i}=\int_{S^{2n+1}}\overline{x}_i(\alpha f_\phi-R_h)dV_h\mbox{ for }1\leq i\leq n+1.
$$
For brevity, set $B=\sqrt{n+1}\,b$,
$\beta_\theta=(\beta_\theta^1,\cdots,\beta_\theta^{2n+2})$, then by
(\ref{6.a}), (\ref{6.21}) and Lemma \ref{lem6.3}
\begin{equation}\label{6.26}
\begin{split}
|B^i-\beta^i_\theta|&=|\sqrt{n+1}\,b^i-\beta^i_\theta|\\
&=\left|\sqrt{n+1}\int_{S^{2n+1}}x_i(\alpha
f_\phi-R_h)dV_h-\int_{S^{2n+1}}\varphi_i^h(\alpha f_\phi-R_h)
dV_h\right|\\
&=\left|\int_{S^{2n+1}}(\varphi_i-\varphi_i^h)(\alpha
f_\phi-R_h)dV_h\right|\\
&\leq\|\varphi_i-\varphi_i^h\|_{L^2(S^{2n+1},h)}\|\alpha
f_\phi-R_h\|_{L^2(S^{2n+1},h)}\\
&\leq
C\|\varphi_i-\varphi_i^h\|_{L^2(S^{2n+1},\theta_0)}F_2(h(t))^{\frac{1}{2}}=o(1)F_2(t)^{\frac{1}{2}}
\end{split}
\end{equation}
for $i=1,2,...,2n+2$, where $o(1)\rightarrow 0$ as
$t\rightarrow\infty$.

\begin{lem}\label{lem6.4}
With error $o(1)\rightarrow 0$ as $t\rightarrow\infty$, there holds
$$\frac{dB(t)}{dt}=o(1)F_2(t)^{\frac{1}{2}}.$$
\end{lem}
\begin{proof}
By (\ref{4.22}), we have
$$b=\int_{S^{2n+1}}(x,\overline{x})\alpha f_\phi dV_h
-\int_{S^{2n+1}}(x,\overline{x})v\left(-(2+\frac{2}{n})\Delta_{\theta_0}v+R_{\theta_0}v\right)dV_{\theta_0}.$$
Thus
\begin{equation}\label{6.27}
\begin{split}
\frac{db}{dt}&=\alpha_t\int_{S^{2n+1}}(x,\overline{x})f_\phi
dV_h+\int_{S^{2n+1}}(x,\overline{x})\alpha
df_\phi\cdot\xi\,dV_h\\
&\hspace{4mm} +(2+\frac{2}{n})\int_{S^{2n+1}}(x,\overline{x})\alpha
f_\phi v^{1+\frac{2}{n}}v_tdV_{\theta_0}\\
&\hspace{4mm}-\int_{S^{2n+1}}(x,\overline{x})v_t\left(-(2+\frac{2}{n})
\Delta_{\theta_0}v+R_{\theta_0}v\right)dV_{\theta_0}\\
&\hspace{4mm}-\int_{S^{2n+1}}(x,\overline{x})v\left(-(2+\frac{2}{n})
\Delta_{\theta_0}v_t+R_{\theta_0}v_t\right)dV_{\theta_0}\\
&=\alpha_t\int_{S^{2n+1}}(x,\overline{x})f_\phi
dV_h+\int_{S^{2n+1}}(x,\overline{x})\alpha
df_\phi\cdot\xi dV_h\\
&\hspace{4mm} +(2+\frac{2}{n})\int_{S^{2n+1}}(x,\overline{x})\alpha
f_\phi v^{1+\frac{2}{n}}v_tdV_{\theta_0}-2R_{\theta_0}\int_{S^{2n+1}}(x,\overline{x})vv_tdV_{\theta_0}\\
&\hspace{4mm}+(2+\frac{2}{n})\int_{S^{2n+1}}(x,\overline{x})
(v_t\Delta_{\theta_0}v+v\Delta_{\theta_0}v_t)dV_{\theta_0}.
\end{split}
\end{equation}
We are going to estimate the terms on the right hand side of (\ref{6.27}).
By (\ref{4.19}) and Lemma \ref{lem4.14}, and by (3.4)  in \cite{Ho3}, the first term on the right hand side of (\ref{6.27}) can be bounded by
\begin{equation*}
\alpha_t\int_{S^{2n+1}}(x,\overline{x})f_\phi
dV_h=\alpha_t\int_{S^{2n+1}}(x,\overline{x})(f_\phi-f(\widehat{P(t)}))
dV_h=o(1)F_2(t)^{\frac{1}{2}}.
\end{equation*}
Observe that by (\ref{6.a}) and integration by parts, the last four terms on the right hand side of (\ref{6.27}) can be rewritten as
\begin{equation}\label{6.28}
\begin{split}
&\int_{S^{2n+1}}(x,\overline{x})\alpha
df_\phi\cdot\xi\,dV_h+(2+\frac{2}{n})\int_{S^{2n+1}}(x,\overline{x})\alpha
f_\phi v^{1+\frac{2}{n}}v_tdV_{\theta_0}\\
&-2R_{\theta_0}\int_{S^{2n+1}}(x,\overline{x})vv_tdV_{\theta_0}
+(2+\frac{2}{n})\int_{S^{2n+1}}(x,\overline{x})(v_t\Delta_{\theta_0}v+v\Delta_{\theta_0}v_t)dV_{\theta_0}\\
&=\int_{S^{2n+1}}(x,\overline{x})\alpha
df_\phi\cdot\xi\,dV_h+2(2+\frac{2}{n})\int_{S^{2n+1}}
v_t\langle\nabla_{\theta_0}(x,\overline{x}),\nabla_{\theta_0}v\rangle_{\theta_0}dV_{\theta_0}
\\
&\hspace{4mm}+(2+\frac{2}{n})\int_{S^{2n+1}}(x,\overline{x})v_t\left[\alpha
f_\phi
v^{1+\frac{2}{n}}-\left(\frac{nR_{\theta_0}}{n+1}+\frac{n}{2}\right)v
+2\Delta_{\theta_0}v\right]dV_{\theta_0}\\
&=\left[\int_{S^{2n+1}}(x,\overline{x})\alpha
df_\phi\cdot\xi\,dV_h+(2+\frac{2}{n})\int_{S^{2n+1}}(x,\overline{x})v_t\big(\alpha
f_\phi v^{1+\frac{2}{n}}-R_{\theta_0}v\big)dV_{\theta_0}\right]
\\
&\hspace{4mm}+2(2+\frac{2}{n})\int_{S^{2n+1}}(x,\overline{x})v_t\Delta_{\theta_0}v\,dV_{\theta_0}
+2(2+\frac{2}{n})\int_{S^{2n+1}}
v_t\langle\nabla_{\theta_0}(x,\overline{x}),\nabla_{\theta_0}v\rangle_{\theta_0}dV_{\theta_0}\\
&=I_1+I_2+I_3.
\end{split}
\end{equation}
By (\ref{5.a}), Lemma \ref{lem4.14}, Lemma \ref{lem5.1} and Lemma \ref{lem3.2},
by integration by parts and H\"{o}lder's inequality, we obtain
\begin{equation*}
\begin{split}
I_1&=\int_{S^{2n+1}}(x,\overline{x}) \left[(n+1)(\alpha
f_\phi-R_h)+v^{-(2+\frac{2}{n})}\mbox{div}'_{\theta_0}(v^{2+\frac{2}{n}}\xi)\right](\alpha f_\phi
v^{2+\frac{2}{n}}-R_{\theta_0}v^2)dV_{\theta_0}\\
&\hspace{4mm}+\int_{S^{2n+1}}(x,\overline{x})\alpha
df_\phi\cdot\xi\,dV_h\\
&=(n+1)\int_{S^{2n+1}}(x,\overline{x})(\alpha f_\phi-R_h)(\alpha
f_\phi v^{2+\frac{2}{n}}-R_{\theta_0}v^2)dV_{\theta_0}\\
&\hspace{4mm}-\int_{S^{2n+1}}(\xi,\vec{0})(\alpha f_\phi
v^{2+\frac{2}{n}}-R_{\theta_0}v^2)dV_{\theta_0}-\frac{2}{n}R_{\theta_0}\int_{S^{2n+1}}(x,\overline{x}) v(dv\cdot\xi)dV_{\theta_0}\\
&\leq C\|\alpha f_\phi
v^{2+\frac{2}{n}}-R_{\theta_0}v^2\|_{L^2(S^{2n+1},\theta_0)}\big(\|\alpha
f_\phi-R_h\|_{L^2(S^{2n+1},h)}+\|\xi\|_{L^\infty}\big)\\
&\hspace{4mm}+C\|\nabla_{\theta_0}v\|_{L^2(S^{2n+1},\theta_0)}\|\xi\|_{L^\infty}
=o(1)F_2(t)^{\frac{1}{2}},
\end{split}
\end{equation*}
where $\mbox{div}'_{\theta_0}$ is the subdivergence operator of type $(1,0)$ with respect to the contact form $\theta_0$ (see \cite{Cheng} for the definition).
By (\ref{5.a}), Lemma \ref{lem4.14} and Lemma \ref{lem5.1}, we get
\begin{equation*}
\begin{split}
I_2&=(2n+2)\int_{S^{2n+1}}(x,\overline{x})(\alpha
f_\phi-R_h)v\Delta_{\theta_0}v\,dV_{\theta_0}\\
&\hspace{4mm}+2\int_{S^{2n+1}}(x,\overline{x})v^{-(1+\frac{2}{n})}\mbox{div}'_{\theta_0}(v^{2+\frac{2}{n}}\xi)
\Delta_{\theta_0}v\,dV_{\theta_0}\\
&=(2n+2)\int_{S^{2n+1}}(x,\overline{x})(\alpha f_\phi-R_h)v\Delta_{\theta_0}v\,
dV_{\theta_0}-2\int_{S^{2n+1}}(\xi,\vec{0})\,v\Delta_{\theta_0}v\,dV_{\theta_0}\\
&\hspace{4mm}+2(1+\frac{2}{n})\int_{S^{2n+1}}(x,\overline{x})(dv\cdot\xi)\Delta_{\theta_0}v\,dV_{\theta_0}
-2\int_{S^{2n+1}}(x,\overline{x})v\big(d(\Delta_{\theta_0}v)\cdot\xi\big)dV_{\theta_0}\\
&\leq C(\|\Delta_{\theta_0}v\|_{L^2(S^{2n+1},\theta_0)}F_2(t)^{\frac{1}{2}}+o(1)\|\xi\|_{L^\infty})=o(1)F_2(t)^{\frac{1}{2}},
\end{split}
\end{equation*}
where we have used the estimate
\begin{equation*}
\begin{split}
&-2\int_{S^{2n+1}}(x,\overline{x})v\big(d(\Delta_{\theta_0}v)\cdot\xi\big)dV_{\theta_0}
=2\int_{S^{2n+1}}(x,\overline{x})v\big[d(R_hv^{1+\frac{2}{n}}-R_{\theta_0}v)\cdot\xi\big]dV_{\theta_0}\\
&=2\int_{S^{2n+1}}(x,\overline{x})v\Big[d\big((R_h-\alpha f_\phi)v^{1+\frac{2}{n}}+(\alpha f_\phi-R_{\theta_0})
v^{1+\frac{2}{n}}+(v^{1+\frac{2}{n}}-v)R_{\theta_0}\big)\cdot\xi\Big]dV_{\theta_0}\\
&=o(1)\|\xi\|_{L^\infty},
\end{split}
\end{equation*}
thanks to Lemma \ref{lem4.14}, Lemma \ref{lem3.2} and Lemma \ref{lem6.2}. Similarly, we find that
\begin{equation*}
\begin{split}
I_3&=(2n+2)\int_{S^{2n+1}}(\alpha f_\phi-R_h)v\langle\nabla_{\theta_0}(x,\overline{x}),\nabla_{\theta_0}v\rangle_{\theta_0}dV_{\theta_0}\\
&\hspace{4mm}+2\int_{S^{2n+1}}v^{-(1+\frac{2}{n})}\mbox{div}_{\theta_0}(v^{2+\frac{2}{n}}\xi)
\langle\nabla_{\theta_0}(x,\overline{x}),\nabla_{\theta_0}v\rangle_{\theta_0}dV_{\theta_0}\\
&=(2n+2)\int_{S^{2n+1}}(\alpha f_\phi-R_h)v\langle\nabla_{\theta_0}(x,\overline{x}),\nabla_{\theta_0}v\rangle_{\theta_0}dV_{\theta_0}\\
&\hspace{4mm}+2(1+\frac{2}{n})\int_{S^{2n+1}}(dv\cdot\xi)
\langle\nabla_{\theta_0}(x,\overline{x}),\nabla_{\theta_0}v\rangle_{\theta_0}dV_{\theta_0}\\
&\hspace{4mm}
-2\int_{S^{2n+1}}v
\big(d(\langle\nabla_{\theta_0}(x,\overline{x}),\nabla_{\theta_0}v\rangle_{\theta_0})\cdot\xi\big)dV_{\theta_0}\\
&\leq C(\|\nabla_{\theta_0}v\|_{L^2(S^{2n+1},\theta_0)}F_2(t)^{\frac{1}{2}}+\|\xi\|_{L^\infty}\|v-1\|_{S_1^2(S^{2n+1},\theta_0)})
=o(1)F_2(t)^{\frac{1}{2}}.
\end{split}
\end{equation*}
Inserting these estimates of $I_1$, $I_2$ and $I_3$ into (\ref{6.28}), we obtain the desired result.
\end{proof}

\begin{lem}\label{lem6.5}
For sufficiently large time $t$, there holds
$$F_2(t)=(1+o(1))|B(t)|^2$$
with error $o(1)\rightarrow 0$ as $t\rightarrow\infty$.
\end{lem}
\begin{proof}
For brevity, we set
$\widehat{F}_2(t)=\displaystyle\sum_{i=2n+2}^\infty|\beta_\theta^i|^2.$
By (\ref{6.26}), we have
\begin{equation}\label{6.30}
\begin{split}
F_2(t)&=\int_{S^{2n+1}}(\alpha f-R_\theta)^2dV_\theta
=\sum_{i,j=1}^\infty\beta_\theta^i\beta_\theta^j\int_{S^{2n+1}}\varphi_i^\theta\varphi_j^\theta\,dV_\theta\\
&=\sum_{i=1}^\infty|\beta_\theta^i|^2=|\beta_\theta|^2+\widehat{F}_2(t)
=|B|^2+\widehat{F}_2(t)+o(1)F_2(t).
\end{split}
\end{equation}
Since
\begin{equation*}
\begin{split}
G_2(t)=\int_{S^{2n+1}}|\nabla_\theta(\alpha
f-R_\theta)|_{\theta}^2dV_\theta&=-\int_{S^{2n+1}}(\alpha
f-R_\theta)\Delta_\theta (\alpha f-R_\theta)dV_\theta\\
&=\sum_{i,j=1}^\infty\beta_\theta^i\beta_\theta^j\int_{S^{2n+1}}\varphi_i^\theta
\big(-\Delta_\theta\varphi_j^\theta\big)\,dV_\theta\\
&=\sum_{i,j=1}^\infty\beta_\theta^i\beta_\theta^j\lambda_j^\theta
\int_{S^{2n+1}}\varphi_i^\theta\varphi_j^\theta\,dV_\theta=\sum_{i=1}^\infty\lambda_i^\theta|\beta_\theta^i|^2,
\end{split}
\end{equation*}
we have
\begin{equation}\label{6.31}
\begin{split}
\frac{n}{2}F_2(t)-G_2(t)
&=\frac{n}{2}\sum_{i=1}^\infty|\beta_\theta^i|^2-\sum_{i=1}^\infty\lambda_i^\theta|\beta_\theta^i|^2\\
&=\frac{n}{2}\sum_{i=1}^\infty|\beta_\theta^i|^2-\sum_{i=1}^\infty\lambda_i|\beta_\theta^i|^2+
\sum_{i=1}^\infty(\lambda_i-\lambda_i^\theta)|\beta_\theta^i|^2\\
&=\sum_{i=2n+3}^\infty(\frac{n}{2}-\lambda_i)|\beta_\theta^i|^2+
o(1)\sum_{i=1}^\infty|\beta_\theta^i|^2\\
&\leq
(\frac{n}{2}-\lambda_{2n+3})\widehat{F}_2(t)+o(1)F_2(t),
\end{split}
\end{equation}
where we have used Lemma \ref{lem6.3} and the fact that
$0=\lambda_0<\lambda_1=\cdots\lambda_{2n+2}=\displaystyle\frac{n}{2}<\lambda_{2n+3}\leq\lambda_i$
for $i\geq 2n+3$. From (\ref{6.31}) and Lemma \ref{lem6.1}, we
deduce
\begin{equation}\label{6.32}
\begin{split}
\frac{d}{dt}F_2(t)&\leq (n+1+o(1))(nF_2(t)-2G_2(t))+o(1)F_2(t)\\
&\leq 2(n+1)(\frac{n}{2}-\lambda_{2n+3})\widehat{F}_2(t)+o(1)F_2(t).
\end{split}
\end{equation}

Suppose there exists some sufficiently large time $t_1$ such that
$|B(t_1)|^2\geq \widehat{F}_2(t_1)$. Denote
$$F_2(t)=(1+\delta(t))|B(t)|^2$$
near $t_1$. Then we have $\displaystyle-\frac{1}{2}\leq
\delta(t)\leq 2$ for all time $t$ sufficiently close to $t_1$ by
continuity of $\displaystyle\frac{\widehat{F}_2(t)}{|B(t)|^2}$ at
$t=t_1$. By (\ref{6.32}), we get
\begin{equation}\label{6.33}
\begin{split}
&\frac{d\delta(t)}{dt}|B(t)|^2+2(1+\delta(t))B(t)\frac{dB(t)}{dt}\\
&=\frac{d}{dt}F_2(t)\leq
2(n+1)(\frac{n}{2}-\lambda_{2n+3})\widehat{F}_2(t)+o(1)F_2(t)\\
&=2(n+1)(\frac{n}{2}-\lambda_{2n+3})\delta(t)|B(t)|^2+o(1)F_2(t).
\end{split}
\end{equation}
It follows from  Lemma \ref{lem6.4} that
$$\left|B(t)\frac{dB(t)}{dt}\right|=o(1)|B(t)|F_2(t)^{\frac{1}{2}}\leq o(1)F_2(t)$$
since $|B(t)|\leq F_2(t)$. Substituting it into (\ref{6.33}) and
dividing $|B(t)|^2$ on both sides, we find
\begin{equation*}
\frac{d\delta(t)}{dt}\leq2(n+1)(\frac{n}{2}-\lambda_{2n+3})\delta(t)+o(1)\frac{F_2(t)}{|B(t)|^2}
=\left[2(n+1)(\frac{n}{2}-\lambda_{2n+3})+o(1)\right]\delta(t).
\end{equation*}
Since $\lambda_{2n+3}>\displaystyle\frac{n}{2}$, this implies
$\delta(t)\rightarrow 0$ as $t\rightarrow\infty$, as required. It
follows from this argument that our choice of $t_1$ must satisfies
that $|o(1)|\leq (n+1)(\lambda_{2n+3}-\displaystyle\frac{n}{2})$
when $t\geq t_1$.

It reduces to seek a time $t_1$ such that $|B(t_1)|^2\geq
\widehat{F}_2(t_1)$ for sufficiently large $t_1$. Assume, on the
contrary, that $|B(t)|^2<\widehat{F}_2(t)$ for all sufficiently
large $t$. Therefore by (\ref{6.30})
$$F_2(t)
=|B|^2+\widehat{F}_2(t)+o(1)F_2(t)<2\widehat{F}_2(t)+o(1)F_2(t),$$
which implies that
\begin{equation*}
\begin{split}
\frac{d}{dt}F_2(t) &\leq
-2(n+1)(\lambda_{2n+3}-\frac{n}{2})\widehat{F}_2(t)+o(1)F_2(t)\\
&\leq-(n+1)(\lambda_{2n+3}-\frac{n}{2})F_2(t)+o(1)F_2(t)
\end{split}
\end{equation*}
by (\ref{6.32}). Hence, we have
\begin{equation}\label{6.34}
F_2(t)\leq C e^{-\frac{(n+1)}{2}(\lambda_{2n+3}-\frac{n}{2})t}
\end{equation}
for $t\geq t_2$ and $C$ depending only on $t_2$. Let $Q$ be the
unique concentration point described in Theorem \ref{thm4.7}, and
$B_{r_0}(Q)=B_{r_0}(Q,\theta_0)$. For any $r_0>0$, we have
\begin{equation*}
\begin{split}
\left|\frac{d}{dt}\mbox{Vol}(B_{r_0}(Q),\theta)\right|
&=\left|\frac{d}{dt}\left(\int_{B_{r_0}(Q)}dV_\theta\right)\right|=(n+1)\left|\int_{B_{r_0}(Q)}(\alpha
f-R_\theta)dV_\theta\right|\\
&\leq(n+1)\mbox{Vol}(S^{2n+1},\theta)^{\frac{1}{2}}\left(\int_{B_{r_0}(Q)}(\alpha
f-R_\theta)^2dV_\theta\right)^{\frac{1}{2}}\\
&\leq
(n+1)\mbox{Vol}(S^{2n+1},\theta_0)^{\frac{1}{2}}F_2(t)^{\frac{1}{2}}\\
&\leq C
e^{-\frac{(n+1)}{4}(\lambda_{2n+3}-\frac{n}{2})t}\hspace{2mm}\mbox{
for }t\geq t_2,
\end{split}
\end{equation*}
by (\ref{2.3}) and (\ref{6.34}).
Thus, by integrating the above inequality from $t_2$ to a larger
$t$, we get
\begin{equation}\label{6.35}
\begin{split}\mbox{Vol}(B_{r_0}(Q),\theta(t))&<\mbox{Vol}(B_{r_0}(Q),\theta(t_2))
+\frac{4C}{(n+1)(\lambda_{2n+3}-\frac{n}{2})}\,
e^{-\frac{(n+1)}{4}(\lambda_{2n+3}-\frac{n}{2})t_2}\\
&<\mbox{Vol}(S^{2n+1},\theta_0)/2
\end{split}
\end{equation}
uniformly for $t\geq t_2$ by
first choosing $t_2$ sufficiently large
and then choosing $r_0$ sufficiently small.  On the other hand, from Theorem
\ref{thm4.7}, we know that
$$\mbox{Vol}(B_{r_0}(Q),\theta(t))\rightarrow\mbox{Vol}(S^{2n+1},\theta_0)\hspace{2mm}
\mbox{ as }t\rightarrow\infty$$ which yields a contradiction with
(\ref{6.35}). Thus the proof is complete.
\end{proof}

\begin{lem}\label{lem6.6}
With a uniform constant $C>0$, there holds
$$\|v-1\|_{S_2^2(S^{2n+1},\theta_0)}\leq C(F_2(t)^{\frac{1}{2}}
+\|f_\phi-f(\Theta(t))\|_{L^2(S^{2n+1},\theta_0)}).$$
\end{lem}
\begin{proof}
Expand $v^{2+\frac{2}{n}}-1$ and $v-1$ in terms of eigenfunctions to
get
$$v^{2+\frac{2}{n}}-1=\sum_{i=0}^\infty
V^i\varphi_i\hspace{2mm}\mbox{ and
}\hspace{2mm}v-1=\sum_{i=0}^\infty v^i\varphi_i.$$ By
Proposition 2.1 in \cite{Ho3}, we have
$$\int_{S^{2n+1}}(v^{2+\frac{2}{n}}-1)dV_{\theta_0}=\int_{S^{2n+1}}(u^{2+\frac{2}{n}}-1)dV_{\theta_0}=0$$
which implies that $V^0=0$. On the other hand, due to the
normalization (\ref{4.19}) of $v$, we have $V^i=0$ for $1\leq i\leq
2n+2$. Observe that by Taylor's expansion and Lemma \ref{lem4.14},
\begin{equation*}
\begin{split}
(2+\frac{2}{n})v^i&=(2+\frac{2}{n})\int_{S^{2n+1}}(v-1)\varphi_i\,dV_{\theta_0}\\
&=\int_{S^{2n+1}}(v^{2+\frac{2}{n}}-1)\varphi_i\,dV_{\theta_0}+O(\|v-1\|^2_{S^2_1(S^{2n+1},\theta_0)})\\
&=V^i+o(1)\|v-1\|_{S^2_1(S^{2n+1},\theta_0)}.
\end{split}
\end{equation*}
Thus it follows that
\begin{equation}\label{6.36}
\sum_{i=0}^{2n+2}|v^i|^2=o(1)\|v-1\|_{S^2_1(S^{2n+1},\theta_0)}^2.
\end{equation}

We may rewrite (\ref{4.22}) in the form
\begin{equation}\label{6.37}
\begin{split}
-(2+\frac{2}{n})\Delta_{\theta_0}v&=(R_hv^{1+\frac{2}{n}}-R_{\theta_0}v)\\
&=\left[\vphantom{\left(\alpha
f(\Theta(t))-\frac{1}{\mbox{Vol}(S^{2n+1},\theta_0)}\int_{S^{2n+1}}R_hdV_h\right)}(R_h-\alpha f_\phi)+(\alpha f_\phi-\alpha
f(\Theta(t)))\right.\\
&\hspace{4mm}\left.+\left(\alpha
f(\Theta(t))-\frac{1}{\mbox{Vol}(S^{2n+1},\theta_0)}\int_{S^{2n+1}}R_hdV_h\right)\right]v^{1+\frac{2}{n}}\\
&\hspace{4mm}+
\left(\frac{1}{\mbox{Vol}(S^{2n+1},\theta_0)}\int_{S^{2n+1}}R_hdV_h
-R_{\theta_0}\right)v^{1+\frac{2}{n}}+R_{\theta_0}(v^{1+\frac{2}{n}}-v).
\end{split}
\end{equation}
We are going to estimate the terms on the right hand side of (\ref{6.37}).
By (4.34), Proposition 2.1 and Lemma 2.4 in \cite{Ho3}, and by Lemma
\ref{lem4.14}, we have
\begin{equation}\label{6.38}
\begin{split}
&\left|\alpha
f(\Theta(t))\mbox{Vol}(S^{2n+1},\theta_0)-\int_{S^{2n+1}}R_hdV_h\right|\\
&\leq\left|\alpha \int_{S^{2n+1}}(
f(\Theta(t))-f_\phi)dV_h\right|+\left|\int_{S^{2n+1}}(\alpha
f-R_h)dV_h\right|\\
&\leq
C\|f(\Theta(t))-f_\phi\|_{L^2(S^{2n+1},\theta_0)}+\mbox{Vol}(S^{2n+1},\theta_0)^{\frac{1}{2}}\left(\int_{S^{2n+1}}(\alpha
f-R_h)^2dV_h\right)^{\frac{1}{2}}.
\end{split}
\end{equation}
On the other hand,
\begin{equation*}
\begin{split}
E(v-1)&=\int_{S^{2n+1}}\left((2+\frac{2}{n})
|\nabla_{\theta_0}v|_{\theta_0}^2+R_{\theta_0}(v-1)^2\right)dV_{\theta_0}\\
&=\int_{S^{2n+1}}\left((2+\frac{2}{n})|\nabla_{\theta_0}v|_{\theta_0}^2+R_{\theta_0}v^2\right)dV_{\theta_0}
-2R_{\theta_0}\int_{S^{2n+1}}(v-1)dV_{\theta_0}\\
&\hspace{4mm}-R_{\theta_0}\int_{S^{2n+1}}dV_{\theta_0}\\
&=\int_{S^{2n+1}}R_hdV_{h}-2R_{\theta_0}\int_{S^{2n+1}}(v-1)dV_{\theta_0}
-R_{\theta_0}\mbox{Vol}(S^{2n+1},\theta_0),
\end{split}
\end{equation*}
which implies that
\begin{equation}\label{6.39}
\begin{split}
\left|\int_{S^{2n+1}}R_hdV_{h}-R_{\theta_0}\mbox{Vol}(S^{2n+1},\theta_0)\right|
&\leq
E(v-1)+2R_{\theta_0}\left|\int_{S^{2n+1}}(v-1)dV_{\theta_0}\right|\\
&=E(v-1)+C|v^0|\\
&\leq E(v-1)+o(1)\|v-1\|_{S^2_1(S^{2n+1},\theta_0)}
\end{split}
\end{equation}
by (\ref{6.36}). Observe that we have
\begin{equation}\label{6.40}
\begin{split}
E(v-1)&=(2+\frac{2}{n})\int_{S^{2n+1}}|\nabla_{\theta_0}v|_{\theta_0}^2dV_{\theta_0}\\
&\hspace{4mm}
+R_{\theta_0}\int_{S^{2n+1}}\left(v-\frac{1}{\mbox{Vol}(S^{2n+1},\theta_0)}
\int_{S^{2n+1}}v\,dV_{\theta_0}\right)^2dV_{\theta_0}\\
&\hspace{4mm}+\frac{R_{\theta_0}}{\mbox{Vol}(S^{2n+1},\theta_0)}
\left(\int_{S^{2n+1}}(v-1)dV_{\theta_0}\right)^2.
\end{split}
\end{equation}
Since the first eigenvalue of the sub-Laplacian of $\theta_0$ is
$n/2$, together with (\ref{6.36}), we obtain from (\ref{6.40}) that
\begin{equation}\label{6.41}
\begin{split}
E(v-1)&\leq(2+\frac{2}{n})\int_{S^{2n+1}}|\nabla_{\theta_0}v|_{\theta_0}^2dV_{\theta_0}
+\frac{2R_{\theta_0}}{n}\int_{S^{2n+1}}|\nabla_{\theta_0}v|_{\theta_0}^2dV_{\theta_0}\\
&\hspace{4mm}+o(1)\|v-1\|_{S^2_1(S^{2n+1},\theta_0)}^2\\
&=(1+\frac{2}{n})(n+1)\int_{S^{2n+1}}|\nabla_{\theta_0}v|_{\theta_0}^2dV_{\theta_0}
+o(1)\|v-1\|_{S^2_1(S^{2n+1},\theta_0)}^2.
\end{split}
\end{equation}
We also need the following:
\begin{equation}\label{6.42}
v^{1+\frac{2}{n}}-v=\frac{2}{n}(v-1)+o(|v-1|).
\end{equation}

Since $\lambda_{2n+3}\geq(1-2\epsilon(n))\displaystyle\frac{n}{2}$ for some constant $\epsilon(n)>0$ depending only $n$,
with sufficiently small $\delta>0$ and $\epsilon_0>0$ we have
\begin{equation}\label{6.43}
(2+\frac{2}{n})^{-2}R_{\theta_0}^2(\frac{2}{n})^2(1+\delta)(1+\epsilon_0)
=\frac{n^2}{4}(1+\delta)(1+\epsilon_0)\leq(1-\epsilon(n))\lambda_{2n+3}^2,
\end{equation}
we find that
\begin{equation}\label{6.44}
\begin{split}
&\sum_{i=0}^\infty\lambda_i^2|v^i|^2=\|\Delta_{\theta_0}v\|^2_{L^2(S^{2n+1},\theta_0)}\\
&\leq C(\delta)\left[\|(R_h-\alpha f_\phi)v^{1+\frac{2}{n}}\|_{L^2(S^{2n+1},\theta_0)}^2+\|(\alpha f_\phi-\alpha
f(\Theta(t)))v^{1+\frac{2}{n}}\|_{L^2(S^{2n+1},\theta_0)}^2\vphantom{\left(\alpha
f(\Theta(t))-\frac{1}{\mbox{Vol}(S^{2n+1},\theta_0)}\int_{S^{2n+1}}R_hdV_h\right)^2}\right.\\
&\hspace{8mm}\left.+\left(\alpha
f(\Theta(t))-\frac{1}{\mbox{Vol}(S^{2n+1},\theta_0)}\int_{S^{2n+1}}R_hdV_h\right)^2
\|v^{1+\frac{2}{n}}\|_{L^2(S^{2n+1},\theta_0)}^2\right]\\
&\hspace{4mm}+(2+\frac{2}{n})^{-2}(1+\delta)
\left[(1+\epsilon_0)R_{\theta_0}^2\|v^{1+\frac{2}{n}}-v\|_{L^2(S^{2n+1},\theta_0)}^2\vphantom{\left(\alpha
f(\Theta(t))-\frac{1}{\mbox{Vol}(S^{2n+1},\theta_0)}\int_{S^{2n+1}}R_hdV_h\right)^2}\right.\\
&\hspace{8mm}\left.+(1+\epsilon_0^{-1})
\left(\frac{1}{\mbox{Vol}(S^{2n+1},\theta_0)}\int_{S^{2n+1}}R_hdV_h
-R_{\theta_0}\right)^2\|v^{1+\frac{2}{n}}\|_{L^2(S^{2n+1},\theta_0)}^2\right]\\
&\leq  C(\delta)\Big(F_2(t)+\| f_\phi-
f(\Theta(t))\|_{L^2(S^{2n+1},\theta_0)}^2\Big)\\
&\hspace{4mm}+(2+\frac{2}{n})^{-2}R_{\theta_0}^2(\frac{2}{n})^2(1+\delta)
(1+\epsilon_0)\|v-1\|_{L^2(S^{2n+1},\theta_0)}^2
+o(1)\|v-1\|_{S^2_1(S^{2n+1},\theta_0)}^2\\
&\hspace{4mm}+2(2+\frac{2}{n})^{-2}(1+\frac{2}{n})^2(n+1)^2\cdot
\frac{(1+\delta)(1+\epsilon_0^{-1})}{\mbox{Vol}(S^{2n+1},\theta_0)^2}
\left(\int_{S^{2n+1}}|\nabla_{\theta_0}v|_{\theta_0}^2dV_{\theta_0}\right)^2
\\
&\leq C(\delta)\Big(F_2(t)+\| f_\phi-
f(\Theta(t))\|_{L^2(S^{2n+1},\theta_0)}^2\Big)+(1-\epsilon(n))^2\lambda_{2n+3}^2\sum_{i=0}^\infty|v^i|^2
 +o(1)\|v-1\|_{S^2_1(S^{2n+1},\theta_0)}^2\\
&\hspace{4mm}
+\frac{(n+2)^2}{2}\cdot\frac{(1+\delta)(1+\epsilon_0^{-1})}{\mbox{Vol}(S^{2n+1},\theta_0)^2}
\left(\int_{S^{2n+1}}|\Delta_{\theta_0}v|^2dV_{\theta_0}\right)
\left(\int_{S^{2n+1}}(v-1)^2dV_{\theta_0}\right)
\end{split}
\end{equation}
where the first  inequality follows from (\ref{6.37}) and Young's
inequality, and the second inequality follows from (\ref{6.38}), (\ref{6.39}),
(\ref{6.41}), (\ref{6.42}), and
Lemma \ref{lem4.14}, and the last inequality follows from
(\ref{6.43}) and
\begin{equation*}
\begin{split}
\left(\int_{S^{2n+1}}|\nabla_{\theta_0}v|_{\theta_0}^2dV_{\theta_0}\right)^2
&=\left(\int_{S^{2n+1}}(v-1)\Delta_{\theta_0}v\,dV_{\theta_0}\right)^2\\
&\leq\left(\int_{S^{2n+1}}|\Delta_{\theta_0}v|^2dV_{\theta_0}\right)
\left(\int_{S^{2n+1}}(v-1)^2dV_{\theta_0}\right)
\end{split}
\end{equation*}
by H\"{o}lder's inequality. Since
$\displaystyle\int_{S^{2n+1}}(v-1)^2dV_{\theta_0}\rightarrow 0$ as
$t\rightarrow\infty$, we can choose sufficiently large $t_0$ such
that if $t\geq t_0$, then
\begin{equation}\label{6.45}
\frac{(n+2)^2}{2}\cdot\frac{(1+\delta)(1+\epsilon_0^{-1})}{\mbox{Vol}(S^{2n+1},\theta_0)}
\left(\int_{S^{2n+1}}(v-1)^2dV_{\theta_0}\right)<\frac{1}{2}.
\end{equation}
Thus by (\ref{6.36}) and (\ref{6.45}), we can absorb the last three
terms on the right hand side of (\ref{6.44}) to conclude that
\begin{equation}\label{6.46}
\int_{S^{2n+1}}|\Delta_{\theta_0}v|^2dV_{\theta_0}\leq
C(\delta)\Big(F_2(t)+\| f_\phi-
f(\Theta(t))\|_{L^2(S^{2n+1},\theta_0)}^2\Big)
+o(1)\|v-1\|_{S^2_1(S^{2n+1},\theta_0)}^2.
\end{equation}

Now we also have
\begin{equation}\label{6.47}
\begin{split}
&\frac{1}{2}\int_{S^{2n+1}}|\nabla_{\theta_0}v|_{\theta_0}^2dV_{\theta_0}+\int_{S^{2n+1}}(v-1)^2dV_{\theta_0}\\
&=\frac{1}{2}\int_{S^{2n+1}}|\nabla_{\theta_0}v|_{\theta_0}^2dV_{\theta_0}
+\int_{S^{2n+1}}\left(v-\frac{1}{\mbox{Vol}(S^{2n+1},\theta_0)}
\int_{S^{2n+1}}v\,dV_{\theta_0}\right)^2dV_{\theta_0}\\
&\hspace{4mm} +\frac{1}{\mbox{Vol}(S^{2n+1},\theta_0)}\left(\int_{S^{2n+1}}(v-1)dV_{\theta_0}\right)^2\\
&\leq(\frac{1}{2}+\frac{2}{n})\int_{S^{2n+1}}|\nabla_{\theta_0}v|_{\theta_0}^2dV_{\theta_0}
+o(1)\|v-1\|_{S^2_1(S^{2n+1},\theta_0)}^2\\
&=(\frac{1}{2}+\frac{2}{n})\int_{S^{2n+1}}(v-1)\Delta_{\theta_0}v\,dV_{\theta_0}
+o(1)\|v-1\|_{S^2_1(S^{2n+1},\theta_0)}^2\\
&\leq\frac{1}{2}(\frac{1}{2}+\frac{2}{n})^2\int_{S^{2n+1}}|\Delta_{\theta_0}v|^2dV_{\theta_0}
+\frac{1}{2}\int_{S^{2n+1}}(v-1)^2dV_{\theta_0}
+o(1)\|v-1\|_{S^2_1(S^{2n+1},\theta_0)}^2
\end{split}
\end{equation}
where the first inequality follows from (\ref{6.36}) and the fact
that the first eigenvalue $\lambda_1$ of the sub-Laplacian for
$\theta_0$ is $n/2$. By absorbing the second term on the right hand
side of (\ref{6.47}) to the left hand side, we get
\begin{equation*}
\begin{split}
&\frac{1}{2}\int_{S^{2n+1}}(v-1)^2dV_{\theta_0}+\frac{1}{2}\int_{S^{2n+1}}|\nabla_{\theta_0}v|_{\theta_0}^2dV_{\theta_0}\\
&\leq\frac{1}{2}(\frac{1}{2}+\frac{2}{n})^2\int_{S^{2n+1}}|\Delta_{\theta_0}v|^2dV_{\theta_0}
+o(1)\|v-1\|_{S^2_1(S^{2n+1},\theta_0)}^2\\
&\leq C(\delta)\Big(F_2(t)+\| f_\phi-
f(\Theta(t))\|_{L^2(S^{2n+1},\theta_0)}^2\Big)
+o(1)\|v-1\|_{S^2_1(S^{2n+1},\theta_0)}^2
\end{split}
\end{equation*}
by (\ref{6.46}). Hence we conclude that
\begin{equation}\label{6.48}
\int_{S^{2n+1}}(v-1)^2dV_{\theta_0}+\int_{S^{2n+1}}|\nabla_{\theta_0}v|_{\theta_0}^2dV_{\theta_0}
\leq C(\delta)\Big(F_2(t)+\| f_\phi-
f(\Theta(t))\|_{L^2(S^{2n+1},\theta_0)}^2\Big).
\end{equation}
Substituting (\ref{6.48}) back to (\ref{6.46}), we obtain
\begin{equation}\label{6.49}
\int_{S^{2n+1}}|\Delta_{\theta_0}v|^2dV_{\theta_0}\leq
C(\delta)\Big(F_2(t)+\| f_\phi-
f(\Theta(t))\|_{L^2(S^{2n+1},\theta_0)}^2\Big).
\end{equation}
Now the assertion follows from (\ref{6.48}) and (\ref{6.49}).
\end{proof}

\begin{lem}\label{lem6.7}
For all $t>0$, there hold
\begin{equation*}
\begin{split}
&b-\langle b,(\widehat{\Theta(t)},\widehat{\Theta(t)})\rangle(\widehat{\Theta(t)},\widehat{\Theta(t)})\\
&=\epsilon A_1 \left(\frac{\partial
f(\widehat{\Theta(t)})}{\partial a_1} +\sqrt{-1}\frac{\partial
f(\widehat{\Theta(t)})}{\partial b_1} ,\cdots,\frac{\partial
f(\widehat{\Theta(t)})}{\partial a_n} +\sqrt{-1}\frac{\partial
f(\widehat{\Theta(t)})}{\partial b_n} ,0,\right.\\
&\hspace{16mm}\left.\frac{\partial
f(\widehat{\Theta(t)})}{\partial a_1} -\sqrt{-1}\frac{\partial
f(\widehat{\Theta(t)})}{\partial b_1} ,\cdots,\frac{\partial
f(\widehat{\Theta(t)})}{\partial a_n} -\sqrt{-1}\frac{\partial
f(\widehat{\Theta(t)})}{\partial b_n} ,0\right)+O(\epsilon^2)
\end{split}
\end{equation*}
 and
$$\langle b,(\widehat{\Theta(t)},\widehat{\Theta(t)})\rangle=-2\epsilon^2A_2\alpha\Delta_{\theta_0}
f(\widehat{\Theta(t)})+O(\epsilon^2)|\nabla_{\theta_0}
f(\widehat{\Theta(t)})|_{\theta_0}^2+O(\epsilon^3),$$ where $A_1$ and $A_2$ are
positive constants defined as in $(\ref{6.60})$ and $(\ref{6.65})$
respectively.
\end{lem}
\begin{proof}
Using (\ref{5.1}), we find
\begin{equation*}
\begin{split}
\frac{n}{2}b&=\frac{n}{2}\int_{S^{2n+1}}(x,\overline{x})(\alpha f_\phi-R_h)dV_h
=-\int_{S^{2n+1}}\Delta_{\theta_0}(x,\overline{x})(\alpha f_\phi-R_h)dV_h\\
&=\alpha\int_{S^{2n+1}}\langle\nabla_{\theta_0}(x,\overline{x}),\nabla_{\theta_0}f_\phi\rangle_{\theta_0}dV_h+E_1
\end{split}
\end{equation*}
with error
\begin{equation*}
E_1=(2+\frac{2}{n})\int_{S^{2n+1}}(\alpha f_\phi-R_h)\langle\nabla_{\theta_0}(x,\overline{x}),\nabla_{\theta_0}v\rangle_{\theta_0}
v^{1+\frac{2}{n}}dV_{\theta_0}.
\end{equation*}
By H\"{o}lder's inequality and Lemma \ref{lem4.14}, the error term can be estimated as
\begin{equation}\label{6.50}
|E_1|\leq C\|\nabla_{\theta_0}v\|_{L^2(S^{2n+1},\theta_0)} \|\alpha
f_\phi-R_h\|_{L^2(S^{2n+1},h)}\leq
C\|v-1\|_{S_1^2(S^{2n+1},\theta_0)}F_2^{\frac{1}{2}}.
\end{equation}
Thus we obtain
\begin{equation}\label{6.51}
\begin{split}
\frac{n}{2}b
&=\alpha\int_{S^{2n+1}}\langle\nabla_{\theta_0}(x,\overline{x}),\nabla_{\theta_0}f_\phi\rangle_{\theta_0}dV_{\theta_0}+E_1+E_2\\
&=\alpha\int_{S^{2n+1}}\Delta_{\theta_0}(x,\overline{x})(f_\phi-f(\widehat{\Theta(t)}))dV_{\theta_0}+E_1+E_2\\
&=\frac{n}{2}\alpha\int_{S^{2n+1}}(x,\overline{x})(f_\phi-f(\widehat{\Theta(t)}))dV_{\theta_0}+E_1+E_2,
\end{split}
\end{equation}
where
\begin{equation*}
\begin{split}
E_2&=\alpha\int_{S^{2n+1}}\langle\nabla_{\theta_0}
(x,\overline{x}),\nabla_{\theta_0}f_\phi\rangle_{\theta_0}(v^{2+\frac{2}{n}}-1)dV_{\theta_0}\\
&=\frac{n}{2}\alpha\int_{S^{2n+1}}(x,\overline{x})(f_\phi-f(\widehat{\Theta(t)}))(v^{2+\frac{2}{n}}-1)dV_{\theta_0}\\
&\hspace{4mm}
-(2+\frac{2}{n})\alpha\int_{S^{2n+1}}\langle\nabla_{\theta_0}(x,\overline{x}),\nabla_{\theta_0}v\rangle_{\theta_0}
(f_\phi-f(\widehat{\Theta(t)}))v^{1+\frac{2}{n}}dV_{\theta_0}.
\end{split}
\end{equation*}
Then it follows from H\"{o}lder's inequality and Lemma \ref{lem4.14}
that
\begin{equation}\label{6.52}
\begin{split}
|E_2|
&\leq C(\|v-1\|_{L^2(S^{2n+1},\theta_0)}+\|\nabla_{\theta_0}v\|_{L^2(S^{2n+1},\theta_0)})
\|f_\phi-f(\widehat{\Theta(t)})\|_{L^2(S^{2n+1},\theta_0)}\\
&\leq\|v-1\|_{S_1^2(S^{2n+1},\theta_0)}\|f_\phi-f(\widehat{\Theta(t)})\|_{L^2(S^{2n+1},\theta_0)}.
\end{split}
\end{equation}

We henceforth focus on the term
$$\int_{S^{2n+1}}(x,\overline{x})(f_\phi-f(\widehat{\Theta(t)}))dV_{\theta_0}.$$
We will keep the coordinate for $\mathbb{H}^n$ such that the north pole $N$ of $S^{2n+1}$ is $\widehat{\Theta(t)}$. If we use the
tangent plane of the sphere at the north pole $N=(0,...,0,1)$ as local coordinates for $S^{2n+1}$, then
$$\phi(z,\tau)=\left(\frac{2\epsilon z}{1+\epsilon^2|z|^2-\sqrt{-1}\epsilon^2\tau},
\frac{1-\epsilon^2|z|^2+\sqrt{-1}\epsilon^2\tau}{1+\epsilon^2|z|^2-\sqrt{-1}\epsilon^2\tau}\right),\hspace{2mm}(z,\tau)\in\mathbb{H}^n$$
where $\epsilon=\displaystyle\frac{1}{r(t)}$.
Hence, in $B_{\epsilon^{-1}}(0)$, we can expand around $(z,\tau)=(0,0)$
\begin{equation}\label{6.53}
\begin{split}
&(f_\phi-f(\widehat{\Theta(t)}))(\Psi(z,\tau))\\
&=
\sum_{i=1}^n\frac{\partial f(\phi(z,\tau))}{\partial a_i}\Big|_{(z,\tau)=(0,0)}a_i
+\sum_{i=1}^n\frac{\partial f(\phi(z,\tau))}{\partial b_i}\Big|_{(z,\tau)=(0,0)}b_i\\
&\hspace{4mm}+\frac{\partial f(\phi(z,\tau))}{\partial\tau}\Big|_{(z,\tau)=(0,0)}\tau
+\frac{1}{2}\frac{\partial^2 f(\phi(z,\tau))}{\partial\tau^2}\Big|_{(z,\tau)=(0,0)}\tau^2\\
&\hspace{4mm}
+\frac{1}{2}\sum_{i,j=1}^n\left(\frac{\partial^2 f(\phi(z,\tau))}{\partial a_i\partial a_j}\Big|_{(z,\tau)=(0,0)}a_ia_j
+\frac{\partial^2 f(\phi(z,\tau))}{\partial b_i\partial b_j}\Big|_{(z,\tau)=(0,0)}b_ib_j\right)\\
&\hspace{4mm}
+\frac{1}{2}\sum_{i=1}^n\left(\frac{\partial^2 f(\phi(z,\tau))}{\partial a_i\partial\tau}\Big|_{(z,\tau)=(0,0)}a_i\tau
+\frac{\partial^2 f(\phi(z,\tau))}{\partial b_i\partial \tau}\Big|_{(z,\tau)=(0,0)}b_i\tau\right)\\
&\hspace{4mm}
+O(\epsilon^3(|z|^4+\tau^2)^{\frac{3}{4}})\\
&=[df(\widehat{\Theta(t)})\cdot d\phi|_{(z,\tau)=(0,0)}]\cdot(z,\tau)\\
&\hspace{4mm}+\frac{1}{2}(\nabla df)(\widehat{\Theta(t)})(d\phi|_{(z,\tau)=(0,0)}(z,\tau),d\phi|_{(z,\tau)=(0,0)}(z,\tau))
+O(\epsilon^3(|z|^4+\tau^2)^{\frac{3}{4}})\\
&=df(\widehat{\Theta(t)})(\epsilon z,\epsilon^2\tau)
+\frac{1}{2}(\nabla df)(\widehat{\Theta(t)})((\epsilon z,\epsilon^2\tau),(\epsilon z,\epsilon^2\tau))
+O(\epsilon^3(|z|^4+\tau^2)^{\frac{3}{4}}),
\end{split}
\end{equation}
where $z=(z_1,...,z_n)=(a_1+\sqrt{-1}b_1,...,a_n+\sqrt{-1}b_n)\in\mathbb{C}^n$.

First, by (\ref{6.3}) and the boundedness of $x$ and $f$, we have
\begin{equation}\label{6.54}
\begin{split}
\left|\int_{\Psi(\mathbb{H}^n\setminus B_{\epsilon^{-1}}(0))}(x,\overline{x})(f_\phi-f(\widehat{\Theta(t)}))dV_{\theta_0}\right|
&\leq C\int_{\mathbb{H}^n\setminus B_{\epsilon^{-1}}(0)}\frac{dzd\tau}{(\tau^2+(1+|z|^2)^2)^{n+1}}\\
&=O(\epsilon^{2n}).
\end{split}
\end{equation}
Using (\ref{6.54}), we can give a estimate of
\begin{equation}\label{6.55}
\begin{split}
&\int_{S^{2n+1}}|f_\phi-f(\widehat{\Theta(t)})|^2dV_{\theta_0}\\
&=\int_{B_{\epsilon^{-1}}(0)}|f(\Psi(z,\tau))-f(\widehat{\Theta(t)})|^2\frac{4^{n+1}dzd\tau}{(\tau^2+(1+|z|^2)^2)^{n+1}}
+O(\epsilon^{2n})\\
&\leq
C\epsilon^2|\nabla_{\theta_0}f(\widehat{\Theta(t)})|^2_{\theta_0}
\int_{B_{\epsilon^{-1}}(0)}\frac{(\tau^2+|z|^4)^{\frac{1}{2}}dzd\tau}{(\tau^2+(1+|z|^2)^2)^{n+1}}\\
&\hspace{4mm}+C\epsilon^4\int_{B_{\epsilon^{-1}}(0)}\frac{(\tau^2+|z|^4)dzd\tau}{(\tau^2+(1+|z|^2)^2)^{n+1}}
+O(\epsilon^{2n})\\
&\leq
C|\nabla_{\theta_0}f(\widehat{\Theta(t)})|^2_{\theta_0}\epsilon^2+C\epsilon^3,
\end{split}
\end{equation}
where we have used the estimate (\ref{6.56}) in the last step.
Next, by (\ref{6.53}), we have
\begin{equation}\label{6.57}
\begin{split}
&\int_{B_{\epsilon^{-1}}(0)}(x,\overline{x})(f(\Psi(z))-f(\widehat{\Theta(t)}))\frac{4^{n+1}dzd\tau}{(\tau^2+(1+|z|^2)^2)^{n+1}}\\
&=\int_{B_{\epsilon^{-1}}(0)}(x,\overline{x})df(\widehat{\Theta(t)})(\epsilon z,\epsilon^2\tau)\frac{4^{n+1}dzd\tau}{(\tau^2+(1+|z|^2)^2)^{n+1}}\\
&\hspace{4mm}
+\frac{1}{2}\int_{B_{\epsilon^{-1}}(0)}(x,\overline{x})(\nabla df)(\widehat{\Theta(t)})((\epsilon z,\epsilon^2\tau),
(\epsilon z,\epsilon^2\tau))\frac{4^{n+1}dzd\tau}{(\tau^2+(1+|z|^2)^2)^{n+1}}+E_3\\
&:=(I_1,I_2)+(II_1,II_2)+E_3
\end{split}
\end{equation}
where
\begin{equation*}
\begin{split}
I_1&=\int_{B_{\epsilon^{-1}}(0)}xdf(\widehat{\Theta(t)})(\epsilon z,\epsilon^2\tau)\frac{4^{n+1}dzd\tau}{(\tau^2+(1+|z|^2)^2)^{n+1}},\\
I_2&=\int_{B_{\epsilon^{-1}}(0)}\overline{x}df(\widehat{\Theta(t)})(\epsilon z,\epsilon^2\tau)\frac{4^{n+1}dzd\tau}{(\tau^2+(1+|z|^2)^2)^{n+1}},\\
II_1&=\frac{1}{2}\int_{B_{\epsilon^{-1}}(0)}x(\nabla df)(\widehat{\Theta(t)})((\epsilon z,\epsilon^2\tau),
(\epsilon z,\epsilon^2\tau))\frac{4^{n+1}dzd\tau}{(\tau^2+(1+|z|^2)^2)^{n+1}},\\
II_2&=\frac{1}{2}\int_{B_{\epsilon^{-1}}(0)}\overline{x}(\nabla df)(\widehat{\Theta(t)})((\epsilon z,\epsilon^2\tau),
(\epsilon z,\epsilon^2\tau))\frac{4^{n+1}dzd\tau}{(\tau^2+(1+|z|^2)^2)^{n+1}},
\end{split}
\end{equation*}
with the error $E_3$ bounded by
\begin{equation}\label{6.58}
|E_3|\leq C\epsilon^3\int_{B_{\epsilon^{-1}}(0)}\frac{(|z|^4+\tau^2)^{\frac{3}{4}}}{(\tau^2+(1+|z|^2)^2)^{n+1}}dzd\tau
=O(\epsilon^{3})
\end{equation}
by (\ref{6.56}). Now we estimate $b$ by dividing its components into
two cases:\\
\textit{Case (i).} We deal with the tangential part first. By
(\ref{6.53}) and symmetry, we have
\begin{equation}\label{6.59}
\begin{split}
&I_1-\langle I_1, \widehat{\Theta(t)}\rangle\widehat{\Theta(t)}\\
 &=
\int_{B_{\epsilon^{-1}}(0)}\Big(\frac{2z}{1+|z|^2-\sqrt{-1}\tau},0\Big)df(\widehat{\Theta(t)})(\epsilon z,\epsilon^2\tau)\frac{4^{n+1}dzd\tau}{(\tau^2+(1+|z|^2)^2)^{n+1}}\\
&=
\int_{B_{\epsilon^{-1}}(0)}\big(2z_1(1+|z|^2+\sqrt{-1}\tau),\cdots,
2z_n(1+|z|^2+\sqrt{-1}\tau),0\big)\\
&\hspace{4mm}\cdot\left(\sum_{i=1}^n\frac{\partial
f(\widehat{\Theta(t)})}{\partial a_i}\epsilon
a_i+\sum_{i=1}^n\frac{\partial f(\widehat{\Theta(t)})}{\partial
b_i}\epsilon b_i+\frac{\partial f(\widehat{\Theta(t)})}{\partial
\tau}\epsilon^2\tau\right)\frac{4^{n+1}dzd\tau}{(\tau^2+(1+|z|^2)^2)^{n+2}}\\
&= 2\epsilon\int_{B_{\epsilon^{-1}}(0)}\left(\frac{\partial
f(\widehat{\Theta(t)})}{\partial a_1} a_1^2+\sqrt{-1}\frac{\partial
f(\widehat{\Theta(t)})}{\partial b_1} b_1^2,\cdots,\frac{\partial
f(\widehat{\Theta(t)})}{\partial a_n} a_n^2+\sqrt{-1}\frac{\partial
f(\widehat{\Theta(t)})}{\partial b_n} b_n^2,0\right)\\
&\hspace{4mm}\cdot
\frac{4^{n+1}(1+|z|^2)}{(\tau^2+(1+|z|^2)^2)^{n+2}}dzd\tau\\
&= \epsilon A_1 \left(\frac{\partial
f(\widehat{\Theta(t)})}{\partial a_1} +\sqrt{-1}\frac{\partial
f(\widehat{\Theta(t)})}{\partial b_1} ,\cdots,\frac{\partial
f(\widehat{\Theta(t)})}{\partial a_n} +\sqrt{-1}\frac{\partial
f(\widehat{\Theta(t)})}{\partial b_n} ,0\right)+O(\epsilon^{2n+1}),
\end{split}
\end{equation}
where
\begin{equation}\label{6.60}
A_1=\frac{1}{n}\int_{\mathbb{H}^n}\frac{4^{n+1}|z|^2(1+|z|^2)dzd\tau}{(\tau^2+(1+|z|^2)^2)^{n+2}}>0,
\end{equation}
since
\begin{equation*}
\int_{\mathbb{H}^n\setminus
B_{\epsilon^{-1}}(0)}\frac{4^{n+1}|z|^2(1+|z|^2)dzd\tau}{(\tau^2+(1+|z|^2)^2)^{n+2}}
\leq \int_{\mathbb{H}^n\setminus
B_{\epsilon^{-1}}(0)}\frac{4^{n+1}dzd\tau}{(\tau^2+(1+|z|^2)^2)^{n+1}}=O(\epsilon^{2n})
\end{equation*}
by (\ref{6.3}).
Similarly, we have
\begin{equation}\label{6.59a}
\begin{split}
&I_2-\langle I_2, \widehat{\Theta(t)}\rangle\widehat{\Theta(t)}\\
 &=
\int_{B_{\epsilon^{-1}}(0)}\Big(\frac{2\overline{z}}{1+|z|^2+\sqrt{-1}\tau},0\Big)df(\widehat{\Theta(t)})(\epsilon z,\epsilon^2\tau)\frac{4^{n+1}dzd\tau}{(\tau^2+(1+|z|^2)^2)^{n+1}}\\
&=
\int_{B_{\epsilon^{-1}}(0)}\big(2\overline{z}_1(1+|z|^2-\sqrt{-1}\tau),\cdots,
2\overline{z}_n(1+|z|^2-\sqrt{-1}\tau),0\big)\\
&\hspace{4mm}\cdot\left(\sum_{i=1}^n\frac{\partial
f(\widehat{\Theta(t)})}{\partial a_i}\epsilon
a_i+\sum_{i=1}^n\frac{\partial f(\widehat{\Theta(t)})}{\partial
b_i}\epsilon b_i+\frac{\partial f(\widehat{\Theta(t)})}{\partial
\tau}\epsilon^2\tau\right)\frac{4^{n+1}dzd\tau}{(\tau^2+(1+|z|^2)^2)^{n+2}}\\
&= 2\epsilon\int_{B_{\epsilon^{-1}}(0)}\left(\frac{\partial
f(\widehat{\Theta(t)})}{\partial a_1} a_1^2-\sqrt{-1}\frac{\partial
f(\widehat{\Theta(t)})}{\partial b_1} b_1^2,\cdots,\frac{\partial
f(\widehat{\Theta(t)})}{\partial a_n} a_n^2-\sqrt{-1}\frac{\partial
f(\widehat{\Theta(t)})}{\partial b_n} b_n^2,0\right)\\
&\hspace{4mm}\cdot
\frac{4^{n+1}(1+|z|^2)}{(\tau^2+(1+|z|^2)^2)^{n+2}}dzd\tau\\
&=\epsilon A_1 \left(\frac{\partial
f(\widehat{\Theta(t)})}{\partial a_1} -\sqrt{-1}\frac{\partial
f(\widehat{\Theta(t)})}{\partial b_1} ,\cdots,\frac{\partial
f(\widehat{\Theta(t)})}{\partial a_n} -\sqrt{-1}\frac{\partial
f(\widehat{\Theta(t)})}{\partial b_n} ,0\right)+O(\epsilon^{2n+1}),
\end{split}
\end{equation}
where $A_1$ is given in (\ref{6.60}). By (\ref{6.53}) and symmetry, we have
\begin{equation}\label{6.61}
\begin{split}
&II_1-\langle II_1, \widehat{\Theta(t)}\rangle\widehat{\Theta(t)}\\
 &=
\frac{1}{2}\int_{B_{\epsilon^{-1}}(0)}\Big(\frac{2z}{1+|z|^2-\sqrt{-1}\tau},0\Big)(\nabla df)(\widehat{\Theta(t)})((\epsilon z,\epsilon^2\tau),
(\epsilon z,\epsilon^2\tau))\frac{4^{n+1}dzd\tau}{(\tau^2+(1+|z|^2)^2)^{n+1}}\\
&=
\int_{B_{\epsilon^{-1}}(0)}\big(2z_1(1+|z|^2+\sqrt{-1}\tau),\cdots,
2z_n(1+|z|^2+\sqrt{-1}\tau),0\big)\\
&\hspace{4mm}\cdot\frac{1}{2}\left(\sum_{i,j=1}^n\frac{\partial^2
f(\widehat{\Theta(t)})}{\partial a_i\partial a_j}\epsilon^2
a_ia_j+\sum_{i,j=1}^n\frac{\partial^2 f(\widehat{\Theta(t)})}{\partial
b_i\partial b_j}\epsilon^2 b_ib_j+\sum_{i=1}^n\frac{\partial^2
f(\widehat{\Theta(t)})}{\partial a_i\partial
\tau}\epsilon^3a_i\tau\right.\\
&\hspace{8mm}\left.+\sum_{i=1}^n\frac{\partial^2
f(\widehat{\Theta(t)})}{\partial b_i\partial
\tau}\epsilon^3b_i\tau+\frac{\partial^2
f(\widehat{\Theta(t)})}{\partial
\tau^2}\epsilon^4\tau^2\right)\frac{4^{n+1}dzd\tau}{(\tau^2+(1+|z|^2)^2)^{n+1}}\\
&= 0.
\end{split}
\end{equation}
Similarly, we have
\begin{equation}\label{6.61a}
II_2-\langle II_2, \widehat{\Theta(t)}\rangle\widehat{\Theta(t)}=0.
\end{equation}
Using  (\ref{6.50})-(\ref{6.61a}), we can conclude that
\begin{equation}\label{6.62}
\begin{split}
&b-\langle b,(\widehat{\Theta(t)},\widehat{\Theta(t)})\rangle(\widehat{\Theta(t)},\widehat{\Theta(t)})\\
&=\epsilon A_1 \left(\frac{\partial
f(\widehat{\Theta(t)})}{\partial a_1} +\sqrt{-1}\frac{\partial
f(\widehat{\Theta(t)})}{\partial b_1} ,\cdots,\frac{\partial
f(\widehat{\Theta(t)})}{\partial a_n} +\sqrt{-1}\frac{\partial
f(\widehat{\Theta(t)})}{\partial b_n} ,0,\right.\\
&\hspace{16mm}\left.\frac{\partial
f(\widehat{\Theta(t)})}{\partial a_1} -\sqrt{-1}\frac{\partial
f(\widehat{\Theta(t)})}{\partial b_1} ,\cdots,\frac{\partial
f(\widehat{\Theta(t)})}{\partial a_n} -\sqrt{-1}\frac{\partial
f(\widehat{\Theta(t)})}{\partial b_n} ,0\right)+E_4
\end{split}
\end{equation}
 with errors
\begin{equation}\label{6.63}
\begin{split}
|E_4|&\leq C\epsilon^{3}+C\|v-1\|_{S_1^2(S^{2n+1},\theta_0)}
\big(F_2(t)^{\frac{1}{2}}+\|f-f(\widehat{\Theta(t)})\|_{L^2(S^{2n+1},\theta_0)}\big)\\
&\leq C\epsilon^{3}+
CF_2(t)+C\|f-f(\widehat{\Theta(t)})\|_{L^2(S^{2n+1},\theta_0)}^2\\
&\leq
C|\nabla_{\theta_0}f(\widehat{\Theta(t)})|^2_{\theta_0}\epsilon^2+C|b|^2+C\epsilon^3
\end{split}
\end{equation}
where the second inequality follows from Lemma \ref{lem6.6}, and the
third inequality follows from (\ref{6.55}) and Lemma \ref{lem6.5}.\\
\textit{Case (ii).} By
(\ref{6.61}) and symmetry, we have
\begin{equation}\label{6.64}
\begin{split}
\langle I_1, \widehat{\Theta(t)}\rangle
 &=
\int_{B_{\epsilon^{-1}}(0)}\frac{1-|z|^2+\sqrt{-1}\tau}{1+|z|^2-\sqrt{-1}\tau}df(\widehat{\Theta(t)})(\epsilon z,\epsilon^2\tau)\frac{4^{n+1}dzd\tau}{(\tau^2+(1+|z|^2)^2)^{n+1}}\\
&=
\int_{B_{\epsilon^{-1}}(0)}(1-|z|^2+\sqrt{-1}\tau)(1+|z|^2+\sqrt{-1}\tau)\\
&\hspace{4mm}\cdot\left(\sum_{i=1}^n\frac{\partial
f(\widehat{\Theta(t)})}{\partial a_i}\epsilon
a_i+\sum_{i=1}^n\frac{\partial f(\widehat{\Theta(t)})}{\partial
b_i}\epsilon b_i+\frac{\partial f(\widehat{\Theta(t)})}{\partial
\tau}\epsilon^2\tau\right)\frac{4^{n+1}dzd\tau}{(\tau^2+(1+|z|^2)^2)^{n+2}}\\
&= 0.
\end{split}
\end{equation}
Similarly, we have
\begin{equation}\label{6.64a}
\langle I_2, \widehat{\Theta(t)}\rangle=0.
\end{equation}
On the other hand, by
(\ref{6.61}) and symmetry, we have
\begin{equation}\label{6.65}
\begin{split}
\langle II_1, \widehat{\Theta(t)}\rangle
 &=
\frac{1}{2}\int_{B_{\epsilon^{-1}}(0)}\frac{1-|z|^2+\sqrt{-1}\tau}{1+|z|^2-\sqrt{-1}\tau}(\nabla df)(\widehat{\Theta(t)})((\epsilon z,\epsilon^2\tau),
(\epsilon z,\epsilon^2\tau))\frac{4^{n+1}dzd\tau}{(\tau^2+(1+|z|^2)^2)^{n+1}}\\
&=\frac{1}{2}\int_{B_{\epsilon^{-1}}(0)}(1-|z|^2+\sqrt{-1}\tau)(1+|z|^2+\sqrt{-1}\tau)\\
&\hspace{4mm}\cdot\frac{1}{2}\left(\sum_{i,j=1}^n\frac{\partial^2
f(\widehat{\Theta(t)})}{\partial a_i\partial a_j}\epsilon^2
a_ia_j+\sum_{i,j=1}^n\frac{\partial^2 f(\widehat{\Theta(t)})}{\partial
b_i\partial b_j}\epsilon^2 b_ib_j+\sum_{i=1}^n\frac{\partial^2
f(\widehat{\Theta(t)})}{\partial a_i\partial
\tau}\epsilon^3a_i\tau\right.\\
&\hspace{4mm}\left.+\sum_{i=1}^n\frac{\partial^2
f(\widehat{\Theta(t)})}{\partial b_i\partial
\tau}\epsilon^3b_i\tau+\frac{\partial^2
f(\widehat{\Theta(t)})}{\partial
\tau^2}\epsilon^4\tau^2\right)\frac{4^{n+1}dzd\tau}{(\tau^2+(1+|z|^2)^2)^{n+2}}\\
&=\frac{1}{2}\int_{B_{\epsilon^{-1}}(0)}(1-|z|^4-\tau^2)\cdot\frac{1}{2}\left(\sum_{i=1}^n\frac{\partial^2
f(\widehat{\Theta(t)})}{\partial a_i^2}\epsilon^2
a_i^2\right.\\
&\hspace{4mm}\left.+\sum_{i=1}^n\frac{\partial^2
f(\widehat{\Theta(t)})}{\partial b_i^2}\epsilon^2
b_i^2+\frac{\partial^2 f(\widehat{\Theta(t)})}{\partial
\tau^2}\epsilon^4\tau^2\right)\frac{4^{n+1}dzd\tau}{(\tau^2+(1+|z|^2)^2)^{n+2}}\\
&=\frac{1}{2n}\epsilon^2\Delta_{\theta_0}f(\widehat{\Theta(t)})\int_{\mathbb{H}^n}\frac{4^{n}(1-|z|^4-\tau^2)|z|^2dzd\tau}{(\tau^2+(1+|z|^2)^2)^{n+2}}\\
&\hspace{2mm}+C\epsilon^4\int_{B_{\epsilon^{-1}}(0)}\frac{4^{n+1}(1-|z|^4-\tau^2)\tau^2dzd\tau}{(\tau^2+(1+|z|^2)^2)^{n+2}}+O(\epsilon^{2n})\\
&=-\epsilon^2A_2\Delta_{\theta_0}f(\widehat{\Theta(t)})+O(\epsilon^{4})
\end{split}
\end{equation}
where
\begin{equation}\label{6.66}
A_2:=\frac{1}{2n}\int_{\mathbb{H}^n}\frac{4^{n}(|z|^4+\tau^2-1)|z|^2}{(\tau^2+(1+|z|^2)^2)^{n+2}}dzd\tau
\end{equation}
since
\begin{equation*}
\begin{split}
\left|\int_{\mathbb{H}^n\setminus B_{\epsilon^{-1}}(0)}\frac{(|z|^4+\tau^2-1)|z|^2}{(\tau^2+(1+|z|^2)^2)^{n+2}}dzd\tau\right|
&\leq\int_{\mathbb{H}^n\setminus B_{\epsilon^{-1}}(0)}\frac{|z|^2}{(\tau^2+(1+|z|^2)^2)^{n+1}}dzd\tau\\
&\leq C\int_0^\infty\frac{d\tau}{1+\tau^2}\int_{\epsilon^{-1}}^\infty\frac{r^{2n+1}}{(1+r^2)^{2n}}dr
=O(\epsilon^{2n-2}).
\end{split}
\end{equation*}
Note that $A_2$ is positive because
\begin{equation*}
\begin{split}
&\int_{\mathbb{H}^n}\frac{(|z|^4+\tau^2-1)|z|^2}{(\tau^2+(1+|z|^2)^2)^{n+2}}dzd\tau
=C\int_{\{r^4+\tau^2\geq 0\}}\frac{(r^4+\tau^2-1)r^{2n+1}}{(\tau^2+(1+r^2)^2)^{n+2}}drd\tau\\
&=C\int_{\{r^2+\tau^2\geq 0, r\geq 0\}}\frac{(r^2+\tau^2-1)r^{n}}{(r^2+\tau^2+2r+1)^{n+2}}drd\tau
\geq \frac{C}{2^{n+2}}\int_{\{r^2+\tau^2\geq 0, r\geq 0\}}\frac{(r^2+\tau^2-1)r^{n}}{(r^2+\tau^2+1)^{n+2}}drd\tau\\
&=\frac{C}{2^{n+2}}\int_0^\pi\int_{0}^\infty\frac{(r^2-1)(r\sin\theta)^n}{(r^2+1)^{n+2}}rdrd\theta
=\frac{C}{2^{n+2}}\int_0^\pi\sin^n\theta d\theta\int_{0}^\infty\frac{(r^2-1)r^{n+1}}{(r^2+1)^{n+2}}dr,
\end{split}
\end{equation*}
where we have used the change of variables $r^2\mapsto r$ in the second equality, and we have
changed the coordinates $(r,\tau)$ to the polar coordinates $(r,\theta)$ in the third equality.
To see that the last term is positive, we note that
\begin{equation*}
\begin{split}
\int_{0}^\infty\frac{(r^2-1)r^{n+1}}{(r^2+1)^{n+2}}dr
&=\int_{1}^\infty\frac{(r^2-1)r^{n+1}}{(r^2+1)^{n+2}}dr+\int_{0}^1\frac{(r^2-1)r^{n+1}}{(r^2+1)^{n+2}}dr\\
&=\int_{1}^\infty\frac{(r^2-1)r^{n+1}}{(r^2+1)^{n+2}}dr
+\int_{\infty}^1\frac{(\frac{1}{t^2}-1)\frac{1}{t^{n+1}}}{(\frac{1}{t^2}+1)^{n+2}}\left(-\frac{1}{t^2}dt\right)\\
&=\int_{1}^\infty\frac{(r^2-1)r^{n+1}}{(r^2+1)^{n+2}}dr+\int^{\infty}_1\frac{(1-t^2)t^{n-1}}{(t^2+1)^{n+2}}dt\\
&=\int_{1}^\infty\frac{(r^2-1)(r^{n+1}-r^{n-1})}{(r^2+1)^{n+2}}dr=\int_{1}^\infty\frac{(r^2-1)^2r^{n-1}}{(r^2+1)^{n+2}}dr>0.
\end{split}
\end{equation*}
Similarly, we have
\begin{equation}\label{6.65a}
\begin{split}
\langle II_2, \widehat{\Theta(t)}\rangle
 &=
\frac{1}{2}\int_{B_{\epsilon^{-1}}(0)}\frac{1-|z|^2-\sqrt{-1}\tau}{1+|z|^2+\sqrt{-1}\tau}(\nabla df)(\widehat{\Theta(t)})((\epsilon z,\epsilon^2\tau),
(\epsilon z,\epsilon^2\tau))\frac{4^{n+1}dzd\tau}{(\tau^2+(1+|z|^2)^2)^{n+1}}\\
&=\frac{1}{2}\int_{B_{\epsilon^{-1}}(0)}(1-|z|^2-\sqrt{-1}\tau)(1+|z|^2-\sqrt{-1}\tau)\\
&\hspace{4mm}\cdot\frac{1}{2}\left(\sum_{i,j=1}^n\frac{\partial^2
f(\widehat{\Theta(t)})}{\partial a_i\partial a_j}\epsilon^2
a_ia_j+\sum_{i,j=1}^n\frac{\partial^2 f(\widehat{\Theta(t)})}{\partial
b_i\partial b_j}\epsilon^2 b_ib_j+\sum_{i=1}^n\frac{\partial^2
f(\widehat{\Theta(t)})}{\partial a_i\partial
\tau}\epsilon^3a_i\tau\right.\\
&\hspace{4mm}\left.+\sum_{i=1}^n\frac{\partial^2
f(\widehat{\Theta(t)})}{\partial b_i\partial
\tau}\epsilon^3b_i\tau+\frac{\partial^2
f(\widehat{\Theta(t)})}{\partial
\tau^2}\epsilon^4\tau^2\right)\frac{4^{n+1}dzd\tau}{(\tau^2+(1+|z|^2)^2)^{n+3}}\\
&=\frac{1}{2}\int_{B_{\epsilon^{-1}}(0)}(1-|z|^4-\tau^2)\cdot\frac{1}{2}\left(\sum_{i=1}^n\frac{\partial^2
f(\widehat{\Theta(t)})}{\partial a_i^2}\epsilon^2
a_i^2\right.\\
&\hspace{4mm}\left.+\sum_{i=1}^n\frac{\partial^2
f(\widehat{\Theta(t)})}{\partial b_i^2}\epsilon^2
b_i^2+\frac{\partial^2 f(\widehat{\Theta(t)})}{\partial
\tau^2}\epsilon^4\tau^2\right)\frac{4^{n+1}dzd\tau}{(\tau^2+(1+|z|^2)^2)^{n+3}}\\
&=\frac{1}{2n}\epsilon^2\Delta_{\theta_0}f(\widehat{\Theta(t)})\int_{\mathbb{H}^n}\frac{4^{n}(1-|z|^4-\tau^2)|z|^2dzd\tau}{(\tau^2+(1+|z|^2)^2)^{n+3}}\\
&\hspace{2mm}+C\epsilon^4\int_{B_{\epsilon^{-1}}(0)}\frac{4^{n+1}(1-|z|^4-\tau^2)\tau^2dzd\tau}{(\tau^2+(1+|z|^2)^2)^{n+3}}+O(\epsilon^{2n+2})\\
&=-\epsilon^2A_2\Delta_{\theta_0}f(\widehat{\Theta(t)})+O(\epsilon^{4})
\end{split}
\end{equation}
where $A_2$ is given in (\ref{6.66}). Using (\ref{6.49})-(\ref{6.58}), (\ref{6.64})-(\ref{6.65a}), Lemma
\ref{lem6.5} and \ref{lem6.6}, we have
\begin{equation}\label{6.67}
\langle b,
(\widehat{\Theta(t)},\widehat{\Theta(t)})\rangle=-2A_2\alpha\Delta_{\theta_0}f(\widehat{\Theta(t)})\epsilon^2+E_5,
\end{equation}
with error
\begin{equation}\label{6.68}
|E_5|\leq
C|\nabla_{\theta_0}f(\widehat{\Theta(t)})|^2_{\theta_0}\epsilon^2+C|b|^2+C\epsilon^3.
\end{equation}

Now we have the estimate
\begin{equation}\label{6.69}
\begin{split}
|b|^2&=|b-\langle
b,(\widehat{\Theta(t)},\widehat{\Theta(t)})\rangle(\widehat{\Theta(t)},\widehat{\Theta(t)})|^2+C\langle
b,(\widehat{\Theta(t)},\widehat{\Theta(t)})\rangle^2\\
&\leq
C|\nabla_{\theta_0}f(\widehat{\Theta(t)})|^2_{\theta_0}\epsilon^2
+O(\epsilon^4)(1+|\Delta_{\theta_0}f(\widehat{\Theta(t)})|^2)+C|b|^4
\end{split}
\end{equation}
in view of (\ref{6.62}), (\ref{6.63}), (\ref{6.67}) and (\ref{6.68}). By Lemma \ref{lem3.2} and
\ref{lem6.5}, $|b|^2\rightarrow 0$ as $t\rightarrow\infty$. Hence,
we can choose sufficiently large $t$ such that $C|b|^2\leq 1/2$ so
that the last term on the right hand side of (\ref{6.69}) can be
absorbed to the right hand side, i.e.
\begin{equation}\label{6.70}
|b|^2\leq
C|\nabla_{\theta_0}f(\widehat{\Theta(t)})|^2_{\theta_0}\epsilon^2
+O(\epsilon^4)(1+|\Delta_{\theta_0}f(\widehat{\Theta(t)})|^2).
\end{equation}
 Now the assertion follows from
(\ref{6.62}), (\ref{6.63}), (\ref{6.67}), (\ref{6.68}) and
(\ref{6.70}).
\end{proof}

From
(\ref{6.55}), (\ref{6.70}), Lemma \ref{lem6.5}, and Lemma \ref{lem6.6}, we obtain
\begin{equation}\label{6.71}
F_2(t)=C|\nabla_{\theta_0}f(\widehat{\Theta(t)})|^2_{\theta_0}\epsilon^2+O(\epsilon^4)\mbox{
and }\|v-1\|_{S_1^2(S^{2n+1},\theta_0)}^2\leq
C|\nabla_{\theta_0}f(\widehat{\Theta(t)})|^2_{\theta_0}\epsilon^2+O(\epsilon^3).
\end{equation}

\begin{lem}\label{lem6.8}
With $O(1)\leq C$ as $t\rightarrow\infty$, there holds
\begin{equation*}
\begin{split}
b&=\frac{\epsilon\mbox{\emph{Vol}}(S^{2n+1},\theta_0)}{n+1}
\left(\frac{dz_1}{dt},\cdots,
\frac{dz_n}{dt},
-\frac{1}{2}\frac{dr}{dt}+\sqrt{-1}\epsilon Im\Big(\frac{dz(t)}{dt}\cdot
\overline{z(t)}\Big)+\sqrt{-1}\frac{\epsilon}{2}\frac{d\tau}{dt},\right.\\
&\left.\hspace{3.3cm}
\frac{d\overline{z}_1}{dt},\cdots,
\frac{d\overline{z}_n}{dt},
-\frac{1}{2}\frac{dr}{dt}-\sqrt{-1}\epsilon Im\Big(\frac{dz(t)}{dt}\cdot
\overline{z(t)}\Big)-\sqrt{-1}\frac{\epsilon}{2}\frac{d\tau}{dt}\right)\\
&\hspace{3cm}+O(|\nabla_{\theta_0}f(\widehat{\Theta(t)})|_{\theta_0}^2)\epsilon^2+O(\epsilon^3).
\end{split}
\end{equation*}
\end{lem}
\begin{proof}
Using (\ref{5.b}), we have
\begin{equation}\label{6.72}
(n+1)b=(n+1)\int_{S^{2n+1}}(x,\overline{x})(\alpha(t)f_\phi-R_h)dV_h=\int_{S^{2n+1}}(\xi,\overline{\xi})
dV_h=(X,\overline{X})+I,
\end{equation}
where $X$ is the vector given in (\ref{5.2})
and
$$I=\int_{S^{2n+1}}(\xi,\overline{\xi}) (v^{2+\frac{2}{n}}-1)dV_{\theta_0}$$
which can be estimated as follows:
\begin{equation}\label{6.73}
\begin{split}
|I|&\leq
C\|\xi\|_{L^\infty}\|v^{2+\frac{2}{n}}-1\|_{S_1^2(S^{2n+1},\theta_0)}\leq
CF_2^{\frac{1}{2}}\|v-1\|_{S_1^2(S^{2n+1},\theta_0)}\\
&\leq C(F_2+\|v-1\|_{S_1^2(S^{2n+1},\theta_0)}^2)\leq
C|\nabla_{\theta_0}f(\widehat{\Theta(t)})|^2_{\theta_0}\epsilon^2+O(\epsilon^3)
\end{split}
\end{equation}
in view of (\ref{6.71}), Lemma \ref{lem4.14}, and Lemma \ref{lem5.1}. Now
the assertion follows from (\ref{5.18}), (\ref{5.19}), (\ref{6.72}),
and (\ref{6.73}).
\end{proof}

\begin{lem}\label{lem6.9}
With $o(1)\rightarrow 0$ as $t\rightarrow\infty$, there holds
$$\mbox{\emph{Vol}}(S^{2n+1},\theta_0)^2-|\Theta(t)|^2=\big(4\mbox{\emph{Vol}}(S^{2n+1},\theta_0)A_3+o(1)\big)\epsilon^2,$$
where $A_3$ is the positive number defined as in
$(\ref{6.77})$.
\end{lem}
\begin{proof}
As before, we choose coordinates of $\mathbb{H}^n$ such that
$(0,...,0,1)=\widehat{\Theta(t)}$ corresponding to the point
$\widehat{\Theta(t)}$. Hence $\phi(t)$ will have the usual
representation. In particular,
\begin{equation}\label{6.74}
\phi_{n+1}=\displaystyle\frac{1-\epsilon^2|z|^2+\sqrt{-1}\epsilon^2\tau}{1+\epsilon^2|z|^2-\sqrt{-1}\epsilon^2\tau}
=1-2\epsilon^2\frac{|z|^2-\sqrt{-1}\tau}{1+\epsilon^2|z|^2-\sqrt{-1}\epsilon^2\tau}.
\end{equation}
Thus by symmetry we have
\begin{equation}\label{6.75}
\begin{split}
\Theta(t)_{n+1}&=\int_{S^{2n+1}}\phi_{n+1}(t)dV_{\theta_0}\\
&=\mbox{Vol}(S^{2n+1},\theta_0)-2\epsilon^2\int_{\mathbb{H}^n}\frac{|z|^2-\sqrt{-1}\tau}{1+\epsilon^2|z|^2-\sqrt{-1}\epsilon^2\tau}
\frac{dzd\tau}{(\tau^2+(1+|z|^2)^2)^{n+1}}\\
&=\mbox{Vol}(S^{2n+1},\theta_0)-2\epsilon^2
\int_{\mathbb{H}^n}\frac{\epsilon^2(|z|^4+\tau^2)+|z|^2}{(1+\epsilon^2|z|^2)^2+\epsilon^4\tau^2}
\frac{dzd\tau}{(\tau^2+(1+|z|^2)^2)^{n+1}}.
\end{split}
\end{equation}
Observe that
\begin{equation}\label{6.76}
\begin{split}
&\int_{\mathbb{H}^n}\frac{\epsilon^2(|z|^4+\tau^2)+|z|^2}{(1+\epsilon^2|z|^2)^2+\epsilon^4\tau^2}
\frac{dzd\tau}{(\tau^2+(1+|z|^2)^2)^{n+1}}\\
&=\int_{\mathbb{H}^n}
\frac{|z|^2dzd\tau}{(\tau^2+(1+|z|^2)^2)^{n+1}}
\\
&\hspace{4mm}-\int_{\mathbb{H}^n}\frac{\epsilon^2(|z|^4-\tau^2)+\epsilon^4(|z|^6+|z|^2\tau^2)}{(1+\epsilon^2|z|^2)^2+\epsilon^4\tau^2}
\frac{dzd\tau}{(\tau^2+(1+|z|^2)^2)^{n+1}}.
\end{split}
\end{equation}
We are going to estimate the terms on the right hand side of (\ref{6.76}). First note that
\begin{equation}\label{6.77}
\begin{split}
A_3&:=\int_{\mathbb{H}^n}
\frac{|z|^2dzd\tau}{(\tau^2+(1+|z|^2)^2)^{n+1}}
\leq\int_{\{0\leq|z|<\infty\}}\left(\int^{\infty}_{-\infty}\frac{d\tau}{1+\tau^2}\right)
\frac{|z|^2dz}{(1+|z|^2)^{2n}}\\
&=\pi\int_{\{0\leq|z|<\infty\}}\frac{|z|^2dz}{(1+|z|^2)^{2n}}=C\int_0^{\infty}\frac{r^{2n+1}dr}{(1+r^2)^{2n}}\\
&\leq C\int_0^{1}\frac{r^{2n+1}dr}{(1+r^2)^{2n}}+C\int_1^{\infty}\frac{dr}{r^{2n-1}}<\infty
\end{split}
\end{equation}
when $n\geq 2$.
On the other hand, we have
\begin{equation}\label{6.78}
\begin{split}
&\int_{\mathbb{H}^n}\frac{\epsilon^2(|z|^4-\tau^2)+\epsilon^4(|z|^6+|z|^2\tau^2)}{(1+\epsilon^2|z|^2)^2+\epsilon^4\tau^2}
\frac{dzd\tau}{(\tau^2+(1+|z|^2)^2)^{n+1}}\\
&= C\int_0^\infty\int_0^\infty
\frac{\epsilon^2(r^4-\tau^2)+\epsilon^4(r^6+r^2\tau^2)}{(1+\epsilon^2r^2)^2+\epsilon^4\tau^2}
\frac{r^{2n-1}drd\tau}{(\tau^2+(1+r^2)^2)^{n+1}}\\
&=C\int_0^\infty\int_0^\infty
\frac{\epsilon^2r^4+\epsilon^4r^6}{(1+\epsilon^2r^2)^2+\epsilon^4\tau^2}
\frac{r^{2n-1}drd\tau}{(\tau^2+(1+r^2)^2)^{n+1}}\\
&\hspace{4mm}-C\int_0^\infty\int_0^\infty
\frac{\epsilon^2\tau^2-\epsilon^4r^2\tau^2}{(1+\epsilon^2r^2)^2+\epsilon^4\tau^2}
\frac{r^{2n-1}drd\tau}{(\tau^2+(1+r^2)^2)^{n+1}}
\end{split}
\end{equation}
for some constant $C$ depending only on $n$. Note that
\begin{equation}\label{6.79}
\begin{split}
&\int_0^\infty\int_0^\infty
\frac{\epsilon^2r^4}{(1+\epsilon^2|z|^2)^2+\epsilon^4\tau^2}
\frac{r^{2n-1}drd\tau}{(\tau^2+(1+r^2)^2)^{n+1}}\\
&\leq \int_0^\infty\left(\int_0^1+\int_1^{\frac{1}{\epsilon^2}}+\int_{\frac{1}{\epsilon^2}}^\infty\right)
\frac{\epsilon^2r^{2n+3}}{(1+\epsilon^4\tau^2)(\tau^2+(1+r^2)^2)^{n+1}}d\tau dr\\
&\leq \int_0^\infty\left(\int_0^1 \frac{1}{(1+r^2)^{2n+2}}d\tau
+\int_1^{\frac{1}{\epsilon^2}}\frac{1}{2\tau(1+r^2)^{2n+1}}d\tau\right.\\
&\hspace{8mm}\left.
+\int_{\frac{1}{\epsilon^2}}^\infty\frac{1}{2\epsilon^4\tau^3(1+r^2)^{2n+1}}d\tau\right)\epsilon^2r^{2n+3}dr\\
&= \int_0^\infty\left( \frac{1}{(1+r^2)^{2n+2}}
+\frac{-\log\epsilon}{(1+r^2)^{2n+1}}+\frac{1}{4(1+r^2)^{2n+1}}\right)\epsilon^2r^{2n+3}dr\\
&=\epsilon^2\int_0^\infty\frac{r^{2n+3}dr}{(1+r^2)^{2n+2}}
-\epsilon^2\log\epsilon\int_0^\infty\frac{r^{2n+3}dr}{(1+r^2)^{2n+1}}+\frac{\epsilon^2}{4}\int_0^\infty\frac{r^{2n+3}dr}{(1+r^2)^{2n+1}}
\leq C\epsilon,
\end{split}
\end{equation}
where we have used $\tau^2+(1+r^2)^2\geq 2\tau(1+r^2)$  in the
second inequality, and the last inequality follows from
$-\epsilon\log\epsilon\leq C$ for $\epsilon$ being sufficiently
small and
\begin{equation}\label{6.80}
\int_0^\infty\frac{r^{k}dr}{(1+r^2)^{l}}\leq
\int_0^1r^{k}dr+\int_1^\infty\frac{dr}{r^{2l-k}}=\frac{1}{k+1}+\frac{1}{2l-k-1}
\mbox{ if }2l-k\geq 2\mbox{ and }k+1>0.
\end{equation}
Similarly, we can estimate
\begin{equation}\label{6.81}
\begin{split}
&\int_0^\infty\int_0^\infty
\frac{\epsilon^4r^6}{(1+\epsilon^2r^2)^2+\epsilon^4\tau^2}
\frac{r^{2n-1}drd\tau}{(\tau^2+(1+r^2)^2)^{n+1}}\\
&\leq \int_0^\infty\left(\int_0^{\frac{1}{\epsilon^2}}+\int_{\frac{1}{\epsilon^2}}^\infty\right)
\frac{\epsilon^4r^{2n+5}}{(1+\epsilon^4\tau^2)(\tau^2+(1+r^2)^2)^{n+1}}d\tau dr\\
&\leq \int_0^\infty\left(\int_0^{\frac{1}{\epsilon^2}} d\tau
+\int_{\frac{1}{\epsilon^2}}^\infty\frac{1}{\epsilon^4\tau^2}d\tau\right)\frac{\epsilon^4r^{2n+5}}{(1+r^2)^{2n+2}}dr\\
&=2\epsilon^2\int_0^\infty\frac{r^{2n+5}dr}{(1+r^2)^{2n+2}} \leq
C\epsilon,
\end{split}
\end{equation}
where the last inequality follows from (\ref{6.80}). Note also that
\begin{equation}\label{6.82}
\begin{split}
&\int_0^\infty\int_0^\infty
\frac{\epsilon^2\tau^2}{(1+\epsilon^2r^2)^2+\epsilon^4\tau^2}
\frac{r^{2n-1}drd\tau}{(\tau^2+(1+r^2)^2)^{n+1}}\\
&\leq
\int_0^\infty\left(\int_0^{\frac{1}{\epsilon^{3/2}}}+\int_{\frac{1}{\epsilon^{3/2}}}^\infty\right)
\frac{\epsilon^2r^{2n-1}}{(1+\epsilon^4\tau^2)(\tau^2+(1+r^2)^2)^{n}}d\tau dr\\
&\leq \int_0^\infty\left(\int_0^{\frac{1}{\epsilon^{3/2}}}
\frac{1}{(1+r^2)^{2n}}d\tau
+\int_{\frac{1}{\epsilon^{3/2}}}^\infty\frac{1}{\epsilon^4\tau^3(1+r^2)^{2n-1}}d\tau\right)\epsilon^2r^{2n-1}dr\\
&= \epsilon^{1/2}\int_0^\infty \frac{r^{2n-1}dr}{(1+r^2)^{2n}}
+\frac{\epsilon}{2}\int_0^\infty \frac{
r^{2n-1}dr}{(1+r^2)^{2n-1}}\leq C \epsilon^{1/2}
\end{split}
\end{equation}
where we have used $\tau^2+(1+r^2)^2\geq 2\tau(1+r^2)$ in the second
inequality, and the last inequality follows from (\ref{6.80}). On
the other hand, we can estimate
\begin{equation}\label{6.83}
\begin{split}
&\int_0^\infty\int_0^\infty
\frac{\epsilon^4r^2\tau^2}{(1+\epsilon^2r^2)^2+\epsilon^4\tau^2}
\frac{r^{2n-1}drd\tau}{(\tau^2+(1+r^2)^2)^{n+1}}\\
&\leq
\int_0^\infty\left(\int_0^{\frac{1}{\epsilon^2}}+\int_{\frac{1}{\epsilon^2}}^\infty\right)
\frac{\epsilon^4r^{2n+1}}{(1+\epsilon^4\tau^2)(1+r^2)^{2n}}d\tau dr\\
&\leq \int_0^\infty\left(\int_0^{\frac{1}{\epsilon^2}} d\tau
+\int_{\frac{1}{\epsilon^2}}^\infty\frac{1}{\epsilon^4\tau^2}d\tau\right)\frac{\epsilon^4r^{2n+1}}{(1+r^2)^{2n}}dr\\
&=2\epsilon^2\int_0^\infty\frac{r^{2n+1}dr}{(1+r^2)^{2n}} \leq
C\epsilon,
\end{split}
\end{equation}
where the last inequality follows from (\ref{6.80}). Combining
(\ref{6.75})-(\ref{6.83}), we conclude that
\begin{equation}\label{6.84}
\Theta(t)_{n+1}=\mbox{Vol}(S^{2n+1},\theta_0)-2A_3\epsilon^2+o(\epsilon^2).
\end{equation}
From this, we have
\begin{equation*}
\begin{split}
\mbox{Vol}(S^{2n+1},\theta_0)^2-|\Theta(t)|^2
&=(\mbox{Vol}(S^{2n+1},\theta_0)+\Theta(t)_{n+1})(\mbox{Vol}(S^{2n+1},\theta_0)-\Theta(t)_{n+1})\\
&=(2\mbox{Vol}(S^{2n+1},\theta_0)+o(1))(2A_3\epsilon^2+o(\epsilon^2))\\
&=(4\mbox{Vol}(S^{2n+1},\theta_0)A_3+o(1))\epsilon^2,
\end{split}
\end{equation*}
as required.
\end{proof}

\begin{lem}\label{lem6.10}
With $o(1)\rightarrow 0$ as $t\rightarrow\infty$, there hold
\begin{equation*}
\begin{split}
&\frac{d\Theta(t)_i}{dt}=(A_4+o(1))\epsilon^2\frac{dz_i(t)}{dt}+o(1)\epsilon^2z_i(t)\frac{dr(t)}{dt},\hspace{2mm}1\leq i\leq n;\\
&\frac{d}{dt}(\mbox{\emph{Vol}}(S^{2n+1},\theta_0)^2-|\Theta(t)|^2)
=(\mbox{\emph{Vol}}(S^{2n+1},\theta_0)^2-|\Theta(t)|^2)\left[\left(\frac{A_5}{2A_3}+o(1)\right)\epsilon\frac{dr(t)}{dt}
+o(1)\epsilon\frac{d\tau(t)}{dt}\right],
\end{split}
\end{equation*}
where $A_3$ is the positive constant defined as in $(\ref{6.77})$, $A_4$ and $A_5$ are the positive constants defined as in $(\ref{6.89})$ and $(\ref{6.91})$ respectively.
\end{lem}
\begin{proof}
As we have remarked before, we have
$$\phi(t)=\Psi\circ\delta_{q(t),r(t)}\circ\pi.$$
Differentiating the identity
\begin{equation*}
\Theta(t)=(\Theta(t)_1,...,\Theta(t)_{n+1})=\int_{S^{2n+1}}\phi(t)dV_{\theta_0}
=\int_{\mathbb{H}^n}\Psi\circ\delta_{q(t),r(t)}(z,\tau)
\frac{4^{n+1}dzd\tau}{(\tau^2+(1+|z|^2)^2)^{n+1}}
\end{equation*}
at time $t$, we obtain
\begin{equation}\label{6.85}
\frac{d\Theta(t)}{dt}=\int_{\mathbb{H}^n}d\Psi_y\left(\frac{d}{dt}\delta_{q(t),r(t)}(z,\tau)\right)
\frac{4^{n+1}dzd\tau}{(\tau^2+(1+|z|^2)^2)^{n+1}}
\end{equation}
where $y=\delta_{q(t),r(t)}(z,\tau)$. By (\ref{5.13}), we
have
\begin{equation}\label{6.86}
\begin{split}
&\frac{d}{dt}\delta_{q(t),r(t)}(z,\tau)\\
&=\frac{dr(t)}{dt}
\sum_{j=1}^n\left(a_j\frac{\partial}{\partial a_j}+
b_j\frac{\partial}{\partial
b_j}\right)+2\frac{dr(t)}{dt}\left[r(t)\tau+\sum_{j=1}^n\big(a_jb_j(t)
-b_ja_j(t)\big)\right]\frac{\partial}{\partial\tau}\\
&\hspace{4mm}
+\sum_{j=1}^n\left(\frac{da_j(t)}{dt}\frac{\partial}{\partial
a_j} +\frac{db_j(t)}{dt}\frac{\partial}{\partial b_j}\right)+
\left[2r(t)\sum_{j=1}^n\left(a_j\frac{db_j(t)}{dt}-b_j\frac{da_j(t)}{dt}\right)+\frac{d\tau(t)}{dt}
\right] \frac{\partial}{\partial\tau}.
\end{split}
\end{equation}
Let $\epsilon=1/r(t)$. Combining (\ref{5.7}), (\ref{5.8}),
(\ref{6.85}) and (\ref{6.86}), we obtain by symmetry that
\begin{equation}\label{6.88}
\begin{split}
\frac{d\Theta(t)_{n+1}}{dt}&=\int_{\mathbb{H}^n}\left\{\frac{dr(t)}{dt}
\sum_{j=1}^n\left[a_j\left(\frac{-4\epsilon^3a_j}{(\epsilon^2+|z|^2-\sqrt{-1}\tau)^2}\right)+
b_j\left(\frac{-4\epsilon^3a_j}{(\epsilon^2+|z|^2-\sqrt{-1}\tau)^2}\right)\right]\right.\\
&\hspace{4mm}+2\frac{dr(t)}{dt}\Big[r(t)\tau+\sum_{j=1}^n\big(a_jb_j(t)
-b_ja_j(t)\big)\Big]\frac{2\sqrt{-1}\epsilon^4}{(\epsilon^2+|z|^2-\sqrt{-1}\tau)^2}\\
&\hspace{4mm}
+\sum_{j=1}^n\left[\frac{da_j(t)}{dt}\left(\frac{-4\epsilon^3a_j}{(\epsilon^2+|z|^2-\sqrt{-1}\tau)^2}\right)
+\frac{db_j(t)}{dt}\left(\frac{-4\epsilon^3b_j}{(\epsilon^2+|z|^2-\sqrt{-1}\tau)^2}\right)\right]\\
&\hspace{4mm} \left.+
\left[2r(t)\sum_{j=1}^n\left(a_j\frac{db_j(t)}{dt}-b_j\frac{da_j(t)}{dt}\right)+\frac{d\tau(t)}{dt}
\right]\frac{2\sqrt{-1}\epsilon^4}{(\epsilon^2+|z|^2-\sqrt{-1}\tau)^2}\right\}\cdot\frac{4^{n+1}dzd\tau}{(\tau^2+(1+|z|^2)^2)^{n+1}}\\
&=-\epsilon^3\frac{dr(t)}{dt}\int_{\mathbb{H}^n}
\frac{4|z|^2}{(\epsilon^2+|z|^2-\sqrt{-1}\tau)^2}\cdot\frac{4^{n+1}dzd\tau}{(\tau^2+(1+|z|^2)^2)^{n+1}}\\
&\hspace{4mm}+\epsilon^4\frac{d\tau(t)}{dt}\int_{\mathbb{H}^n}\frac{2\sqrt{-1}}{(\epsilon^2+|z|^2-\sqrt{-1}\tau)^2}
\cdot\frac{4^{n+1}dzd\tau}{(\tau^2+(1+|z|^2)^2)^{n+1}}\\
&=-4\epsilon^3\frac{dr(t)}{dt}\int_{\mathbb{H}^n}
\frac{|z|^2[(\epsilon^2+|z|^2)^2-\tau^2]}{[(\epsilon^2+|z|^2)^2+\tau^2]^2}\cdot\frac{4^{n+1}dzd\tau}{(\tau^2+(1+|z|^2)^2)^{n+1}}\\
&\hspace{4mm}+2\epsilon^4\sqrt{-1}\frac{d\tau(t)}{dt}\int_{\mathbb{H}^n}
\frac{[(\epsilon^2+|z|^2)^2-\tau^2]}{[(\epsilon^2+|z|^2)^2+\tau^2]^2}\cdot\frac{4^{n+1}dzd\tau}{(\tau^2+(1+|z|^2)^2)^{n+1}}\\
&=-4\epsilon^3\frac{dr(t)}{dt}\int_{\mathbb{H}^n}
\left(\frac{|z|^2}{(\epsilon^2+|z|^2)^2+\tau^2}-\frac{2|z|^2\tau^2}{[(\epsilon^2+|z|^2)^2+\tau^2]^2}\right)\cdot\frac{4^{n+1}dzd\tau}{(\tau^2+(1+|z|^2)^2)^{n+1}}\\
&\hspace{4mm}+2\epsilon^4\sqrt{-1}\frac{d\tau(t)}{dt}\int_{\mathbb{H}^n}
\left(\frac{1}{(\epsilon^2+|z|^2)^2+\tau^2}-\frac{2\tau^2}{[(\epsilon^2+|z|^2)^2+\tau^2]^2}\right)\cdot\frac{4^{n+1}dzd\tau}{(\tau^2+(1+|z|^2)^2)^{n+1}}
\end{split}
\end{equation}
and, for $i=1,..., n,$
\begin{equation}\label{6.87}
\begin{split}
\frac{d\Theta(t)_i}{dt}
&=\int_{\mathbb{H}^n}\left\{\frac{dr(t)}{dt}
\sum_{j=1}^n\left[a_j
\Big(\frac{2\epsilon^2\delta_{ij}}{\epsilon^2+|z|^2-\sqrt{-1}\tau}
-\frac{4\epsilon^2(a_i+\sqrt{-1}b_i)a_j}{(\epsilon^2+|z|^2-\sqrt{-1}\tau)^2}\Big)\right.\right.\\
&\hspace{8mm}\left.+
b_j\Big(\frac{2\epsilon^2\sqrt{-1}\delta_{ij}}{\epsilon^2+|z|^2-\sqrt{-1}\tau}
-\frac{4\epsilon^2(a_i+\sqrt{-1}b_i)b_j}{(\epsilon^2+|z|^2-\sqrt{-1}\tau)^2}\Big)\right]\\
&\hspace{4mm}
+2\frac{dr(t)}{dt}\Big[r(t)\tau+\sum_{j=1}^n\big(a_jb_j(t)
-b_ja_j(t)\big)\Big]\frac{2\sqrt{-1}\epsilon^3(a_i+\sqrt{-1}b_i)}{(\epsilon^2+|z|^2-\sqrt{-1}\tau)^2}\\
&\hspace{4mm}
+\sum_{j=1}^n\left[\frac{da_j(t)}{dt}\Big(\frac{2\epsilon^2\delta_{ij}}{\epsilon^2+|z|^2-\sqrt{-1}\tau}
-\frac{4\epsilon^2(a_i+\sqrt{-1}b_i)a_j}{(\epsilon^2+|z|^2-\sqrt{-1}\tau)^2}\Big)\right.\\
&\hspace{8mm}\left.+\frac{db_j(t)}{dt}\Big(\frac{2\epsilon^2\sqrt{-1}\delta_{ij}}{\epsilon^2l^2+|z|^2-\sqrt{-1}\tau}
-\frac{4\epsilon^2(a_i+\sqrt{-1}b_i)b_j}{(\epsilon^2+|z|^2-\sqrt{-1}\tau)^2}\Big)\right]\\
&\hspace{4mm} \left.+
\left[2r(t)\sum_{j=1}^n\left(a_j\frac{db_j(t)}{dt}-b_j\frac{da_j(t)}{dt}\right)+\frac{d\tau(t)}{dt}
\right]
\frac{2\sqrt{-1}\epsilon^3(a_i+\sqrt{-1}b_i)}{(\epsilon^2+|z|^2-\sqrt{-1}\tau)^2}\right\}
\frac{4^{n+1}dzd\tau}{(\tau^2+(1+|z|^2)^2)^{n+1}}\\
&=4\sqrt{-1}\epsilon^3\frac{dr(t)}{dt}\int_{\mathbb{H}^n}\frac{(a_i^2b_i(t)-\sqrt{-1}b_i^2a_i(t))}{(\epsilon^2+|z|^2-\sqrt{-1}\tau)^2}
\cdot\frac{4^{n+1}dzd\tau}{(\tau^2+(1+|z|^2)^2)^{n+1}}\\
&\hspace{4mm}+2\epsilon^2\left(\frac{da_i(t)}{dt}+\sqrt{-1}\frac{db_i(t)}{dt}\right)\int_{\mathbb{H}^n}
\frac{1}{(\epsilon^2+|z|^2-\sqrt{-1}\tau)}\cdot\frac{4^{n+1}dzd\tau}{(\tau^2+(1+|z|^2)^2)^{n+1}}\\
&\hspace{4mm}-4\epsilon^2\frac{da_i(t)}{dt}\int_{\mathbb{H}^n}
\frac{a_i^2}{(\epsilon^2+|z|^2-\sqrt{-1}\tau)^2}\cdot\frac{4^{n+1}dzd\tau}{(\tau^2+(1+|z|^2)^2)^{n+1}}\\
&\hspace{4mm}-4\epsilon^2\sqrt{-1}\frac{db_i(t)}{dt}\int_{\mathbb{H}^n}
\frac{b_i^2}{(\epsilon^2+|z|^2-\sqrt{-1}\tau)^2}\cdot\frac{4^{n+1}dzd\tau}{(\tau^2+(1+|z|^2)^2)^{n+1}}\\
&\hspace{4mm}+4\sqrt{-1}\epsilon^3r(t)\frac{db_i(t)}{dt}\int_{\mathbb{H}^n}\frac{a_i^2}{(\epsilon^2+|z|^2-\sqrt{-1}\tau)^2}
\cdot\frac{4^{n+1}dzd\tau}{(\tau^2+(1+|z|^2)^2)^{n+1}}\\
&\hspace{4mm}+4\epsilon^3r(t)\frac{da_i(t)}{dt}\int_{\mathbb{H}^n}\frac{b_i^2}{(\epsilon^2+|z|^2-\sqrt{-1}\tau)^2}
\cdot\frac{4^{n+1}dzd\tau}{(\tau^2+(1+|z|^2)^2)^{n+1}}\\
&=2\epsilon^2\frac{dz_i(t)}{dt}\int_{\mathbb{H}^n}
\frac{1}{(\epsilon^2+|z|^2-\sqrt{-1}\tau)}\cdot\frac{4^{n+1}dzd\tau}{(\tau^2+(1+|z|^2)^2)^{n+1}}\\
&\hspace{4mm}+\frac{2\epsilon^3}{n}z_i(t)\frac{dr(t)}{dt}\int_{\mathbb{H}^n}
\frac{|z|^2}{(\epsilon^2+|z|^2-\sqrt{-1}\tau)^2}\cdot\frac{4^{n+1}dzd\tau}{(\tau^2+(1+|z|^2)^2)^{n+1}}
\\
&=2\epsilon^2\frac{dz_i(t)}{dt}\int_{\mathbb{H}^n}
\frac{\epsilon^2+|z|^2}{(\epsilon^2+|z|^2)^2+\tau^2}\cdot\frac{4^{n+1}dzd\tau}{(\tau^2+(1+|z|^2)^2)^{n+1}}\\
&\hspace{4mm}+\frac{2\epsilon^3}{n}z_i(t)\frac{dr(t)}{dt}\int_{\mathbb{H}^n}
\frac{|z|^2[(\epsilon^2+|z|^2)^2-\tau^2]}{[(\epsilon^2+|z|^2)^2+\tau^2]^2}\cdot\frac{4^{n+1}dzd\tau}{(\tau^2+(1+|z|^2)^2)^{n+1}}\\
&=2\epsilon^2\frac{dz_i(t)}{dt}\int_{\mathbb{H}^n}
\frac{\epsilon^2+|z|^2}{(\epsilon^2+|z|^2)^2+\tau^2}\cdot\frac{4^{n+1}dzd\tau}{(\tau^2+(1+|z|^2)^2)^{n+1}}\\
&\hspace{4mm}+\frac{2\epsilon^3}{n}z_i(t)\frac{dr(t)}{dt}\int_{\mathbb{H}^n}
\left(\frac{|z|^2}{(\epsilon^2+|z|^2)^2+\tau^2}-\frac{2|z|^2\tau^2}{[(\epsilon^2+|z|^2)^2+\tau^2]^2}\right)\cdot\frac{4^{n+1}dzd\tau}{(\tau^2+(1+|z|^2)^2)^{n+1}}.
\end{split}
\end{equation}
By
Young's inequality
$$|z|^3|\tau|^{\frac{1}{2}}\leq \frac{3}{4}|z|^4+\frac{1}{4}\tau^2\leq |z|^4+\tau^2,$$
we have
\begin{equation}\label{6.93.5}
\left|\frac{1}{(\epsilon^2+|z|^2)^2+\tau^2}-\frac{1}{|z|^4+\tau^2}\right|
=\frac{2\epsilon^2|z|^2+\epsilon^4}{[(\epsilon^2+|z|^2)^2+\tau^2](|z|^4+\tau^2)}\leq \frac{2}{|z|^4+\tau^2}\leq \frac{2}{|z|^3|\tau|^{\frac{1}{2}}}
\end{equation}
and
\begin{equation}\label{6.93.55}
\begin{split}
\left|\frac{\tau^2}{[(\epsilon^2+|z|^2)^2+\tau^2]^2}-\frac{\tau^2}{(|z|^4+\tau^2)^2}\right|
&\leq\frac{\tau^2}{(\epsilon^2+|z|^2)^2+\tau^2}
\left|\frac{1}{(\epsilon^2+|z|^2)^2+\tau^2}-\frac{1}{|z|^4+\tau^2}\right|\\
&\hspace{4mm}+\frac{\tau^2}{|z|^4+\tau^2}\left|\frac{1}{(\epsilon^2+|z|^2)^2+\tau^2}-\frac{1}{|z|^4+\tau^2}\right|\\
&\leq 2\left|\frac{|z|^2}{(\epsilon^2+|z|^2)^2+\tau^2}-\frac{|z|^2}{|z|^4+\tau^2}\right|\leq \frac{2}{|z|^3|\tau|^{\frac{1}{2}}}.
\end{split}
\end{equation}
We also have
\begin{equation}\label{6.94}
\begin{split}
&\left|\frac{|z|^2}{(\epsilon^2+|z|^2)^2+\tau^2}-\frac{|z|^2}{|z|^4+\tau^2}\right|
=\frac{|z|^2(2\epsilon^2|z|^2+\epsilon^4)}{[(\epsilon^2+|z|^2)^2+\tau^2](|z|^4+\tau^2)}\\
&\leq\frac{2\epsilon^2(|z|^4+2\epsilon^2|z|^2)}{[(\epsilon^2+|z|^2)^2+\tau^2](|z|^4+\tau^2)}\leq \frac{2\epsilon^2}{|z|^4+\tau^2}
\leq\frac{2\epsilon^2}{|z|^3|\tau|^{\frac{1}{2}}}
\end{split}
\end{equation}
which implies that
\begin{equation}\label{6.95}
\begin{split}
&\left|\frac{2|z|^2\tau^2}{[(\epsilon^2+|z|^2)^2+\tau^2]^2}-\frac{2|z|^2\tau^2}{(|z|^4+\tau^2)^2}\right|\\
&\leq\frac{2\tau^2}{(\epsilon^2+|z|^2)^2+\tau^2}
\left|\frac{|z|^2}{(\epsilon^2+|z|^2)^2+\tau^2}-\frac{|z|^2}{|z|^4+\tau^2}\right|\\
&\hspace{4mm}+\frac{2\tau^2}{|z|^4+\tau^2}\left|\frac{|z|^2}{(\epsilon^2+|z|^2)^2+\tau^2}-\frac{|z|^2}{|z|^4+\tau^2}\right|\\
&\leq 4\left|\frac{|z|^2}{(\epsilon^2+|z|^2)^2+\tau^2}-\frac{|z|^2}{|z|^4+\tau^2}\right|\leq \frac{8\epsilon^2}{|z|^3|\tau|^{\frac{1}{2}}}.
\end{split}
\end{equation}
Since
\begin{equation}\label{6.96}
\begin{split}
&\int_{\mathbb{H}^n}
\frac{1}{|z|^3|\tau|^{\frac{1}{2}}}\cdot\frac{dzd\tau}{(\tau^2+(1+|z|^2)^2)^{n+1}}
\leq 2\int^\infty_0\frac{d\tau}{\tau^{\frac{1}{2}}(1+\tau^2)}\int_{\{|z|\geq 0\}}\frac{dz}{|z|^3(1+|z|^2)^{2n-2}}\\
&\leq C\left(\int_0^1\frac{d\tau}{\tau^{\frac{1}{2}}}+\int^\infty_1\frac{d\tau}{1+\tau^2}\right)\left(\int_0^\infty\frac{r^{2n-4}dr}{(1+r^2)^{2n}}\right)
\leq C
\end{split}
\end{equation}
when $n\geq 2$ by (\ref{6.80}),
by the estimates (\ref{6.93.5})-(\ref{6.95}),
 we can rewrite (\ref{6.87}) as
\begin{equation*}
\begin{split}
\frac{d\Theta(t)_i}{dt}
&=(A_4\epsilon^2+O(\epsilon^3))\frac{dz_i(t)}{dt}
+O(\epsilon^3)z_i(t)\frac{dr(t)}{dt}
\end{split}
\end{equation*}
where $A_4$ is the positive constant given by
\begin{equation}\label{6.89}
A_4=2\int_{\mathbb{H}^n}
\frac{|z|^2}{|z|^4+\tau^2}\cdot\frac{4^{n+1}dzd\tau}{(\tau^2+(1+|z|^2)^2)^{n+1}},
\end{equation}
and  we can rewrite (\ref{6.88}) as
\begin{equation}\label{6.90}
\frac{d\Theta(t)_{n+1}}{dt}=(-\epsilon^3A_5+O(\epsilon^4))\frac{dr(t)}{dt}
+O(\epsilon^4)\frac{d\tau(t)}{dt}
\end{equation}
where
$A_5$ is the constant given by
\begin{equation}\label{6.91}
A_5=
\int_{\mathbb{H}^n}
\frac{4|z|^2(|z|^4-\tau^2)}{(|z|^4+\tau^2)^2}\cdot\frac{4^{n+1}dzd\tau}{(\tau^2+(1+|z|^2)^2)^{n+1}}.
\end{equation}
Note that $A_5$ is positive. To see this, note that the right hand side of (\ref{6.91}) can be written as
\begin{equation}\label{6.91a}
C\int_0^\infty\int_0^\infty\frac{(r^4-\tau^2)r^{2n+1}}{(r^4+\tau^2)[\tau^2+(1+r^2)^2]^{n+1}}drd\tau
\end{equation}
for some positive constant $C$. So it suffices to prove that the integral in (\ref{6.91a}) is positive. Let $u=r^2$ and $\tau=v$
and then using the polar coordinates $u=r\cos\theta$ and $v=r\sin\theta$, the integral
can be written as
\begin{equation*}
\begin{split}
&\int_0^\infty\int_0^\infty\frac{(r^4-\tau^2)r^{2n+1}}{(r^4+\tau^2)[\tau^2+(1+r^2)^2]^{n+1}}drd\tau\\
&=\frac{1}{2}\int_0^\infty\int_0^\infty\frac{(u^2-v^2)u^n}{(u^2+v^2)[v^2+(1+u)^2]^{n+1}}dudv\\
&=\frac{1}{2}\int_0^{\frac{\pi}{2}}\int_0^\infty\frac{r^{n+1}(\cos^2\theta-\sin^2\theta)\cos^n\theta}{(r^2+2r\cos\theta+1)^{n+1}}drd\theta\\
&=\frac{1}{2}\left(\int_0^{\frac{\pi}{4}}+\int_{\frac{\pi}{4}}^{\frac{\pi}{2}}\right)\int_0^\infty\frac{r^{n+1}(\cos^2\theta-\sin^2\theta)\cos^n\theta}{(r^2+2r\cos\theta+1)^{n+1}}drd\theta\\
&=\frac{1}{2}\int_0^{\frac{\pi}{4}}\int_0^\infty\frac{r^{n+1}(\cos^2\theta-\sin^2\theta)\cos^n\theta}{(r^2+2r\cos\theta+1)^{n+1}}drd\theta\\
&\hspace{4mm}
-\frac{1}{2}\int_{\frac{\pi}{4}}^{0}\int_0^\infty
\frac{r^{n+1}(\cos^2(\frac{\pi}{2}-\phi)-\sin^2(\frac{\pi}{2}-\phi))\cos^n(\frac{\pi}{2}-\phi)}{(r^2+2r\cos(\frac{\pi}{2}-\phi)+1)^{n+1}}drd\phi\\
&=\frac{1}{2}\int_0^{\frac{\pi}{4}}\int_0^\infty r^{n+1}(\cos^2\theta-\sin^2\theta)\left[\frac{\cos^n\theta}{(r^2+2r\cos\theta+1)^{n+1}}
-\frac{\sin^n\theta}{(r^2+2r\sin\theta+1)^{n+1}}\right]drd\theta\\
&=\frac{1}{2}\int_0^{\frac{\pi}{4}}\int_0^\infty r^{n+1}(\cos^2\theta-\sin^2\theta)\big[h(r,\cos\theta)-h(r,\sin\theta)\big]drd\theta,
\end{split}
\end{equation*}
where $h(r,z)=\displaystyle\frac{z^n}{(r^2+2rz+1)^{n+1}}$. Note that for $n\geq 2$ and $r>0$, $h(r,z)$ is an increasing function in $z\in[0,1]$. In particular, we have
$h(r,\cos\theta)\geq h(r,\sin\theta)$ for $\theta\in [0,\pi/4]$, which implies that the integral in (\ref{6.91a}) is positive,
 and hence $A_5$ is positive.

Now by (\ref{6.90}) and Lemma \ref{lem6.9}, we obtain
\begin{equation}\label{6.92}
\frac{d\Theta(t)_{n+1}}{dt}=-(\mbox{Vol}(S^{2n+1},\theta_0)^2-|\Theta(t)|^2)
\left[\left(\frac{A_5}{4\mbox{Vol}(S^{2n+1},\theta_0)A_3}+o(1)\right)\epsilon\frac{dr(t)}{dt}
+o(1)\epsilon\frac{d\tau(t)}{dt}\right].
\end{equation}
By symmetry,  $\Theta(t)_i=0$ for $1\leq i\leq n$. Thus by (\ref{6.84}) and (\ref{6.92}) we conclude that
\begin{equation*}
\begin{split}
&\frac{d}{dt}(\mbox{Vol}(S^{2n+1},\theta_0)^2-|\Theta(t)|^2)
=-\widehat{\Theta(t)}_{n+1}\frac{d\Theta(t)_{n+1}}{dt}-\Theta(t)_{n+1}\frac{d\widehat{\Theta(t)}_{n+1}}{dt}\\
&=2(\mbox{Vol}(S^{2n+1},\theta_0)+o(1))(\mbox{Vol}(S^{2n+1},\theta_0)^2-|\Theta(t)|^2)\\
&\hspace{4mm}\cdot
\left[\left(\frac{A_5}{4\mbox{Vol}(S^{2n+1},\theta_0)A_3}+o(1)\right)\epsilon\frac{dr(t)}{dt}
+o(1)\epsilon\frac{d\tau(t)}{dt}\right]\\
&=(\mbox{Vol}(S^{2n+1},\theta_0)^2-|\Theta(t)|^2)\left[\left(\frac{A_5}{2A_3}+o(1)\right)\epsilon\frac{dr(t)}{dt}
+o(1)\epsilon\frac{d\tau(t)}{dt}\right],
\end{split}
\end{equation*}
as required.
\end{proof}

\begin{prop}\label{prop6.11}
As $t\rightarrow\infty$, the contact form $\theta(t)$ concentrate at the critical point $Q$ of $f$
satisfying $\Delta_{\theta_0}f(Q)\leq 0$.
\end{prop}
\begin{proof}
It follows from Lemma \ref{lem6.7}-\ref{lem6.10} that
\begin{equation}\label{6.98}
\begin{split}
&\frac{d}{dt}(\mbox{Vol}(S^{2n+1},\theta_0)^2-|\Theta(t)|^2)\\
&=(2A_5\mbox{Vol}(S^{2n+1},\theta_0)+o(1))\epsilon^3\frac{dr(t)}{dt}+o(1)\epsilon^3\frac{d\tau(t)}{dt}\\
&=-(2(n+1)A_5+o(1))\epsilon^2(b_{n+1}+b_{2n+2})+O(|\nabla_{\theta_0}f(\widehat{\Theta(t)})|_{\theta_0}^2)\epsilon^4+O(\epsilon^5)\\
&=4(n+1)A_2A_5\alpha\epsilon^4(\Delta_{\theta_0}f(\widehat{\Theta(t)})
+O(1)|\nabla_{\theta_0}f(\widehat{\Theta(t)})|^2_{\theta_0}+O(\epsilon)).
\end{split}
\end{equation}
By (\ref{6.98}) and Lemma \ref{lem6.9}, we find
$$\left|\frac{d}{dt}(\mbox{Vol}(S^{2n+1},\theta_0)^2-|\Theta(t)|^2)\right|\leq C(\mbox{Vol}(S^{2n+1},\theta_0)^2-|\Theta(t)|^2)^2,$$
which yields
$$\mbox{Vol}(S^{2n+1},\theta_0)^2-|\Theta(t)|^2\geq\frac{C_0}{t},$$
for some constant $C_0>0$, while by Lemma \ref{lem6.9}, we have
\begin{equation}\label{6.93}
\epsilon^2\geq\frac{C_1}{t},
\end{equation}
for $t\geq t_0$ with some sufficiently large $t_0>0$ and a uniform constant $C_1>0$. It follows from Lemmas \ref{lem6.7}-\ref{lem6.10} that
\begin{equation*}
\begin{split}
\frac{d}{dt}f(\Theta(t))=\frac{d}{dt}f(\widehat{\Theta(t)})
&=\frac{1}{2|\Theta(t)|}\sum_{i=1}^n\frac{\partial f}{\partial z_i}(\widehat{\Theta(t)})\frac{d\Theta(t)_i}{dt}+O(\epsilon^3)\\
&\geq C\epsilon^2(|f'(\widehat{\Theta(t)})|^2+o(1))
\end{split}
\end{equation*}
where $f'$ denotes the gradient of $f$ with respect to the standard Riemannian metric on
$S^{2n+1}$. This implies by (\ref{6.93}) that
$$\left|\frac{d}{dt}f(\Theta(t))\right|\geq\frac{C_2}{t}(|f'(\widehat{\Theta(t)})|^2+o(1))$$
where $C_2>0$ and the error $o(1)\rightarrow 0$ as $t\rightarrow\infty$. Since $t^{-1}$ is divergent, the flow $(\Theta(t))_{t\geq 0}$
must accumulate at a critical point of $f$. To see the critical point with $\Delta_{\theta_0}f(Q)\leq 0$ are the only possible limit points of $\Theta(t)$,
first we observe that if $\Delta_{\theta_0}f(Q)>0$, then by (\ref{6.98}), we have, for sufficiently large $t$,
$\displaystyle\frac{d}{dt}(\mbox{Vol}(S^{2n+1},\theta_0)^2-|\Theta(t)|^2)>0$. Hence it will contradict the fact that
$\mbox{Vol}(S^{2n+1},\theta_0)^2-|\Theta(t)|^2\rightarrow 0$ as $t\rightarrow\infty$.

Therefore, the shadow flow $(\Theta(t))_{t\geq 0}$ converges to a unique point $Q\in S^{2n+1}$.
\end{proof}

\begin{lem}\label{lem6.12}
Under the assumptions of Theorem \ref{thm4.7}, let $u(0)=u_0\in
C^\infty_f$ be initial data of the flow $(\ref{2.3})$. Then as
$t\rightarrow\infty$, we have
$$E_f(u(t))\rightarrow R_{\theta_0}\mbox{\emph{Vol}}(S^{2n+1},\theta_0)^{\frac{1}{n+1}}f(Q)^{-\frac{n}{n+1}},$$
where $Q=\lim_{t\rightarrow\infty}\Theta(t)$ is the unique limit of
the shadow flow $\Theta(t)$ associated with $u(t)$.
\end{lem}
\begin{proof}
Note that
\begin{equation}\label{6.97}
E(u(t))=E(v(t))\rightarrow
R_{\theta_0}\mbox{Vol}(S^{2n+1},\theta_0)\hspace{2mm}\mbox{ as }t\rightarrow\infty
\end{equation}
by  Lemma \ref{lem4.14}. On the other hand, by Lemma
\ref{lem4.14}, we have
$$\int_{S^{2n+1}}fdV_\theta=\int_{S^{2n+1}}f_\phi
dV_\theta\rightarrow f(Q)\mbox{Vol}(S^{2n+1},\theta_0)\hspace{2mm}\mbox{ as }t\rightarrow\infty.$$
Combining these, the assertion follows.
\end{proof}

\section{Existence of conformal contact form}\label{section7}

In this section, for $p\in S^{2n+1}$, $0<\epsilon<1$, as before, we denote by $\phi_{-p,\epsilon}$ the projection with $-p$
at infinity, that is, $p$ becomes the north pole in the coordinates. Define a map
$$j:S^{2n+1}\times (0,\infty)\ni (p,\epsilon)\mapsto u_{p,\epsilon}=|\det(d\phi_{-p,\epsilon})|^{\frac{n}{2n+2}}\in C^\infty_*$$
where
$$C^\infty_*:=\Big\{0<u\in C^\infty(S^{2n+1}): \theta=u^{\frac{2}{n}}\theta_0
\mbox{ satisfies }\int_{S^{2n+1}}u^{2+\frac{2}{n}}dV_{\theta_0}
=\int_{S^{2n+1}}dV_{\theta_0}\Big\}.$$
Also let $\theta_{p,\epsilon}=\phi_{p,\epsilon}^*(\theta_0)=u_{p,\epsilon}^{\frac{2}{n}}\theta_0$ to get
$$
dV_{\theta_{p,\epsilon}}=u_{p,\epsilon}^{2+\frac{2}{n}}dV_{\theta_0}\rightharpoonup
\mbox{Vol}(S^{2n+1},\theta_0)\delta_p
$$
in the weak sense of measures as $\epsilon\rightarrow 0$. For $\gamma\in\mathbb{R}$, denote by
$$L_\gamma=\{u\in C^\infty_*: E_f(u)\leq \gamma\},$$
the sub-level set of $E_f$. For convenience, labeling all the critical points of $f$ by
$p_1,..., p_N$ such that $f(p_i)\leq f(p_j)$ for $1\leq i\leq j\leq N$, we set
$$\beta_i=R_{\theta_0}\mbox{Vol}(S^{2n+1},\theta_0)^{\frac{1}{n+1}}f(p_i)^{-\frac{n}{n+1}}=\lim_{\epsilon\rightarrow 0}E_f(u_{p_i,\epsilon}),\hspace{2mm}1\leq i\leq N.$$
In view of Proposition \ref{prop6.11}, under our assumption of $f$, minimum points of $f$ cannot be concentration points,
namely, the energy level where the concentration occurs is strictly less than $\beta_1$. Without loss of generality,
we assume all critical levels $f(p_i)$, $1\leq i\leq N$, are different, so that there exists a $\nu_0>0$ such that
$\beta_i-2\nu_0>\beta_{i+1}$, in fact, we can take $\nu_0=\displaystyle\frac{1}{3}\min_{1\leq i\leq N-1}\{\beta_i-\beta_{i+1}\}>0$.
In the following, denote by $u(t,u_0)$ the flow (\ref{2.3}) with initial data $u_0\in C_*^\infty$, and again denote the shadow flow
by
$$\Theta(t,u_0)=\int_{S^{2n+1}}\phi(t,u_0)dV_{S^{2n+1}}\hspace{2mm}\mbox{ with }
\widehat{\Theta(t,u_0)}=\frac{\Theta(t,u_0)}{\|\Theta(t,u_0)\|}\mbox{ if }\|\Theta(t,u_0)\|\neq 0.$$

Our main purpose of this section is to set up the following:

\begin{prop}\label{prop7.1}
\emph{(i)} If $\beta_1<\beta_0\leq\beta$, where $\beta$ has been chosen as in $(\ref{4.15})$, then the set $L_{\beta_0}$ is contractible.\\
\emph{(ii)} For any $0<\nu\leq \nu_0$ and each $1\leq i\leq N$, the sets $L_{\beta_i-\nu}$ and $L_{\beta_{i+1}+\nu}$ are homotopic
equivalent.\\
\emph{(iii)} For each critical point $p_i$ of $f$ where $\Delta_{\theta_0}f(p_i)>0$, the
sets $L_{\beta_i+\nu_0}$ and $L_{\beta_{i}-\nu_0}$ are homotopic equivalent.\\
\emph{(iv)} For each critical point $p_i$ of $f$ where $\Delta_{\theta_0}f(p_i)<0$, the
set $L_{\beta_i+\nu_0}$ is homotopic to the set $L_{\beta_{i}-\nu_0}$ with $(2n+1-ind(f,p_i))$-cell attached.
\end{prop}

By assuming Proposition \ref{prop7.1}, we can complete the proof of our main theorem.

\begin{proof}[Proof of Theorem \ref{thm1.1}]
By contradiction, suppose that the flow does not converge and $f$ cannot be realized as the Webster scalar curvature of
a contact form conformal to the standard contact form $\theta_0$ of $S^{2n+1}$. Then Proposition \ref{prop7.1} shows that,
$L_{\beta_0}$ is contractible for some suitable $\beta_0$ chosen in part (i) of Proposition \ref{prop7.1}; in addition, the flow
gives a homotopy equivalence of the set $L_{\beta_0}$ with a set $E_\infty$ whose homotopy type consists of a point
$\{p_0\}$ with cells of dimension $(2n+1-ind(f,p_i))$ attached for each critical point $p_i$ of $f$ where $\Delta_{\theta_0}f(p_i)<0$.

From \cite{Chang}, Theorem 4.3 on page 36, we conclude that the identity
\begin{equation}\label{7.53}
\sum_{i=0}^{2n+1}t^im_i=1+(1+t)\sum_{i=0}^{2n+1}t^ik_i
\end{equation}
holds with $k_i\geq 0$ and $m_i$ is given as in (\ref{1.0}).
Equating the coefficients of $t$ in the polynomials on the left and right hand side, we obtain (\ref{1.1}),
which violates the hypothesis in Theorem \ref{thm1.1} and thus leads to the desired contradiction.
\end{proof}

\textit{Remark.} By forming the alternating sum of the terms in (\ref{1.1}),
which corresponds to setting $t =-1$ in (\ref{7.53}), we obtain
\begin{equation*}
\sum_{f'(x)=0,\,\Delta_{\theta_0}f(x)<0}(-1)^{ind(f,x)}= -1,
\end{equation*}
which contradicts (\ref{1.2}). From this, we see that Theorem \ref{thm1.1} implies Theorem \ref{thm1.2}.

The rest of this section is devoted to proving Proposition \ref{prop7.1}. By the long existence of the flow
(\ref{2.3}) which was proved in part I, we can
assume that, for any fixed initial data $u_0$ and any finite $T>0$,
there exists $C(T)>0$ such that $\|u\|_{L^\infty([0,T]\times
C^{4n+4}(S^{2n+1}))}\leq C(T)$.

\begin{lem}\label{lem7.2}
Given any $T>0$, let $u_i(t)=u(t,u_i^0)$ be the solutions to our
flow $(\ref{2.3})$ with initial data $u_i^0\in C^\infty_f$ such that
$\|u_i\|_{L^\infty([0,T]\times C^{4n+4}(S^{2n+1}))}\leq C(T)$,
$i=1,2$. Then there exists a constant $C>0$ depending on $T$, $n$
and $\|u_i\|_{L^\infty([0,T]\times C^{4n+4}(S^{2n+1}))}$, $i=1,2$,
such that
$$\sup_{0\leq t\leq T}\|u_1(t)-u_2(t)\|_{S^2_{4n+4}(S^{2n+1},\theta_0)}
\leq C\|u_1^0-u_2^0\|_{S^2_{4n+4}(S^{2n+1},\theta_0)}.$$
\end{lem}
\begin{proof}
By the long existence of the flow (\ref{2.3}) which was proved in part I, we know that $u_i(t)$, $i=1,2$ are smooth in any
given finite time interval $[0,T]$. Moreover, by Lemma 2.8 in \cite{Ho3},
there exists constant $C_i=C_i(T)>0$ such that
\begin{equation}\label{7.1}
C_i^{-1}\leq \|u_i(t)\|_{L^\infty(S^{2n+1}\times[0,T])}\leq C_i\mbox{ for }i=1,2.
\end{equation}
For simplicity, we let $\theta_i=u_i(t)^{\frac{2}{n}}\theta_0$,
$R_i=R_{\theta_i}$, and
by (\ref{2.2}) and (\ref{2.4})
the factor $\alpha_i(t)$ can be expressed as
$$\alpha(u_i)=\alpha_i(t)=\frac{\int_{S^{2n+1}}\left((2+\frac{2}{n})
|\nabla_{\theta_0}u_i(t)|_{\theta_0}^2+R_{\theta_0}u_i(t)^2\right)dV_{\theta_0}}
{\int_{S^{2n+1}}fu_i(t)^{2+\frac{2}{n}}dV_{\theta_0}}$$ for $i=1,2$.
If set $w=u_2-u_1$, we can estimate the term
$\alpha(u_2)-\alpha(u_1)$ as follows:
\begin{equation}\label{7.2}
\begin{split}
&\alpha(u_2)-\alpha(u_1)=\int_{0}^1\frac{\partial}{\partial s}\alpha(u_1+sw)ds\\
&=\int_0^1\left[\vphantom{\frac{E(u_1+sw)}{(\int_{S^{2n+1}}f(u_1+sw)^{2+\frac{2}{n}}dV_{\theta_0})^2}}
\frac{2\int_{S^{2n+1}}\big((1-s)R_1u_1^{1+\frac{2}{n}}+sR_2u_2^{1+\frac{2}{n}}\big)w\,dV_{\theta_0}}
{\int_{S^{2n+1}}f(u_1+sw)^{2+\frac{2}{n}}dV_{\theta_0}}\right.\\
&\hspace{4mm}\left.
\vphantom{\frac{2\int_{S^{2n+1}}\big((1-s)R_1u_1^{1+\frac{2}{n}}+sR_2u_2^{1+\frac{2}{n}}\big)w\,dV_{\theta_0}}
{\int_{S^{2n+1}}f(u_1+sw)^{2+\frac{2}{n}}dV_{\theta_0}}}
-(2+\frac{2}{n})\frac{E(u_1+sw)}{(\int_{S^{2n+1}}f(u_1+sw)^{2+\frac{2}{n}}dV_{\theta_0})^2}
\int_{S^{2n+1}}f(u_1+sw)^{1+\frac{2}{n}}w\,dV_{\theta_0}\right]ds\\
&\leq C(\|R_1\|_{L^2(S^{2n+1},\theta_1)}+\|R_2\|_{L^2(S^{2n+1},\theta_2)}+E(u_1)+E(u_2))\|w\|_{L^2(S^{2n+1},\theta_0)}\\
&\leq C\|w\|_{L^2(S^{2n+1},\theta_0)},
\end{split}
\end{equation}
where we have used (\ref{7.1}), Lemma 2.11 in \cite{Ho3}, and the fact that $E(u_i)\leq\gamma$ for $i=1,2$, and
$E(u_1+sw)=E((1-s)u_1+su_2)\leq (1-s)E(u_1)+sE(u_2)\leq E(u_1)+E(u_2).$
From (\ref{2.3}) and (\ref{2.4}), that is,
$$
\frac{\partial u_i}{\partial t}=\frac{n}{2}(\alpha(u_i) f- R_i)u_i\hspace{2mm}\mbox{ for }i=1,2,
$$
and
$$
-(2+\frac{2}{n})\Delta_{\theta_0}u_i+R_{\theta_0}u_i=R_iu_i^{1+\frac{2}{n}}\hspace{2mm}\mbox{ for }i=1,2,
$$
a direct computation yields
\begin{equation}\label{7.3}
\begin{split}
\frac{\partial w}{\partial t}&=\frac{\partial u_2}{\partial t}-\frac{\partial u_1}{\partial t}\\
&=\frac{n}{2}(R_1u_1-R_2u_2)+\frac{n}{2}\big[\alpha(u_2) fu_2-\alpha(u_1) fu_1\big]\\
&=\frac{n}{2}(R_1u_1-R_2u_2)+\frac{n}{2}\big[\alpha(u_1) fw+\alpha(u_2) fu_2-\alpha(u_1) fu_2\big]\\
&=\frac{n}{2}\Big[u_2^{-\frac{2}{n}}\big[(R_{\theta_0}u_2-R_2u_2^{1+\frac{2}{n}})
-(R_{\theta_0}u_1-R_1u_1^{1+\frac{2}{n}})\big]\\
&\hspace{8mm}+\big[\alpha(u_1)f-R_1u_1^{1+\frac{2}{n}}\big(\frac{u_2^{-\frac{2}{n}}-u_1^{-\frac{2}{n}}}{u_2-u_1}\big)
-R_{\theta_0}u_2^{-\frac{2}{n}}\big]w\Big]\\
&\hspace{4mm}+\frac{n}{2}(\alpha(u_2)-\alpha(u_1))u_2f\\
&=\frac{n}{2}\big[(2+\frac{2}{n})u_2^{-\frac{2}{n}}(\Delta_{\theta_0}u_2-\Delta_{\theta_0}u_1)+d(x,t)w\big]+
\frac{n}{2}(\alpha(u_2)-\alpha(u_1))u_2f\\
&=\frac{n}{2}\big[(2+\frac{2}{n})u_2^{-\frac{2}{n}}\Delta_{\theta_0}w+d(x,t)w\big]+
\frac{n}{2}(\alpha(u_2)-\alpha(u_1))u_2f,
\end{split}
\end{equation}
where $d(x,t)=\alpha(u_1)f+b(x,t)-R_{\theta_0}u_2^{-\frac{2}{n}}$ and
$b(x,t)=-R_1u_1^{1+\frac{2}{n}}\displaystyle\left(\frac{u_2^{-\frac{2}{n}}-u_1^{-\frac{2}{n}}}{u_2-u_1}\right)$.
Thus, from (\ref{7.3}), we have
\begin{equation}\label{7.4}
\begin{split}
&\frac{d}{dt}\left(\int_{S^{2n+1}}w^2dV_{\theta_0}\right)
=\int_{S^{2n+1}}2w\frac{\partial w}{\partial t}\,dV_{\theta_0}\\
&=n\int_{S^{2n+1}}\left((2+\frac{2}{n})u_2^{-\frac{2}{n}}w\Delta_{\theta_0}w+d(x,t)w^2\right)dV_{\theta_0}
+n(\alpha(u_2)-\alpha(u_1))\int_{S^{2n+1}}u_2fw\,dV_{\theta_0}.
\end{split}
\end{equation}
By (\ref{7.1}), H\"{o}lder's inequality and Young's inequality, we have
\begin{equation}\label{7.5}
\begin{split}
&\int_{S^{2n+1}}u_2^{-\frac{2}{n}}w\Delta_{\theta_0}w\,dV_{\theta_0}\\
&=-\int_{S^{2n+1}}u_2^{-\frac{2}{n}}|\nabla_{\theta_0}w|_{\theta_0}^2dV_{\theta_0}
+\frac{2}{n}\int_{S^{2n+1}}u_2^{-(1+\frac{2}{n})}w\langle\nabla_{\theta_0}u_2,\nabla_{\theta_0}w\rangle_{\theta_0}dV_{\theta_0}\\
&\leq -\frac{1}{2}C_2^{-\frac{2}{n}}\int_{S^{2n+1}}|\nabla_{\theta_0}w|_{\theta_0}^2dV_{\theta_0}
+C\int_{S^{2n+1}}w^2dV_{\theta_0}.
\end{split}
\end{equation}
By (\ref{7.1}), (\ref{7.2}) and Lemma 2.11 in \cite{Ho3}, one has
\begin{equation}\label{7.6}
\left|(\alpha(u_2)-\alpha(u_1))\int_{S^{2n+1}}u_2fw\,dV_{\theta_0}\right|
+\left|\int_{S^{2n+1}}d(x,t)w^2dV_{\theta_0}\right|
\leq C\int_{S^{2n+1}}w^2dV_{\theta_0}.
\end{equation}
Combining (\ref{7.4}), (\ref{7.5}) and (\ref{7.6}), we obtain
\begin{equation*}
\frac{d}{dt}\left(\int_{S^{2n+1}}w^2dV_{\theta_0}\right)
+(n+1)C_2^{-\frac{2}{n}}\int_{S^{2n+1}}|\nabla_{\theta_0}w|_{\theta_0}^2dV_{\theta_0}
\leq C_0\int_{S^{2n+1}}w^2dV_{\theta_0}
\end{equation*}
for some constant $C_0$. Therefore, for any $t\in [0,T]$, we can integrate the above differential inequality
from $0$ to $t$ to obtain
\begin{equation}\label{7.7}
\int_{S^{2n+1}}w^2(t)dV_{\theta_0}\leq e^{C_0t}\int_{S^{2n+1}}w^2(0)dV_{\theta_0}.
\end{equation}

Next, for any $p\in\mathbb{N}$ with $p\leq 2n+2$, by (\ref{7.3}) one has
\begin{equation*}
\begin{split}
&\frac{d}{dt}\int_{S^{2n+1}}|(-\Delta_{\theta_0})^pw|^2dV_{\theta_0}
=2\int_{S^{2n+1}}\frac{\partial w}{\partial t}(-\Delta_{\theta_0})^{2p}wdV_{\theta_0}\\
&=n\int_{S^{2n+1}}\left[(2+\frac{2}{n})u_2^{-\frac{2}{n}}\Delta_{\theta_0}w+d(x,t)w\right](-\Delta_{\theta_0})^{2p}w\,dV_{\theta_0}\\
&\hspace{4mm}+n\int_{S^{2n+1}}(\alpha(u_2)-\alpha(u_1))u_2f(-\Delta_{\theta_0})^{2p}w\,dV_{\theta_0}.
\end{split}
\end{equation*}
By Interpolation, H\"{o}lder's and Young's inequalities, we obtain
\begin{equation*}
\begin{split}
&\int_{S^{2n+1}}(-\Delta_{\theta_0})^{2p}w(u_2^{-\frac{2}{n}}\Delta_{\theta_0}w)dV_{\theta_0}\\
&\leq -\frac{1}{2}C_2^{-\frac{2}{n}}\int_{S^{2n+1}}|\nabla_{\theta_0}(-\Delta_{\theta_0})^{p}w|^2_{\theta_0}dV_{\theta_0}
+C\int_{S^{2n+1}}|(-\Delta_{\theta_0})^{p}w|^2dV_{\theta_0}+C\int_{S^{2n+1}}w^2dV_{\theta_0}
\end{split}
\end{equation*}
and also
\begin{equation*}
\left|\int_{S^{2n+1}}d(x,t)w(-\Delta_{\theta_0})^{2p}w\, dV_{\theta_0}\right|
\leq C\int_{S^{2n+1}}|(-\Delta_{\theta_0})^{p}w|^2 dV_{\theta_0}+C\int_{S^{2n+1}}w^2 dV_{\theta_0}.
\end{equation*}
By (\ref{7.2}) and integration by parts, we get
\begin{equation*}
\begin{split}
&\left|\int_{S^{2n+1}}(-\Delta_{\theta_0})^{2p}w(\alpha(u_2)-\alpha(u_1))u_2f\,dV_{\theta_0}\right|\\
&
=\left|\int_{S^{2n+1}}w(\alpha(u_2)-\alpha(u_1))(-\Delta_{\theta_0})^{2p}(u_2f)dV_{\theta_0}\right|\\
&\leq
C\|w\|_{L^2(S^{2n+1},\theta_0)}\left|\int_{S^{2n+1}}w(-\Delta_{\theta_0})^{2p}(u_2f)dV_{\theta_0}\right|
\leq C\|w\|_{L^2(S^{2n+1},\theta_0)}^2.
\end{split}
\end{equation*}
Combining the above estimates, we obtain
\begin{equation*}
\begin{split}
\frac{d}{dt}\int_{S^{2n+1}}|(-\Delta_{\theta_0})^pw|^2dV_{\theta_0}
&\leq C\int_{S^{2n+1}}|(-\Delta_{\theta_0})^{p}w|^2dV_{\theta_0}+C\int_{S^{2n+1}}w^2 dV_{\theta_0}\\
&\leq C\int_{S^{2n+1}}|(-\Delta_{\theta_0})^{p}w|^2dV_{\theta_0}+Ce^{C_0t}\int_{S^{2n+1}}w^2(0)dV_{\theta_0}
\end{split}
\end{equation*}
by (\ref{7.7}). Integrating it from $0$ to $t$, where $t\in [0,T]$, we get
\begin{equation}\label{7.8}
\begin{split}
&\int_{S^{2n+1}}|(-\Delta_{\theta_0})^pw(t)|^2dV_{\theta_0}\\
&\leq e^{Ct}\int_{S^{2n+1}}|(-\Delta_{\theta_0})^{p}w(0)|^2dV_{\theta_0}+e^{(C+C_0)t}(\frac{C}{C_0})\int_{S^{2n+1}}w^2(0)dV_{\theta_0}.
\end{split}
\end{equation}
Therefore, by choosing $p=2n+2$, we can combine (\ref{7.7}) and (\ref{7.8}) to yield
$$\sup_{0\leq t\leq T}\|w(t)\|_{S^2_{4n+4}(S^{2n+1},\theta_0)}
\leq Ce^{Ct}\|w(0)\|_{S^2_{4n+4}(S^{2n+1},\theta_0)}$$
as required.
\end{proof}

\begin{proof}[Proof of Proposition \ref{prop7.1} \emph{(i)}] Let
$\beta_0$ be chosen above, i.e. $\beta_1<\beta_0\leq \beta$. For
$u_0\in L_{\beta_0}$, let $u(t,u_0)$ be the solution of the flow
determined by the initial data $u_0$. By Proposition \ref{prop2.2}, the energy $E_f$
is  decreasing along the flow. In particular, we have
$$E_f(u(t,u_0))\leq \beta_0.$$
Now for sufficiently small $\epsilon>0$, we claim that there exists
$T_1(u_0,\epsilon)>0$ which depends continuously on $u_0$ in the
$S_{4n+4}^2(S^{2n+1},\theta_0)$ topology and if
$t>T_1=T_1(u_0,\epsilon)$, we have
\begin{equation}\label{7.9}
\|v-1\|_{C^1_P(S^{2n+1})}<\epsilon.
\end{equation}

To prove this claim, first note that we can choose $T_2$ large so
that if $t\geq T_2$, then
\begin{equation}\label{7.10}
\|v-1\|_{C^1_P(S^{2n+1})}<\frac{1}{2}.
\end{equation}
This is possible since $\|v-1\|_{C^1(S^{2n+1},\theta_0)}\rightarrow
0$ as $t\rightarrow\infty$ by Lemma \ref{lem4.14}. Thus it follows
from the expression for $\Delta_{\theta_0}v$ as in (\ref{6.44})
that, for some constant $C_1$ which depends on $n$ and $T_3$, the
upper bounds of $F_{4n+4}$ and $\alpha(t)$, the maximum of $f$ as
well as the constant we have found in Lemma \ref{lem6.6},
\begin{equation}\label{7.11}
\int_{S^{2n+1}}|-\Delta_{\theta_0}v|^{2n+2}dV_{\theta_0}\leq
C_1(F_2^{\frac{1}{2}}+\|f_\phi-f(\widehat{\Theta(t)})\|_{L^2(S^{2n+1},\theta_0)}).
\end{equation}
Second it follows from (\ref{7.10}) and Lemma \ref{lem6.6} that for
$t\geq T_2$
\begin{equation}\label{7.12}
\begin{split}
\int_{S^{2n+1}}|v-1|^{2n+2}dV_{\theta_0}
&\leq\left(\frac{1}{2}\right)^{2n}\int_{S^{2n+1}}|v-1|^2dV_{\theta_0}\\
&\leq
C_2(F_2^{\frac{1}{2}}+\|f_\phi-f(\widehat{\Theta(t)})\|_{L^2(S^{2n+1},\theta_0)})
\end{split}
\end{equation}
for some constant $C_2>0$.

Then by Folland-Stein embedding theorem, there exists a constant $C_0>0$ depending only on the dimension $n$ such that
\begin{equation}\label{7.13}
\|v-1\|_{C^1_P(S^{2n+1})}\leq C_0\left[\int_{S^{2n+1}}|-\Delta_{\theta_0}v|^{2n+2}dV_{\theta_0}+\int_{S^{2n+1}}|v-1|^{2n+2}dV_{\theta_0}\right]^{\frac{1}{2n+2}}.
\end{equation}

Now we choose $T_3>T_2$ such that the quantity $|o(1)|<1$ in the Lemma \ref{lem6.1}
 for $t>T_3$.

Choose $B=(n+3)M^{\frac{n}{n+1}}\mbox{Vol}(S^{2n+1},\theta_0)^{\frac{n}{n+1}}$ where $M=\max_{S^{2n+1}}f$ and consider
\begin{equation}\label{7.14}
g(t)=F_2(t)+B\left[E_f(u)(t)-R_{\theta_0}\mbox{Vol}(S^{2n+1},\theta_0)^{\frac{1}{n+1}}f(Q)^{-\frac{n}{n+1}}\right],
\end{equation}
where $Q$ is the unique concentration point of the flow or the shadow flow $\Theta(t)$. It follows from Proposition \ref{prop2.2} and Lemma \ref{lem6.1}
that
\begin{equation*}
\begin{split}
\frac{dg(t)}{dt}&=\frac{d}{dt}F_2(t)-\frac{Bn\int_{S^{2n+1}}(\alpha f-R_\theta)^2u^{2+\frac{2}{n}}dV_{\theta_0}}{
\left(\int_{S^{2n+1}}fu^{2+\frac{2}{n}}dV_{\theta_0}\right)^{\frac{n}{n+1}}}\\
&\leq (n+1+o(1))(nF_2(t)-2G_2(t))+o(1)F_2(t)-n(n+3)F_2(t)<0
\end{split}
\end{equation*}
and $g(t)>0$ for all $t\geq T_3$.

Now for any $\epsilon>0$, since $\lim_{t\rightarrow\infty}\|f_\phi-f(\widehat{\Theta(t)})\|_{L^2(S^{2n+1},\theta_0)}=0$,
there exists a bigger $T_4\geq T_3$ such that for all $t\geq T_4$, we have
$\|f_\phi-f(\widehat{\Theta(t)})\|_{L^2(S^{2n+1},\theta_0)}\leq\displaystyle \frac{\epsilon^{2n+2}}{2C_3(2C_0)^{2n+2}}$ where $C_0$ is
given in the inequality (\ref{7.13}), and $C_3=C_1+C_2$ where $C_1$ and $C_2$ are respectively given in the inequalities
(\ref{7.11}) and (\ref{7.12}). Then we define $\delta=\min\displaystyle\left\{\frac{\epsilon^{4n+4}}{4C_3^2(2C_0)^{4n+4}},g(T_4)\right\}>0$.
Since $\lim_{t\rightarrow\infty}g(t)=0$ in view of Lemma \ref{lem3.2} and \ref{lem6.12}, there exists a $T_5\geq T_4+1$ such
that $g(T_5)<\delta$. Hence the set $\{t: t\geq T_4+1\mbox{ and }g(t)<\delta\}$ is non-empty. Finally we select $T_1(u_0)\equiv T_1(\epsilon,u_0)
=\inf\{t: t\geq T_4+1\mbox{ and }g(t)<\delta\}$. We need the following two properties: (i) $T_1(u_0)$ is continuously dependent on $u_0$ in $S^2_{4n+4}(S^{2n+1},\theta_0)$ and (ii) for all $t\geq T_1(u)$, $\|v-1\|_{C^1_P(S^{2n+1})}<\epsilon$.

In fact, (i) follows from monotonicity of $g$ and continuous dependence on the initial data of our flow in $S^2_{4n+4}(S^{2n+1},\theta_0)$-norm
as we did in Lemma \ref{lem7.2}. For (ii), observe that if $t>T_1(u_0)$, then $g(t)<g(T_1(u_0))\leq \delta$ thanks to the fact that $g$ is decreasing. Since
$E_f(u)(t)-R_{\theta_0}\mbox{Vol}(S^{2n+1},\theta_0)^{\frac{1}{n+1}}f(Q)^{-\frac{n}{n+1}}\geq 0$ for all $t\geq 0$ in view of
Proposition \ref{prop2.2} and Lemma \ref{lem6.12}, we conclude that $F_2(t)\leq\delta$ for all $t\geq T_1(u_0)$. Thus by estimates
(\ref{7.11}), (\ref{7.12}) and (\ref{7.13}), if $t>T_1(u_0)>T_4$, then we have
\begin{equation*}
\begin{split}
\|v-1\|_{C^1_P(S^{2n+1})}&\leq C_0\Big[(C_1+C_2)\big(F_2(t)^{\frac{1}{2}}+\|f_\phi-f(\widehat{\Theta(t)})\|_{L^2(S^{2n+1},\theta_0)}\big)\Big]^{\frac{1}{2n+2}}\\
&\leq C_0\left[(C_1+C_2)\left(\delta^{\frac{1}{2}}+\frac{\epsilon^{2n+2}}{2C_3(2C_0)^{2n+2}}\right)\right]^{\frac{1}{2n+2}}<\epsilon.
\end{split}
\end{equation*}
Therefore our claim is established.

Then we choose two positive constants $\sigma_1$, $\sigma_2$ to
normalize the two functions $v=u(T_1)\circ\phi
(\det(d\phi))^{\frac{n}{2n+2}}$ and $1$, such that
\begin{equation}\label{7.16}
\sigma_1^{2+\frac{2}{n}}\int_{S^{2n+1}}f\circ\phi\,
v^{2+\frac{2}{n}}dV_{\theta_0}=1\hspace{2mm}\mbox{ and }\hspace{2mm}
\sigma_2^{2+\frac{2}{n}}\int_{S^{2n+1}}f\circ\phi\, dV_{\theta_0}=1.
\end{equation}
By (\ref{7.9}), we have
\begin{equation}\label{7.17}
|\sigma_1-\sigma_2|=O(\epsilon).
\end{equation}

Now we define a homotopy on $L_{\beta_0}$ by
\begin{equation*}H(s,u_0)=\left\{
             \begin{array}{ll}
               u(3sT_1,u_0),  \hbox{ if }0\leq s\leq\displaystyle\frac{1}{3}; \\
               \displaystyle\frac{1}{\sigma_1}\Big[(2-3s)\sigma_1^{2+\frac{2}{n}}
               u(T_1,u_0)^{2+\frac{2}{n}}+(3s-1)\sigma_2^{2+\frac{2}{n}}\det(d\phi^{-1})\Big]^{\frac{n}{2n+2}} \hbox{ if }\displaystyle\frac{1}{3}\leq s\leq\frac{2}{3}; \\
               \displaystyle\frac{\sigma_2}{\sigma_1}\Big[\det\big(d(\Psi\circ\delta_{-q(T_1),3(1-s)r(T_1)+(3s-2)}
               \circ\pi)\big)\Big]^{\frac{n}{2n+2}},
               \hbox{ if }\displaystyle\frac{2}{3}\leq s\leq 1.
             \end{array}
           \right.
\end{equation*}
Obviously, $H(s,u_0)$ induces a contraction within $C^\infty_*$. One calculates that $E_f(H(s,u_0))\leq\beta_0$ if
$s\in\displaystyle[0,\frac{1}{3}]\cup[\frac{2}{3},1]$. Hence we have left to check that we also have $E_f(H(s,u_0))\leq\beta_0$
for $s\in\displaystyle[\frac{1}{3},\frac{2}{3}]$. To do this, for simplicity,
set $F(s)=E_f(H(s,u_0))$ for $s\in\displaystyle[\frac{1}{3},\frac{2}{3}]$.
Then we claim that
for sufficiently large $T_1>0$, there holds
\begin{equation}\label{7.18}
\frac{d^2}{ds^2}F(s)>0\hspace{2mm}\mbox{ for all }s\in[\frac{1}{3},\frac{2}{3}].
\end{equation}
Thus we can conclude that $F(s)$ achieves its maximum value at $s=\displaystyle\frac{1}{3}$
or $s=\displaystyle\frac{2}{3}$, namely,
$$E_f(H(s,u_0))\leq\max\Big\{E_f(H(\frac{1}{3},u_0)),E_f(H(\frac{2}{3},u_0))\Big\}\leq\beta_0\hspace{2mm}
\mbox{ for all }s\in[\frac{1}{3},\frac{2}{3}].$$
So the homotopy $H(s,u_0)$ is essentially a contraction within $C_f^\infty$.

In order to show (\ref{7.18}), first by conformal invariance of the energy, we have
$$E_f(H(s,u_0))=E_{f\circ\phi}\Big(H(s,u_0)\circ\phi\,(\det(d\phi))^{\frac{n}{2n+2}}\Big).$$
Then if we set
\begin{equation}\label{7.19}
v_s^{2+\frac{2}{n}}=(2-3s)(\sigma_1v)^{2+\frac{2}{n}}+(3s-1)\sigma_2^{2+\frac{2}{n}},
\end{equation}
we have
$$\sigma_1 H(s,u_0)\circ\phi\,(\det(d\phi))^{\frac{n}{2n+2}}=v_s$$
and
$$E_f(H(s,u_0))=E_{f\circ\phi}(v_s)$$
by using the fact that $E_f(\sigma u)=E_f(u)$ for any constant $\sigma>0$.
Hence we only need to estimate the energy $E_{f\circ\phi}(v_s)$
for $s\in\displaystyle[\frac{1}{3},\frac{2}{3}]$.
Now we denote a dot the $s$-derivative. We can derive from (\ref{7.19}) that
\begin{equation}\label{7.20}
\dot{v}_s=3v_s^{-1-\frac{2}{n}}\big(\sigma_2^{2+\frac{2}{n}}-(\sigma_1v)^{2+\frac{2}{n}}\big)/(2+\frac{2}{n}).
\end{equation}
One has the estimate
\begin{equation}\label{7.21}
\|\dot{v}_s\|_{C^0(S^{2n+1})}=O(\epsilon)
\end{equation}
thanks to (\ref{7.9}) and (\ref{7.17}). Note also that by (\ref{7.20}) we have
\begin{equation}\label{7.22}
\ddot{v}_s=-(1+\frac{2}{n})v_s^{-1}(\dot{v}_s)^2.
\end{equation}
By (\ref{7.9}) and (\ref{7.17}), we have
\begin{equation*}
\begin{split}
\|v_s^{2+\frac{2}{n}}-\sigma_1^{2+\frac{2}{n}}\|_{C^0}
&=\|(2-3s)(\sigma_1v)^{2+\frac{2}{n}}+(3s-1)\sigma_2^{2+\frac{2}{n}}-\sigma_1^{2+\frac{2}{n}}\|_{C^0}\\
&=\|(2-3s)\sigma_1^{2+\frac{2}{n}}(v^{2+\frac{2}{n}}-1)+(3s-1)(\sigma_2^{2+\frac{2}{n}}-\sigma_1^{2+\frac{2}{n}})\|_{C^0}\\
&= O(\epsilon),
\end{split}
\end{equation*}
which implies that
\begin{equation}\label{7.23}
\|v_s-\sigma_1\|_{C^0}=O(\epsilon).
\end{equation}
It follows from (\ref{7.19}) that
$$(2+\frac{2}{n})v_s^{1+\frac{2}{n}}\nabla_{\theta_0}v_s
=(2-3s)(2+\frac{2}{n})\sigma_1^{2+\frac{2}{n}}v^{1+\frac{2}{n}}\nabla_{\theta_0}v,$$
which implies that
\begin{equation}\label{7.24}
\|\nabla_{\theta_0}v_s\|_{C^0}=O(\epsilon)
\end{equation}
by (\ref{7.9}). Moreover, it follows from (\ref{4.19}), (\ref{7.16}) and (\ref{7.19}) that
\begin{equation}\label{7.25}
\begin{split}
&\int_{S^{2n+1}}f\circ\phi\,v_s^{2+\frac{2}{n}}dV_{\theta_0}\\
&=(2-3s)\sigma_1^{2+\frac{2}{n}}\int_{S^{2n+1}}f\circ\phi\,v^{2+\frac{2}{n}}dV_{\theta_0}
+(3s-1)\sigma_2^{2+\frac{2}{n}}\int_{S^{2n+1}}f\circ\phi\,dV_{\theta_0}\\
&=(2-3s)+(3s-1)=1
\end{split}
\end{equation}
and
\begin{equation}\label{7.26}
\begin{split}
&\int_{S^{2n+1}}(x,\overline{x})v_s^{2+\frac{2}{n}}dV_{\theta_0}\\
&=(2-3s)\sigma_1^{2+\frac{2}{n}}\int_{S^{2n+1}}(x,\overline{x})v^{2+\frac{2}{n}}dV_{\theta_0}
+(3s-1)\sigma_2^{2+\frac{2}{n}}\int_{S^{2n+1}}(x,\overline{x})dV_{\theta_0}=0.
\end{split}
\end{equation}
From (\ref{7.25}) and (\ref{7.26}), we obtain
\begin{equation}\label{7.27}
\int_{S^{2n+1}}f\circ\phi\,v_s^{1+\frac{2}{n}}\dot{v}_s\,dV_{\theta_0}=0\hspace{2mm}\mbox{ and }
\hspace{2mm}\int_{S^{2n+1}}(x,\overline{x})v_s^{1+\frac{2}{n}}\dot{v}_s\,dV_{\theta_0}=0.
\end{equation}
On the other hand, for any positive function $f$, a direct computation yields
\begin{equation}\label{7.28}
\begin{split}
dE_f(u)(\eta)&=\frac{d}{dr}E_f(u+r\eta)\Big|_{r=0}\\
&=2\left(\int_{S^{2n+1}}fu^{2+\frac{2}{n}}dV_{\theta_0}\right)^{-\frac{n}{n+1}}\\
&\hspace{4mm}\cdot
\left[\int_{S^{2n+1}}\left((2+\frac{2}{n})\langle\nabla_{\theta_0}u,\nabla_{\theta_0}\eta\rangle_{\theta_0}+R_{\theta_0}u\eta \right)dV_{\theta_0}
\vphantom{\frac{\int_{S^{2n+1}}\left((2+\frac{2}{n})|\langle\nabla_{\theta_0}u|^2_{\theta_0}+R_{\theta_0}u^2\right)dV_{\theta_0}}
{\int_{S^{2n+1}}fu^{2+\frac{2}{n}}dV_{\theta_0}}}\right.\\
&\hspace{8mm}\left.-\frac{\int_{S^{2n+1}}\left((2+\frac{2}{n})|\nabla_{\theta_0}u|^2_{\theta_0}+R_{\theta_0}u^2\right)dV_{\theta_0}}
{\int_{S^{2n+1}}fu^{2+\frac{2}{n}}dV_{\theta_0}}\int_{S^{2n+1}}fu^{1+\frac{2}{n}}\eta\,dV_{\theta_0}\right]
\end{split}
\end{equation}
and
\begin{equation}\label{7.29}
\begin{split}
&d^2E_f(u)(\zeta,\eta)=\frac{d}{dr}\big[dE_f(u+r\zeta)(\eta)\big]\Big|_{r=0}\\
&=2\left(\int_{S^{2n+1}}fu^{2+\frac{2}{n}}dV_{\theta_0}\right)^{-\frac{2n+1}{n+1}}\\
&\hspace{4mm}\cdot
\left\{
\vphantom
{\frac{\int_{S^{2n+1}}\left((2+\frac{2}{n})|\nabla_{\theta_0}u|^2_{\theta_0}+R_{\theta_0}u^2\right)dV_{\theta_0}}
{\int_{S^{2n+1}}fu^{2+\frac{2}{n}}dV_{\theta_0}}}
\int_{S^{2n+1}}\left((2+\frac{2}{n})\langle\nabla_{\theta_0}\zeta,\nabla_{\theta_0}\eta\rangle_{\theta_0}
+R_{\theta_0}\zeta\eta \right)dV_{\theta_0}\cdot
\int_{S^{2n+1}}fu^{2+\frac{2}{n}}dV_{\theta_0}\right.\\
&\hspace{8mm}-(1+\frac{2}{n})\int_{S^{2n+1}}\left((2+\frac{2}{n})|\nabla_{\theta_0}u|^2_{\theta_0}
+R_{\theta_0}u^2 \right)dV_{\theta_0}\cdot
\int_{S^{2n+1}}fu^{\frac{2}{n}}\zeta\eta\,dV_{\theta_0}\\
&\hspace{8mm}-2\left[\int_{S^{2n+1}}\left((2+\frac{2}{n})\langle\nabla_{\theta_0}u,\nabla_{\theta_0}\zeta\rangle_{\theta_0}
+R_{\theta_0}u\zeta \right)dV_{\theta_0}
\cdot
\int_{S^{2n+1}}fu^{1+\frac{2}{n}}\eta\,dV_{\theta_0}\right.\\
&\hspace{16mm}\left.+\int_{S^{2n+1}}\left((2+\frac{2}{n})\langle\nabla_{\theta_0}u,\nabla_{\theta_0}\eta\rangle_{\theta_0}
+R_{\theta_0}u\eta \right)dV_{\theta_0}
\cdot
\int_{S^{2n+1}}fu^{1+\frac{2}{n}}\zeta\,dV_{\theta_0}\right]\\
&\hspace{8mm}\left.+(2+\frac{2}{n})\frac{\int_{S^{2n+1}}\left((2+\frac{2}{n})|\nabla_{\theta_0}u|^2_{\theta_0}+R_{\theta_0}u^2\right)dV_{\theta_0}}
{\int_{S^{2n+1}}fu^{2+\frac{2}{n}}dV_{\theta_0}}\cdot\int_{S^{2n+1}}fu^{1+\frac{2}{n}}\zeta\,dV_{\theta_0}
\cdot\int_{S^{2n+1}}fu^{1+\frac{2}{n}}\eta\,dV_{\theta_0}\right\}.
\end{split}
\end{equation}
We observe that Folland-Stein embedding theorem shows that the map
$$u\mapsto d^2E_f(u)(\cdot,\cdot)\in L(S_1^2(S^{2n+1},\theta_0)\times S_1^2(S^{2n+1},\theta_0),\mathbb{R})$$
is continuous.

Notice that
\begin{equation*}
\begin{split}
&\int_{S^{2n+1}}\langle\nabla_{\theta_0}v_s,\nabla_{\theta_0}(v_s^{-1}\dot{v}_s^2)\rangle_{\theta_0}dV_{\theta_0}\\
&=-\int_{S^{2n+1}}v_s^{-2}|\nabla_{\theta_0}v_s|^2_{\theta_0}\dot{v}_s^2dV_{\theta_0}
+2\int_{S^{2n+1}}v_s^{-1}\dot{v}_s\langle\nabla_{\theta_0}v_s,\nabla_{\theta_0}\dot{v}_s\rangle_{\theta_0}dV_{\theta_0}\\
&=O(\|\nabla_{\theta_0}v_s\|^2_{C^0})\,(\|\dot{v}_s\|^2_{L^2}+\|\nabla_{\theta_0}\dot{v}_s\|^2_{L^2})\\
&=O(\epsilon)\,(\|\dot{v}_s\|^2_{L^2}+\|\nabla_{\theta_0}\dot{v}_s\|^2_{L^2})
\end{split}
\end{equation*}
by (\ref{7.23}) and (\ref{7.24}). Using  (\ref{7.21})-(\ref{7.25}),
 (\ref{7.28}) and (\ref{7.29}), we obtain
\begin{equation}\label{7.30}
\begin{split}
&\frac{d^2}{ds^2}E_{f\circ\phi}(v_s)\\
&=d^2E_{f\circ\phi}(v_s)(\dot{v}_s,\dot{v}_s)+dE_{f\circ\phi}(v_s)(\ddot{v}_s)\\
&=2(2+\frac{2}{n})\int_{S^{2n+1}}|\nabla_{\theta_0}\dot{v}_s|^2_{\theta_0}dV_{\theta_0}+2R_{\theta_0}
\int_{S^{2n+1}}\dot{v}_s^2dV_{\theta_0}\\
&\hspace{4mm}-2(1+\frac{2}{n})\int_{S^{2n+1}}\left((2+\frac{2}{n})|\nabla_{\theta_0}v_s|^2_{\theta_0}
+R_{\theta_0}v_s^2 \right)dV_{\theta_0}\cdot
\int_{S^{2n+1}}fv_s^{\frac{2}{n}}\dot{v}_s^2\,dV_{\theta_0}\\
&\hspace{4mm}
-8\int_{S^{2n+1}}\left((2+\frac{2}{n})\langle\nabla_{\theta_0}v_s,\nabla_{\theta_0}\dot{v}_s\rangle_{\theta_0}
+R_{\theta_0}v_s\dot{v}_s \right)dV_{\theta_0}
\cdot
\int_{S^{2n+1}}fv_s^{1+\frac{2}{n}}\dot{v}_s\,dV_{\theta_0}\\
&\hspace{4mm}+2(2+\frac{2}{n})\int_{S^{2n+1}}\left((2+\frac{2}{n})|\nabla_{\theta_0}v_s|^2_{\theta_0}
+R_{\theta_0}v_s^2 \right)dV_{\theta_0}\cdot
\left(\int_{S^{2n+1}}fv_s^{1+\frac{2}{n}}\dot{v}_s\,dV_{\theta_0}\right)^2\\
&\hspace{4mm}+2\int_{S^{2n+1}}\left((2+\frac{2}{n})\langle\nabla_{\theta_0}v_s,\nabla_{\theta_0}\ddot{v}_s\rangle_{\theta_0}
+R_{\theta_0}v_s\ddot{v}_s \right)dV_{\theta_0}\\
&\hspace{4mm}-2\int_{S^{2n+1}}\left((2+\frac{2}{n})|\nabla_{\theta_0}v_s|^2_{\theta_0}+R_{\theta_0}v_s^2 \right)dV_{\theta_0}\cdot
\int_{S^{2n+1}}fv_s^{1+\frac{2}{n}}\ddot{v}_s\,dV_{\theta_0}\\
&=\left(2(2+\frac{2}{n})+O(\epsilon)\right)\int_{S^{2n+1}}|\nabla_{\theta_0}\dot{v}_s|^2_{\theta_0}dV_{\theta_0}
-\left(\frac{4R_{\theta_0}}{n}+O(\epsilon)\right)\int_{S^{2n+1}}\dot{v}_s^2dV_{\theta_0}.
\end{split}
\end{equation}

Now we decompose $\dot{v}_s=\varphi+w$, where
$$w=\int_{S^{2n+1}}\dot{v}_s\,dV_{\theta_0}+\sum_{i=1}^{2n+2}\left(\int_{S^{2n+1}}\dot{v}_s\varphi_i\,dV_{\theta_0}\right)\varphi_i$$
and $\{\varphi_i\}$ are the eigenfunctions of $-\Delta_{\theta_0}$
given in section \ref{section6.2}. Let $\widehat{P(t)}$ be the limit point of the conformal CR diffeomorphism
 $\phi(t)$ in view of Lemma \ref{lem4.14}, one finds
by (\ref{7.27}) that
\begin{equation*}
\begin{split}
&\sigma_1^{1+\frac{2}{n}}f(\widehat{P(t)})\int_{S^{2n+1}}\dot{v}_s\,dV_{\theta_0}\\
&=\int_{S^{2n+1}}(\sigma_1^{1+\frac{2}{n}}f(\widehat{P(t)})-f\circ\phi\,
v_s^{1+\frac{2}{n}})\dot{v}_s\,dV_{\theta_0}\\
&=\int_{S^{2n+1}}\left[\sigma_1^{1+\frac{2}{n}}(f(\widehat{P(t)})-f\circ\phi)+
f\circ\phi\,(\sigma_1^{1+\frac{2}{n}}-
v_s^{1+\frac{2}{n}})\right]\dot{v}_s\,dV_{\theta_0}
\end{split}
\end{equation*}
and
$$\sigma_1^{1+\frac{2}{n}}\int_{S^{2n+1}}\dot{v}_s\varphi_i\,dV_{\theta_0}
=\int_{S^{2n+1}}(\sigma_1^{1+\frac{2}{n}}-
v_s^{1+\frac{2}{n}})\dot{v}_s\varphi_i\,dV_{\theta_0}.$$ Hence from
Lemma \ref{lem4.14}, (\ref{7.23}), and H\"{o}lder's inequality, we
obtain
\begin{equation*}
\int_{S^{2n+1}}\dot{v}_s\,dV_{\theta_0}=o(1)\|\dot{v}_s\|_{L^2}\hspace{2mm}\mbox{
and
}\hspace{2mm}\int_{S^{2n+1}}\dot{v}_s\varphi_i\,dV_{\theta_0}=O(\epsilon)\|\dot{v}_s\|_{L^2},
\end{equation*}
which implies that
\begin{equation}\label{7.31}
\begin{split}
&\|w\|_{L^2}=o(1)\|\dot{v}_s\|_{L^2}\hspace{2mm}\mbox{ and }\\
&\|\nabla_{\theta_0}w\|_{L^2}
=\left\|\sum_{i=1}^{2n+2}\left(\int_{S^{2n+1}}\dot{v}_s\varphi_i\,dV_{\theta_0}\right)
\nabla_{\theta_0}\varphi_i\right\|_{L^2}=O(\epsilon)\|\dot{v}_s\|_{L^2}
\end{split}
\end{equation}
by the definition of $w$. By (\ref{7.31}) and $\dot{v}_s=\varphi+w$,
we have
\begin{equation}\label{7.32}
\begin{split}
(1+o(1))\int_{S^{2n+1}}|\nabla_{\theta_0}\dot{v}_s|^2_{\theta_0}dV_{\theta_0}
&=\int_{S^{2n+1}}|\nabla_{\theta_0}\varphi|^2_{\theta_0}dV_{\theta_0}\\
&\geq\lambda_{2n+3}\int_{S^{2n+1}}\varphi^2dV_{\theta_0}\\
&=(\lambda_{2n+3}+o(1))\int_{S^{2n+1}}\dot{v}_s^2dV_{\theta_0}.
\end{split}
\end{equation}
Combining (\ref{7.30}) and (\ref{7.32}), we can conclude that
\begin{equation*}
\begin{split}
&\frac{d^2}{ds^2}E_{f\circ\phi}(v_s)\\
&=\left(2(2+\frac{2}{n})+O(\epsilon)\right)\int_{S^{2n+1}}|\nabla_{\theta_0}\dot{v}_s|^2_{\theta_0}dV_{\theta_0}
-\left(\frac{4R_{\theta_0}}{n}+O(\epsilon)\right)\int_{S^{2n+1}}\dot{v}_s^2dV_{\theta_0}\\
&\geq\left(2(2+\frac{2}{n})\lambda_{2n+3}-\frac{4R_{\theta_0}}{n}+o(1)\right)\int_{S^{2n+1}}\dot{v}_s^2dV_{\theta_0}>0,
\end{split}
\end{equation*}
since
$$2(2+\frac{2}{n})\lambda_{2n+3}-\frac{4R_{\theta_0}}{n}
=2(2+\frac{2}{n})\lambda_{2n+3}-\frac{4}{n}\cdot\frac{n(n+1)}{2}>0$$
thanks to $\lambda_{2n+3}>n/2$.

So we have established (\ref{7.18}). Notice that our homotopy $H(s,u_0)$ is the one
which is homotopic to the constant $\sigma_2/\sigma_1$.
Since
this is a constant, its energy is always less than $\beta_0$.
Then clearly this constant will be homotopic to the constant $1$ in the set $C_f^\infty$.
So we have finished the proof of (i).
\end{proof}

\begin{proof}[Proof of Proposition \ref{prop7.1} \emph{(ii)}]
In order to prove (ii), we re-scale the time $t$ by letting $\tau(t)$ solve
\begin{equation}\label{7.32.5}
\frac{d\tau}{dt}=\min\left\{\frac{1}{2},\epsilon^2(t,u_0)\right\},\hspace{2mm}\tau(0)=0.
\end{equation}
By (\ref{6.93}), we see that $\tau(t)\to \infty$ as $t\to\infty$. Set $U=u(\tau(t),u_0)$ and
$\Gamma(\tau)=\Theta(\tau(t),u_0)$. As in the proof of Proposition \ref{prop6.11}, we have
$$
\frac{d}{d\tau}(\mbox{Vol}(S^{2n+1},\theta_0)^2-|\Gamma(\tau)|^2)=4(n+1)A_2A_5\alpha\epsilon^3(\Delta_{\theta_0}f(\widehat{\Gamma(\tau)})
+O(1)|\nabla_{\theta_0}f(\widehat{\Gamma(\tau)})|^2_{\theta_0}+O(\epsilon)).
$$
and
\begin{equation}\label{7.32.6}
\left|\frac{d}{d\tau}(\mbox{Vol}(S^{2n+1},\theta_0)^2-|\Gamma(\tau)|^2)\right|\leq C(\mbox{Vol}(S^{2n+1},\theta_0)^2-|\Gamma(\tau)|^2)^2,
\end{equation}
with error $O(1)$ which means it is bounded as $\epsilon\to 0$. In the following argument, we will still use $t$ for $\tau(t)$, $u(t,u_0)$
for $U(\tau(t),u_0)$ and $\Theta(t)$ for $\Gamma(\tau(t))$ when there is no confusion arising.

Thus, for a given $0<\nu\leq\nu_0$, we claim that there exists $T>0$ such that $u(T,L_{\beta_i-\nu})\subset L_{\beta_{i+1}+\nu}$. Suppose on the contrary,
there exist, for each integer $k$, $T_k\geq 2k$ and an initial data $u_k\in L_{\beta_i-\nu}\setminus L_{\beta_{i+1}+\nu}$, such that
$$E_f(u(T_k,u_k))>\beta_{i+1}+\nu\hspace{2mm}\mbox{ for all }k.$$
By Lemma \ref{lem3.2}, there exists a sequence $t_k\in [T_k/2,T_k]$ such that $\displaystyle\int_{S^{2n+1}}|\alpha(t_k)f-R_{\theta_k}|^2dV_{\theta_k}\rightarrow 0$ as $k\rightarrow\infty$, where $\theta_k=u(t_k,u_k)^{\frac{2}{n}}\theta_0$, $k\in\mathbb{Z}^+$ and $R_{\theta_k}$ is the Webster scalar curvature of $\theta_k$. In fact, if for all $t\in  [T_k/2,T_k]$, $\displaystyle\int_{S^{2n+1}}|\alpha(t)f-R_{\theta(t)}|^2dV_{\theta(t)}\geq\epsilon_0>0$ for some fixed $\epsilon_0>0$ and $k$ sufficiently large, we would have
$$\beta_i-\beta_{i+1}-2\nu\geq\int_{\frac{T_k}{2}}^{T_k}\left(-\frac{dE_f(u(t,u_k))}{dt}\right)dt\geq\epsilon_0 C T_k/2$$
which contradicts the assumption that $T_k\rightarrow\infty$ as $k\rightarrow\infty$.

Let $\widehat{\Theta(t_k,u_k)}\in S^{2n+1}$ be the shadow points of the flow with the initial data $u_k$ valued at time $t_k$.
And $v_k=u(t_k,u_k)\circ\phi(t_k)(\det(d\phi(t_k)))^{\frac{n}{2n+2}}$. Then $v_k\rightarrow 1$ as $k$ sufficiently large. Up to a subsequence, the limit
$\widehat{\Theta}=\lim_{k\rightarrow\infty}\widehat{\Theta(t_k,u_k)}$ exists and $v_k\rightarrow 1$ in $C^{1,\alpha}$ for some $\alpha>0$ according to
Lemma \ref{lem4.14}. Then $\widehat{\Theta}$ must be a critical point of $f$ by (\ref{7.32.6}). Then, as in the proof of Lemma \ref{lem6.12},
the direct calculation shows that
$E_f(u(t_k,u_k))\rightarrow R_{\theta_0}\mbox{Vol}(S^{2n+1},\theta_0)^{\frac{1}{n+1}}f(\widehat{\Theta})^{-\frac{n}{n+1}}$ as $k\rightarrow\infty$. Since
$E_f(u_k)\leq\beta_i-\nu$, by Proposition \ref{prop2.2}, there hold $\widehat{\Theta}=p_{i_0}$ for some $i_0>i$ and
\begin{equation*}
\begin{split}
E_f(u(T_k,u_k))&\leq E_f(u(t_k,u_k))=E_{f\circ\phi(t_k)}(v(t_k,u_k))\\
&\leq R_{\theta_0}\mbox{Vol}(S^{2n+1},\theta_0)^{\frac{1}{n+1}}f(\widehat{\Theta})^{-\frac{n}{n+1}}+\nu\leq\beta_{i+1}+\nu,
\end{split}
\end{equation*}
which yields a contradiction.

For $u_0\in  L_{\beta_i-\nu}\setminus L_{\beta_{i+1}+\nu}$, let
$$T(u_0)=\inf\{t\geq 0: E_f(u(t,u_0))\leq\beta_{i+1}+\nu\}\leq T.$$
As in (i), $T(u_0)$ continuously depends on $u_0$ and the map $K(s,u_0)=u(sT(u_0),u_0)$ for $0\leq s\leq 1$ if $u\in L_{\beta_i-\nu}\setminus L_{\beta_{i+1}+\nu}$
and $K(s,u_0)=u_0$ if $u_0\in L_{\beta_{i+1}+\nu}$ defines the desired homotopy equivalence between $L_{\beta_{i+1}+\nu}$
 and $L_{\beta_i-\nu}$. This finishes the proof of (ii).
\end{proof}

For the proof of (iii) and (iv), we need some additional lemmas.

\begin{lem}\label{lem7.3}
With two dimensional constants $C_1>0$, $C_2>0$, provided that $\|v-1\|_{S_1^2(S^{2n+1},\theta_0)}$ is sufficiently small,
there holds
$$C_1\|v-1\|_{S_1^2(S^{2n+1},\theta_0)}^2\geq E(v)-R_{\theta_0}\mbox{\emph{Vol}}(S^{2n+1},\theta_0)\geq C_2\|v-1\|_{S_1^2(S^{2n+1},\theta_0)}^2,$$
for all $v\in S_1^2(S^{2n+1},\theta_0)\cap C_*^\infty$, the conformal factor of the normalized contact form $h$ satisfying $(\ref{4.19})$.
\end{lem}
\begin{proof}
Note that
$$E(v)-R_{\theta_0}\mbox{Vol}(S^{2n+1},\theta_0)
=\int_{S^{2n+1}}\left((2+\frac{2}{n})|\nabla_{\theta_0}v|_{\theta_0}^2+R_{\theta_0}(v^2-1)\right)dV_{\theta_0}.$$
Note also that
\begin{equation*}
\begin{split}
R_{\theta_0}\int_{S^{2n+1}}(v^2-1)dV_{\theta_0}&=R_{\theta_0}\int_{S^{2n+1}}(v-1)^2dV_{\theta_0}+2R_{\theta_0}\int_{S^{2n+1}}(v-1)dV_{\theta_0}\\
&=R_{\theta_0}\int_{S^{2n+1}}(v-1)^2dV_{\theta_0}+o(1)\|v-1\|_{S_1^2(S^{2n+1},\theta_0)}^2.
\end{split}
\end{equation*}
Thus, it is easy to derive from the above inequalities that there exists some constant $C_1>0$ such that
$$E(v)-R_{\theta_0}\mbox{Vol}(S^{2n+1},\theta_0)
\leq C_1\|v-1\|_{S_1^2(S^{2n+1},\theta_0)}^2.$$
On the other hand, let us assume that $\|v-1\|_{S_1^2(S^{2n+1},\theta_0)}\leq 1$. Since $v$ satisfies (\ref{4.19}), we use
(\ref{6.36}) to estimate
\begin{equation*}
\begin{split}
&E(v)-R_{\theta_0}\mbox{Vol}(S^{2n+1},\theta_0)\\
&=(2+\frac{2}{n})\int_{S^{2n+1}}|\nabla_{\theta_0}v|_{\theta_0}^2dV_{\theta_0}+R_{\theta_0}\int_{S^{2n+1}}(v-1)^2dV_{\theta_0}
+2R_{\theta_0}\int_{S^{2n+1}}(v-1)dV_{\theta_0}\\
&=\min\Big\{(2+\frac{2}{n}),R_{\theta_0}\Big\}
\int_{S^{2n+1}}\big(|\nabla_{\theta_0}v|_{\theta_0}^2+(v-1)^2\big)dV_{\theta_0}
+o(1)\|v-1\|_{S_1^2(S^{2n+1},\theta_0)}^2\\
&=C_2\int_{S^{2n+1}}\big(|\nabla_{\theta_0}v|_{\theta_0}^2+(v-1)^2\big)dV_{\theta_0}
\end{split}
\end{equation*}
for some constant $C_2>0$.
\end{proof}

For $r_0>0$ and each critical point $p_i\in S^{2n+1}$ of $f$, set
\begin{equation*}
\begin{split}
B_{r_0}(p_i)=\Big\{&u\in C^\infty_*: \theta=u^{\frac{2}{n}}\theta_0\mbox{ induces normalized contact form }\\
&h=\phi^*\theta=v^{\frac{2}{n}}\theta_0
\mbox{ with }\phi=\phi_{-p,\epsilon}\mbox{ for some }p\in S^{2n+1}\mbox{ and }\\
&0<\epsilon\leq 1\mbox{ such that }\|v-1\|^2_{S_1^2(S^{2n+1},\theta_0)}+|p-p_i|^2+\epsilon^2<r_0^2\Big\}.
\end{split}
\end{equation*}
As shown in \cite{Malchiodi&Struwe}, the new coordinates $(\epsilon, p, v)$ are introduced to $u\in B_{r_0}(p_i)$.
Under the assumption on $f$, from Morse lemma, we introduce the local coordinates $p=p^++p^-$ near $p_i=0$,
such that
$$f(p)=f(p_i)+|p^+|^2-|p^-|^2.$$

\begin{lem}\label{lem7.4}
For $r_0>0$ and $u=(\epsilon, p,v)\in B_{r_0}(p_i)$, with $o(1)\rightarrow 0$ as $r_0\rightarrow 0$,
there hold\\
\emph{(a)}\begin{equation}\label{7.33}
\begin{split}
\int_{S^{2n+1}}f\circ\phi_{-p,\epsilon}dV_h&=f(p)\mbox{\emph{Vol}}(S^{2n+1},\theta_0)+A_6\epsilon^2\Delta_{\theta_0}f(p)+O(\epsilon^4)\\
&\hspace{2mm}+o(1)\epsilon\|v-1\|_{S_1^2(S^{2n+1},\theta_0)},
\end{split}
\end{equation}
where $A_6$ is the positive constant defined as in $(\ref{7.38})$.\\
\emph{(b)} There holds
\begin{equation}\label{7.34}
\begin{split}
&\left|\frac{\partial}{\partial\epsilon}E_f(u)+\frac{n}{n+1}E(u)\Big(f(p)\mbox{\emph{Vol}}(S^{2n+1},\theta_0)\Big)^{-\frac{2n+1}{n+1}}
\epsilon A_6 \Delta_{\theta_0}f(p)\right|\\
&\leq C\epsilon^{2}+C(\epsilon+|p-p_i|)\|v-1\|_{S_1^2(S^{2n+1},\theta_0)}.
\end{split}
\end{equation}
In particular, if $\Delta_{\theta_0}f(p)>0$, we have
\begin{equation}\label{7.35}
\begin{split}
\frac{\partial}{\partial\epsilon}E_f(u)&\leq-\frac{n^2}{2}f(p)^{-\frac{2n+1}{n+1}}\mbox{\emph{Vol}}(S^{2n+1},\theta_0)^{-\frac{2n}{n+1}}
\epsilon A_6 \Delta_{\theta_0}f(p)\\
&\hspace{2mm}+C\epsilon^{2}+C(\epsilon+|p-p_i|)\|v-1\|_{S_1^2(S^{2n+1},\theta_0)}.
\end{split}
\end{equation}
\emph{(c)} For any $q\in T_p(S^{2n+1})$, there holds
\begin{equation}\label{7.36}
\begin{split}
&\left|\frac{\partial E_f(u)}{\partial p}\cdot q+\frac{n}{n+1}E(v)f(p)^{-\frac{2n+1}{n+1}}df(p)\cdot q\right|\\
&\leq  C\epsilon(\epsilon+\|v-1\|_{S_1^2(S^{2n+1},\theta_0)})|q|.
\end{split}
\end{equation}
\emph{(d)} There exists a uniform constant $C_0>0$ such that
\begin{equation}\label{7.37}
\left\langle\frac{\partial}{\partial v}E_f(u),v-1\right\rangle
\geq C_0\|v-1\|_{S_1^2(S^{2n+1},\theta_0)}^2+o(1)\epsilon\|v-1\|_{S_1^2(S^{2n+1},\theta_0)},
\end{equation}
where $\langle\cdot,\cdot\rangle$ denotes the duality pairing of $S_1^2(S^{2n+1},\theta_0)$ with its dual.
\end{lem}
\begin{proof}
For notational convenience, let
$$A=A(u)=\int_{S^{2n+1}}f\circ\phi_{-p,\epsilon}dV_h.$$
(a) Observe that
$$A-f(p)\mbox{Vol}(S^{2n+1},\theta_0)=\int_{S^{2n+1}}(f\circ\phi_{-p,\epsilon}-f(p))dV_{\theta_0}+I,$$
where the error term $I$ is given by
$$I=\int_{S^{2n+1}}(f\circ\phi_{-p,\epsilon}-f(p))(v^{2+\frac{2}{n}}-1)dV_{\theta_0}$$
which can be estimated as follows:
\begin{equation}\label{7.50}
\begin{split}
|I|&\leq \|f\circ\phi_{-p,\epsilon}-f(p)\|_{L^2(S^{2n+1},\theta_0)}\|v^{2+\frac{2}{n}}-1\|_{L^2(S^{2n+1},\theta_0)}\\
&\leq o(1)\epsilon\|v-1\|_{S_1^2(S^{2n+1},\theta_0)}
\end{split}
\end{equation}
in view of (\ref{6.55}) and $|\nabla_{\theta_0}f(p)|_{\theta_0}\rightarrow 0$ as $r_0\rightarrow 0$. Using the expansion of $f$ in (\ref{6.53})
around $p$,
we obtain by symmetry, (\ref{6.54}) and (\ref{6.58}) that
\begin{equation*}
\begin{split}
&A-f(p)\mbox{Vol}(S^{2n+1},\theta_0)\\
&=\int_{B_{\epsilon^{-1}}(0)}(f\circ\phi_{-p,\epsilon}-f(p))\frac{4^{n+1}dzd\tau}{(\tau^2+(1+|z|^2)^2)^{n+1}}
+O(\epsilon^{2n})+o(1)\epsilon\|v-1\|_{S_1^2(S^{2n+1},\theta_0)}\\
&=\frac{1}{2n}\epsilon^2\Delta_{\theta_0}f(p)\int_{\mathbb{H}^n}\frac{4^{n+1}|z|^2dzd\tau}{(\tau^2+(1+|z|^2)^2)^{n+1}}
+C\epsilon^4\int_{\mathbb{H}^n}\frac{4^{n+1}\tau^2dzd\tau}{(\tau^2+(1+|z|^2)^2)^{n+1}}\\
&\hspace{2mm}
+C\int_{B_{\epsilon^{-1}}(0)}\frac{\epsilon^3(|z|^4+\tau^2)^{\frac{3}{4}}dzd\tau}{(\tau^2+(1+|z|^2)^2)^{n+1}}
+O(\epsilon^{2n})+o(1)\epsilon\|v-1\|_{S_1^2(S^{2n+1},\theta_0)}\\
&=\epsilon^2A_6\Delta_{\theta_0}f(p)+O(\epsilon^{4})+o(1)\epsilon\|v-1\|_{S_1^2(S^{2n+1},\theta_0)}
\end{split}
\end{equation*}
where $A_6$ is given by
\begin{equation}\label{7.38}
A_6:=\frac{1}{2n}\int_{\mathbb{H}^n}\frac{4^{n+1}|z|^2dzd\tau}{(\tau^2+(1+|z|^2)^2)^{n+1}}.
\end{equation}
This yields the first assertion.\\
(b) Note that
$$E_f(u)=\frac{E(u)}{(\int_{S^{2n+1}}fu^{2+\frac{2}{n}}dV_{\theta_0})^{\frac{n}{n+1}}}
=\frac{E(v)}{(\int_{S^{2n+1}}f\circ\phi_{-p,\epsilon}v^{2+\frac{2}{n}}dV_{\theta_0})^{\frac{n}{n+1}}}.$$
Thus it follows that
$$\frac{\partial}{\partial\epsilon}E_f(u)=-\frac{n}{n+1}E(v)A^{-\frac{2n+1}{n+1}}\frac{\partial}{\partial\epsilon}
\int_{S^{2n+1}}f\circ\phi_{-p,\epsilon}dV_{h}.$$
Denote $\phi_{-p,\epsilon}$ by $\Psi_\epsilon$, as in (a), we have
\begin{equation*}
\begin{split}
\frac{\partial}{\partial\epsilon}
\int_{S^{2n+1}}f\circ\phi_{-p,\epsilon}dV_{h}
&=\int_{\mathbb{H}^n}\frac{\partial}{\partial\epsilon}f(\Psi_\epsilon(z,\tau))\frac{4^{n+1}dzd\tau}{(\tau^2+(1+|z|^2)^2)^{n+1}}\\
&\hspace{2mm}
+\int_{\mathbb{H}^n}\frac{\partial}{\partial\epsilon}f(\Psi_\epsilon(z,\tau))(v^{2+\frac{2}{n}}-1)\frac{4^{n+1}dzd\tau}{(\tau^2+(1+|z|^2)^2)^{n+1}}\\
&=I+II.
\end{split}
\end{equation*}

First note that
\begin{equation*}
\begin{split}
\frac{\partial}{\partial\epsilon}f(\Psi_\epsilon(z,\tau))&=\sum_{i=1}^n\frac{\partial f}{\partial x_i}\circ\Psi_\epsilon(z,\tau)
\frac{2z_i(1-\epsilon^2|z|^2+\sqrt{-1}\epsilon^2\tau)}{(1+\epsilon^2|z|^2-\sqrt{-1}\epsilon^2\tau)^2}\\
&\hspace{2mm}+\frac{\partial f}{\partial x_{n+1}}\circ\Psi_\epsilon(z,\tau)
\frac{-4\epsilon(|z|^2-\sqrt{-1}\tau)}{(1+\epsilon^2|z|^2-\sqrt{-1}\epsilon^2\tau)^2}.
\end{split}
\end{equation*}
Since $\displaystyle\sum_{i=1}^{n+1}\left|\frac{\partial f}{\partial x_i}\right|^2=\frac{4}{(1+|z|^2)^2+\tau^2}|\nabla_{\theta_0} f|^2_{\theta_0}$, we have
\begin{equation}\label{7.39}
\left|\frac{\partial}{\partial\epsilon}f(\Psi_\epsilon(z,\tau))\right|\leq C(|z|^4+\tau^2)^{\frac{1}{4}}\hspace{2mm}\mbox{ for }(z,\tau)\in\mathbb{H}^n.
\end{equation}
Now using the expansion of $f$ in (\ref{6.53}), we have
\begin{equation*}
\begin{split}
\frac{\partial}{\partial\epsilon}f(\Psi_\epsilon(z,\tau))&=
\sum_{i=1}^n\left(\frac{\partial f(p)}{\partial a_i}a_i
+\frac{\partial f(p)}{\partial b_i}b_i\right)+2\epsilon\frac{\partial f(p)}{\partial \tau}\tau+2\epsilon^3\frac{\partial^2 f(p)}{\partial\tau^2}\tau^2\\
&\hspace{4mm}
+\epsilon\sum_{i,j=1}^n\left(\frac{\partial^2 f(p)}{\partial a_i\partial a_j}a_ia_j
+\frac{\partial^2 f(p)}{\partial b_i\partial b_j}b_ib_j\right)\\
&\hspace{4mm}
+\frac{3}{2}\epsilon^2\sum_{i=1}^n\left(\frac{\partial^2 f(p)}{\partial a_i\partial\tau}a_i\tau
+\frac{\partial^2 f(p)}{\partial b_i\partial \tau}b_i\tau\right)
+O(\epsilon^2(|z|^4+\tau^2)^{\frac{3}{4}})
\end{split}
\end{equation*}
in $B_{\epsilon^{-1}}(0)$. By symmetry, we obtain
\begin{equation*}
\begin{split}
I&=\frac{\epsilon}{2n}\Delta_{\theta_0}f(p)\int_{\mathbb{H}^n}\frac{4^{n+1}|z|^2dzd\tau}{(\tau^2+(1+|z|^2)^2)^{n+1}}+O(\epsilon^{2n})\\
&\hspace{2mm}
+C\epsilon^3\int_{B_{\epsilon^{-1}}(0)}\frac{\tau^2dzd\tau}{(\tau^2+(1+|z|^2)^2)^{n+1}}
+C\epsilon^2\int_{B_{\epsilon^{-1}}(0)}\frac{(|z|^4+\tau^2)^{\frac{3}{4}}dzd\tau}{(\tau^2+(1+|z|^2)^2)^{n+1}}\\
&=\epsilon A_6\Delta_{\theta_0}f(p)+O(\epsilon^{2})
\end{split}
\end{equation*}
where $A_6$ is the constant defined in (\ref{7.38}). On the other hand, the expansion of $f$ in $B_{\epsilon^{-1}}(0)$ to the first order
\begin{equation*}
\begin{split}
&\frac{\partial}{\partial\epsilon}f(\Psi_\epsilon(z,\tau))=
\sum_{j=1}^n\left(\frac{\partial f(p)}{\partial a_j}a_j
+\frac{\partial f(p)}{\partial b_j}b_j\right)+O(\epsilon(|z|^4+\tau^2)^{\frac{1}{4}})\\
&=\sum_{j=1}^n\left[\left(\frac{\partial f(p)}{\partial a_j}-\frac{\partial f(p_i)}{\partial a_j}\right)a_j
+\left(\frac{\partial f(p)}{\partial b_j}-\frac{\partial f(p_i)}{\partial b_j}\right)b_j\right]+O(\epsilon\sqrt{|z|^4+\tau^2})
\end{split}
\end{equation*}
gives the uniform estimate
\begin{equation}\label{7.40}
\left|\frac{\partial}{\partial\epsilon}f(\Psi_\epsilon(z,\tau))\right|\leq C|p-p_i||z|+C\epsilon\sqrt{|z|^4+\tau^2}\hspace{2mm}\mbox{ in }\hspace{2mm}B_{\epsilon^{-1}}(0).
\end{equation}
By (\ref{7.39}) and (\ref{7.40}),  we get the estimate
\begin{equation*}
\begin{split}
|II|&\leq C\left[(\epsilon+|p-p_i|)\int_{B_{\epsilon^{-1}}(0)}|v^{2+\frac{2}{n}}-1|
\frac{(1+\sqrt{|z|^4+\tau^2})dzd\tau}{(\tau^2+(1+|z|^2)^2)^{n+1}}\right.\\
&\hspace{1cm}\left.+\int_{\mathbb{H}^n\backslash B_{\epsilon^{-1}}(0)}|v^{2+\frac{2}{n}}-1|
\frac{(|z|^4+\tau^2)^{\frac{1}{4}}dzd\tau}{(\tau^2+(1+|z|^2)^2)^{n+1}}\right]\\
&\leq C(\epsilon+|p-p_i|)\|v-1\|_{S_1^2(S^{2n+1},\theta_0)}
\end{split}
\end{equation*}
by (\ref{6.56}) and (\ref{6.3}).
Thus, (\ref{7.34}) follows from the estimates above and (a). Moreover, if $\Delta_{\theta_0}f(p)>0$, the lower bound of
$E(u)=E(v)\geq R_{\theta_0}\mbox{Vol}(S^{2n+1},\theta_0)^{\frac{1}{n+1}}=\displaystyle\frac{n(n+1)}{2}\mbox{Vol}(S^{2n+1},\theta_0)^{\frac{1}{n+1}}$ in view of Lemma 2.3 in \cite{Ho3} and $u\in B_{r_0}(p_i)\subset C^\infty_*$, together with (\ref{7.34})
derive the estimate (\ref{7.35}).\\
(c) For any $q\in T_p(S^{2n+1})$, as shown in (\ref{6.53}), we obtain the expansion of $d(f\circ\phi_{-p,\epsilon})\cdot q$ around $p$
as
\begin{equation*}
\begin{split}
&d(f\circ\phi_{-p,\epsilon})\cdot q-df(p)\cdot q\\
&=
\frac{\epsilon}{2}\sum_{i,j=1}^n\left(\frac{\partial^2 f(p)}{\partial a_i\partial a_j}\Big|_{(z,\tau)=(0,0)}a_i\widetilde{a}_j
+\frac{\partial^2 f(p)}{\partial b_i\partial b_j}\Big|_{(z,\tau)=(0,0)}b_i\widetilde{b}_j\right)\\
&\hspace{4mm}
+\frac{1}{2}\sum_{i=1}^n\left(\frac{\partial^2 f(p)}{\partial a_i\partial\tau}\Big|_{(z,\tau)=(0,0)}(\epsilon\widetilde{a_i}\tau+\epsilon^2a_i\widetilde{\tau})
+\frac{\partial^2 f(p)}{\partial b_i\partial \tau}\Big|_{(z,\tau)=(0,0)}(\epsilon\widetilde{b_i}\tau+\epsilon^2b_i\widetilde{\tau})\right)\\
&\hspace{4mm}+\frac{\epsilon^2}{2}\frac{\partial^2 f(p)}{\partial\tau^2}\Big|_{(z,\tau)=(0,0)}\tau\widetilde{\tau}
+O(\epsilon^2(|z|^4+\tau^2)^{\frac{3}{4}})
\end{split}
\end{equation*}
where $q=(\widetilde{a}_1,...,\widetilde{a}_n,\widetilde{b}_1,...,\widetilde{b}_n,\widetilde{\tau})\in T_p(S^{2n+1})$.
Observe that
\begin{equation*}
\begin{split}
&\frac{\partial E_f(u)}{\partial p}\cdot q+\frac{n}{n+1}E(v)A^{-\frac{2n+1}{n+1}}\mbox{Vol}(S^{2n+1},\theta_0)df(p)\cdot q\\
&=-\frac{n}{n+1}E(v)A^{-\frac{2n+1}{n+1}}\int_{S^{2n+1}}(d(f\circ\phi_{-p,\epsilon})\cdot q-df(p)\cdot q)v^{2+\frac{2}{n}}dV_{\theta_0}\\
&=-\frac{n}{n+1}E(v)A^{-\frac{2n+1}{n+1}}\left[\int_{S^{2n+1}}(d(f\circ\phi_{-p,\epsilon})\cdot q-df(p)\cdot q)dV_{\theta_0}\right.\\
&\hspace{4mm}+\left.\int_{S^{2n+1}}(d(f\circ\phi_{-p,\epsilon})\cdot q-df(p)\cdot q)(v^{2+\frac{2}{n}}-1)dV_{\theta_0}\right]\\
&=-\frac{n}{n+1}E(v)A^{-\frac{2n+1}{n+1}}(I_1+I_2),
\end{split}
\end{equation*}
then the assertion (\ref{7.36}) follows by
$$|I_1|\leq C\epsilon^2|q|\hspace{2mm}\mbox{ and }\hspace{2mm}|I_2|\leq C\epsilon\|v-1\|_{S_1^2(S^{2n+1},\theta_0)}|q|,$$
as similarly obtained in (a).\\
(d) By a direct computation, we have
\begin{equation*}
\begin{split}
\left\langle\frac{\partial}{\partial v}E_f(u),v-1\right\rangle
&=2A^{-\frac{n}{n+1}}\left[\int_{S^{2n+1}}\left((2+\frac{2}{n})|\nabla_{\theta_0}v|^2_{\theta_0}+R_{\theta_0}v(v-1)\right)dV_{\theta_0}\right.\\
&\hspace{2.2cm}\left.-E(v)A^{-1}\int_{S^{2n+1}}f\circ\phi_{-p,\epsilon} v^{\frac{n+2}{n}}(v-1)dV_{\theta_0}\right]\\
&=2A^{-\frac{n}{n+1}}\left[\int_{S^{2n+1}}R_{\theta_0}(v-1)^2dV_{\theta_0}-\int_{S^{2n+1}}R_{\theta_0}(v^{\frac{n+2}{n}}-1)(v-1)dV_{\theta_0}\right.\\
&\hspace{2.2cm}\left.+(2+\frac{2}{n})\int_{S^{2n+1}}|\nabla_{\theta_0}v|^2_{\theta_0}dV_{\theta_0}+I\right],
\end{split}
\end{equation*}
where
\begin{equation}\label{7.41}
\begin{split}
I&=-\int_{S^{2n+1}}(E(v)A^{-1}f\circ\phi_{-p,\epsilon}-R_{\theta_0})v^{\frac{n+2}{n}}(v-1)dV_{\theta_0}\\
&=-(E(v)-R_{\theta_0}\mbox{Vol}(S^{2n+1},\theta_0))A^{-1}\int_{S^{2n+1}}f\circ\phi_{-p,\epsilon}v^{\frac{n+2}{n}}(v-1)dV_{\theta_0}\\
&\hspace{2mm}-R_{\theta_0}A^{-1}\int_{S^{2n+1}}(f\circ\phi_{-p,\epsilon}\cdot\mbox{Vol}(S^{2n+1},\theta_0)-A)v^{\frac{n+2}{n}}(v-1)dV_{\theta_0}\\
&=I_1+I_2.
\end{split}
\end{equation}
In the following, we use the notation in the proof of Lemma \ref{lem6.6}. By the identity
\begin{equation}\label{7.42}
v^{\frac{n+2}{n}}-1=\frac{n+2}{n}(v-1)+o(|v-1|),
\end{equation}
and (\ref{6.36}), we obtain
\begin{equation*}
\begin{split}
&\int_{S^{2n+1}}R_{\theta_0}(v-1)^2dV_{\theta_0}-\int_{S^{2n+1}}R_{\theta_0}(v^{\frac{n+2}{n}}-1)(v-1)dV_{\theta_0}
+(2+\frac{2}{n})\int_{S^{2n+1}}|\nabla_{\theta_0}v|^2_{\theta_0}dV_{\theta_0}\\
&=(2+\frac{2}{n})\int_{S^{2n+1}}|\nabla_{\theta_0}v|^2_{\theta_0}dV_{\theta_0}-\frac{2R_{\theta_0}}{n}\int_{S^{2n+1}}(v-1)^2dV_{\theta_0}
+o(1)\|v-1\|^2_{S_1^2(S^{2n+1},\theta_0)}\\
&=(2+\frac{2}{n})\left(\sum_{i=1}^\infty\lambda_i|v^i|^2-\frac{n}{2}\sum_{i=0}^\infty|v^i|^2\right)+o(1)\|v-1\|^2_{S_1^2(S^{2n+1},\theta_0)}\\
&\geq(2+\frac{2}{n})\left(\frac{\lambda_{2n+3}-n/2}{\lambda_{2n+3}+1}\right)\sum_{i=2n+3}^\infty(\lambda_{i}+1)|v^i|^2+o(1)\|v-1\|^2_{S_1^2(S^{2n+1},\theta_0)}\\
&\geq C_0\|v-1\|^2_{S_1^2(S^{2n+1},\theta_0)}.
\end{split}
\end{equation*}
On the other hand, we can estimate (\ref{7.41}) as follows. By (\ref{7.33}), (\ref{7.42}), and Lemma \ref{lem7.3}, we have
$$|I_1|\leq C\|v-1\|_{S_1^2(S^{2n+1},\theta_0)}^3=o(1)\|v-1\|_{S_1^2(S^{2n+1},\theta_0)}^2.$$
By (\ref{6.55}), (\ref{7.33}), and the fact that
$$|df(p)|\rightarrow 0\hspace{2mm}\mbox{ as }r_0\rightarrow 0,$$
we also have
\begin{equation*}
\begin{split}
|I_2|&\leq C\left|\mbox{Vol}(S^{2n+1},\theta_0)\int_{S^{2n+1}}(f\circ\phi_{-p,\epsilon}-f(p))v^{\frac{n+2}{n}}(v-1)dV_{\theta_0}\right|\\
&\hspace{2mm}+C\left|(f(p)\mbox{Vol}(S^{2n+1},\theta_0)-A)\int_{S^{2n+1}}v^{\frac{n+2}{n}}(v-1)dV_{\theta_0}\right|\\
&\leq C\big(\|f\circ\phi_{-p,\epsilon}-f(p)\|_{L^2(S^{2n+1},\theta_0)}+|f(p)\mbox{Vol}(S^{2n+1},\theta_0)-A|\big)\cdot\|v-1\|_{S_1^2(S^{2n+1},\theta_0)}\\
&\leq(o(1)\epsilon+C\epsilon^2)\|v-1\|_{S_1^2(S^{2n+1},\theta_0)}=o(1)\epsilon\|v-1\|_{S_1^2(S^{2n+1},\theta_0)}.
\end{split}
\end{equation*}
Therefore, the above estimates yields (\ref{7.37}).
\end{proof}

Now we are going to complete the proof of Proposition \ref{prop7.1} by proving part (iii) and (iv).
Recall our convention in the proof of part (ii). Choose $\nu\leq r_0^3\leq\nu_0$ and $r_0>0$ sufficiently small
such that $B_{r_0}(p_i)\subset L_{\beta_i+\nu}\setminus L_{\beta_i-\nu}$. Similar to (ii), for any $1\leq i\leq N$ and a sufficient large $T>0$,
we can show that $u(T,L_{\beta_i+\nu_0})\subset L_{\beta_i+\nu}$. In addition, for any $u_0\in L_{\beta_i+\nu_0}$, if necessary, choosing a larger
$T=T(u_0)>0$, we either have $u(T,u_0)\in L_{\beta_i-\nu_0}$ or $u(t,u_0)\in B_{r_0/4}(p_i)$ for some $t\in [0,T]$.

For $u=(\epsilon, p, \nu)\in B_{r_0}(p_i)$, we have
\begin{equation}\label{7.43}
\begin{split}
&E_f(u)-\beta_i\\
&=\frac{E(u)}{(\int_{S^{2n+1}}fdV_{\theta})^{\frac{n}{n+1}}}-R_{\theta_0}\mbox{Vol}(S^{2n+1},\theta_0)^{\frac{1}{n+1}}f(p_i)^{-\frac{n}{n+1}}\\
&=\frac{E(v)}{(\int_{S^{2n+1}}f\circ\phi_{-p,\epsilon}dV_{h})^{\frac{n}{n+1}}}-R_{\theta_0}\mbox{Vol}(S^{2n+1},\theta_0)^{\frac{1}{n+1}}f(p_i)^{-\frac{n}{n+1}}\\
&=A^{-\frac{n}{n+1}}\Big[\Big(E(v)-R_{\theta_0}\mbox{Vol}(S^{2n+1},\theta_0)\Big)\\
&\hspace{4mm}-R_{\theta_0}f(p_i)^{-\frac{n}{n+1}}\mbox{Vol}(S^{2n+1},\theta_0)^{\frac{1}{n+1}}\Big(A^{\frac{n}{n+1}}
-f(p_i)^{\frac{n}{n+1}}\mbox{Vol}(S^{2n+1},\theta_0)^{\frac{n}{n+1}}\Big)\Big]
\end{split}
\end{equation}
where $A=\displaystyle\int_{S^{2n+1}}f\circ\phi_{-p,\epsilon}dV_{h}$. Note that
\begin{equation*}
\begin{split}
&A^{\frac{n}{n+1}}-f(p_i)^{\frac{n}{n+1}}\mbox{Vol}(S^{2n+1},\theta_0)^{\frac{n}{n+1}}\\
&=f(p_i)^{\frac{n}{n+1}}\mbox{Vol}(S^{2n+1},\theta_0)^{\frac{n}{n+1}}
\left[\left(1+\frac{A-f(p_i)\mbox{Vol}(S^{2n+1},\theta_0)}{f(p_i)\mbox{Vol}(S^{2n+1},\theta_0)}\right)^{\frac{n}{n+1}}-1\right]\\
&=f(p_i)^{\frac{n}{n+1}}\mbox{Vol}(S^{2n+1},\theta_0)^{\frac{n}{n+1}}\left[\frac{n}{n+1}
\frac{A-f(p_i)\mbox{Vol}(S^{2n+1},\theta_0)}{f(p_i)\mbox{Vol}(S^{2n+1},\theta_0)}+O(|A-f(p_i)\mbox{Vol}(S^{2n+1},\theta_0)|^2)\right]
\end{split}
\end{equation*}
and
$$A-f(p_i)\mbox{Vol}(S^{2n+1},\theta_0)=A-f(p)\mbox{Vol}(S^{2n+1},\theta_0)+\mbox{Vol}(S^{2n+1},\theta_0)\big(f(p)-f(p_i)\big),$$
together with Lemma \ref{lem7.4}(a), we find that
\begin{equation}\label{7.44}
\begin{split}
&f(p_i)^{\frac{1}{n+1}}\big[A^{\frac{n}{n+1}}
-f(p_i)^{\frac{n}{n+1}}\mbox{Vol}(S^{2n+1},\theta_0)^{\frac{n}{n+1}}\big]\\
&=\mbox{Vol}(S^{2n+1},\theta_0)^{-\frac{1}{n+1}}\frac{n}{n+1}A_6\epsilon^2\Delta_{\theta_0}f(p)
+\mbox{Vol}(S^{2n+1},\theta_0)^{\frac{n}{n+1}}\frac{n}{n+1}(|p^+|^2-|p^-|^2)\\
&\hspace{4mm}+o(1)(\epsilon^2+|p-p_i|^2+\|v-1\|^2_{S_1^2(S^{2n+1},\theta_0)}).
\end{split}
\end{equation}
Hence, from (\ref{7.43}) and Lemma \ref{lem7.3}, we conclude that
$$E_f(u)-\beta_i\geq C\|v-1\|_{S_1^2(S^{2n+1},\theta_0)}-C(\epsilon^2+|p-p_i|^2).$$
Consequently, for $u\in L_{\beta_i+\nu}\cap B_{r_0}(p_i)$, we have
\begin{equation}\label{7.45}
\|v-1\|_{S_1^2(S^{2n+1},\theta_0)}\leq C(\epsilon^2+|p-p_i|^2+r_0^3).
\end{equation}

Now we still use the same normalization (\ref{7.32.5}) in $t$ used in the proof of part (ii). Now with this scale, (\ref{7.44}), Proposition \ref{prop2.2},
Lemma \ref{lem6.5} and \ref{lem6.7} yield that
\begin{equation*}
\begin{split}
\frac{d}{d\tau}E_f(U(\tau,u_0))&=\epsilon^{-2}\frac{d}{dt}E_f(u(t(\tau),u_0))\\
&\leq-C_3(|f'(p)|^2+\epsilon^2|\Delta_{\theta_0}f(p)|^2)\\
&\leq-C_4(\epsilon^2+|p-p_i|^2),
\end{split}
\end{equation*}
with uniform constants $C_3>0$, $C_4>0$ and for all $u_0\in B_{r_0}(p_i)$. Note that the last inequality holds because with the coordinates we chose,
$|f'(p)|^2=|p-p_i|^2$ and observe that the non-degeneracy condition implies that $|\Delta_{\theta_0}f(p)|>0$ if
$r_0$ is sufficiently small since $p_i$ is a critical point of $f$.

Thus for each $u_0\in B_{r_0}\backslash B_{r_0/4}(p_i)$, we have
\begin{equation}\label{7.46}
\frac{d}{d\tau}E_f(U(\tau,u_0))\leq -C_5r_0^2,
\end{equation}
with a uniform constant $C_5>0$ in view of (\ref{7.45}). Hence, transversal time of the annular region
$L_{\beta_i+\nu}\cap(B_{r_0/2}\backslash B_{r_0/4}(p_i))$ is uniformly positive. Choosing sufficiently large $T^*>0$ and sufficiently small
$\nu>0$, we have
\begin{equation}\label{7.47}
U(T^*,L_{\beta_i+\nu})\subset L_{\beta_i-\nu}\cup(B_{r_0/2}(p_i)\cap L_{\beta_i+\nu}).
\end{equation}
Then
$$T_\nu(u_0)=\min\{T^*,\inf\{t:E_f(U(t,u_0))\leq\beta_i-\nu\}\}$$
continuously depends on $u_0$. Thus the map $(t,u_0)\mapsto U(\min\{t,T_\nu(u_0)\},u_0)$ gives a homotopy equivalence
of $L_{\beta_i+\nu}$ with a subset of $L_{\beta_i-\nu}\cup(B_{r_0/2}(p_i)\cap L_{\beta_i+\nu})$.

With all these preparations, now we are ready to prove part (iii) and (iv).

\begin{proof}[Proof of Proposition \ref{prop7.1} \emph{(iii)}]
Assume $\Delta_{\theta_0}f(p_i)>0$. For $u=(\epsilon,p,v)\in B_{r_0}(p_i)$, denote the vector field $X(u)$ on $B_{r_0}(p_i)$
by setting
$$X(u)=(1,0,0).$$
Then let $G(u,s)$ be the solution of the flow equation
$$\frac{d}{ds}G(u,s)=X(G(u,s)),$$
with initial data $G(u,0)=u$. Since $X$ is transversal to $\partial B_{r_0}(p_i)$ and $G(u,r_0)\notin B_{r_0}(p_i)$, there exists
a first time $0\leq s=s(u)\leq r_0$ such that $G(u,s(u))\notin B_{r_0}(p_i)$ and furthermore the map $u\mapsto s(u)$ is continuous.
Then $H(u,s)=G(u,\min\{s,s(u)\})$ defines a homotopy $H:\overline{B_{r_0}(p_i)}\times[0,r_0]\rightarrow\overline{B_{r_0}(p_i)}$ such that
$$H(B_{r_0}(p_i),r_0)\subset \partial B_{r_0}(p_i)\hspace{2mm}\mbox{ and }\hspace{2mm} H(\cdot,s)|_{\partial B_{r_0}(p_i)}=id,\hspace{2mm}0\leq s\leq r_0.$$
Then by (\ref{7.35}), letting $u_s=H(u,s)$, we have
\begin{equation*}
\begin{split}
\frac{d}{ds}E_f(u_s)&=dE_f(u_s)\cdot X(u_s)=\frac{\partial}{\partial\epsilon}E_f(u_s)\\
&\leq -\frac{n^2}{2}f(p)^{-\frac{2n+1}{n+1}}\mbox{Vol}(S^{2n+1},\theta_0)^{-\frac{2n}{n+1}}
\epsilon A_6 \Delta_{\theta_0}f(p)+o(r_0).
\end{split}
\end{equation*}
It follows that there exists a uniform constant $C_5>0$ such that
$$E_f(H(u,r))\leq E_f(u)-C_5r_0^2\leq \beta_i-\nu\hspace{2mm}\mbox{ for all }u\in B_{r_0/2}(p_i)\cap L_{\beta_i+\nu}$$
if $r_0>0$ is sufficiently small. Composing $H$ with the flow $(t,u_0)\mapsto U(\min\{t,T_\nu(u_0)\},u_0)$, we then obtain a homotpy
$K:L_{\beta_i+\nu_0}\times[0,1]\rightarrow L_{\beta_i+\nu_0}$ such that $K(L_{\beta_i+\nu_0},1)\subset L_{\beta_i+\nu_0}$. Moreover,
by the choice of $r_0>0$, it is easy to show that
$$K(\cdot, r)|_{L_{\beta_i-\nu}}=id\hspace{2mm}\mbox{ for }0\leq r\leq 1.$$
Finally, for each $u_0\in L_{\beta_i-\nu}$, let $T_{\nu_0}(u_0)=\inf\{t\geq 0: E_f(U(t,u_0))\leq \beta_i-\nu_0\}.$ As in the proof of part (ii), the number
$T_{\nu_0}(u_0)$ are uniformly bounded and continuously depend on $u_0$. By composing $K$ with the flow $(t,u)\rightarrow U(\min\{t,T_{\nu_0}(u_0)\},u_0)$,
we therefore obtain a homotopy equivalence of $L_{\beta_i+\nu}$ with $L_{\beta_i-\nu}$. This finishes the proof
of part (iii).
\end{proof}

\begin{proof}[Proof of Proposition \ref{prop7.1} \emph{(iv)}]
Suppose $\Delta_{\theta_0}f(p_i)<0$. From (\ref{7.43}), (\ref{7.44}), and with the constant $C_2$ in Lemma \ref{lem7.3},
we find that
\begin{equation*}
\begin{split}
E_f(u)-\beta_i&\geq A^{-\frac{n}{n+1}}\left[C_2\|v-1\|_{S_1^2(S^{2n+1},\theta_0)}^2-\frac{n}{n+1}R_{\theta_0}A_6\epsilon^2f(p_i)^{-1}\Delta_{\theta_0}f(p)\right.\\
&\hspace{2cm}+\frac{n}{n+1}R_{\theta_0}\mbox{Vol}(S^{2n+1},\theta_0)f(p_i)^{-1}(|p^-|^2-|p^+|^2)\\
&\hspace{2cm}\left.\vphantom{\frac{n}{n+1}}+o(1)(\epsilon^2+|p-p_i|^2+\|v-1\|^2_{S_1^2(S^{2n+1},\theta_0)})
\right],
\end{split}
\end{equation*}
where $o(1)\to 0$ as $r_0\to 0$. Then we deduce that there exists some number $\delta>0$
with $4\delta^2<\displaystyle\frac{7}{16}\min\{1,r_0^2\}$ such that
\begin{equation}\label{7.48}
\epsilon^2+|p^-|^2+\|v-1\|^2_{S_1^2(S^{2n+1},\theta_0)}\leq r_0^2/4,
\end{equation}
for any $u=(\epsilon,p,v)\in B_{r_0}(p_i)\cap L_{\beta_i+\nu}$ with $|p^+|<2\delta r_0$, provided $r_0>0$ is sufficiently small
and $\nu\leq r_0^3$.

Let $a_+=\max\{a,0\}$ for $a\in\mathbb{R}$. We construct a cut-off function $\eta$ defined by
$\eta=\eta(|p^+|)=\displaystyle\Big(1-\frac{(|p^+|-\delta r_0)_+}{\delta r_0}\Big)_+$ with $\delta>0$ given as above.
For $0\leq r\leq 1$, $u=(\epsilon,p,v)\in B_{r_0}(p_i)$, choose $\epsilon_0>0$ sufficiently small such that $0<\displaystyle\frac{1}{3}\epsilon<\epsilon_0<\frac{2}{3}\epsilon$, and define $u_r$ by
$$u_r=(\epsilon_r,p_r,v_r)=(\epsilon+(\epsilon_0-\epsilon)r\eta,p-r\eta p^-,((1-r\eta)v^{2+\frac{2}{n}}+r\eta)^{\frac{n}{2n+2}}).$$

First we claim that if $\|v-1\|_{C^1_P(S^{2n+1})}$ is sufficiently small, then $u_r\in B_{r_0}(p_i)$. To see this, we first
consider the function $g(r)$ with $\eta=1$:
$$g(r)=(\epsilon+(\epsilon_0-\epsilon)r)^2+|p-r p^-|^2+\|v_r-1\|^2_{S_1^2(S^{2n+1},\theta_0)}.$$
Then we have
\begin{equation*}
\begin{split}
g'(r)&=2(\epsilon+(\epsilon_0-\epsilon)r)(\epsilon_0-\epsilon)+2\langle p-r p^-,-p^-\rangle\\
&\hspace{4mm}+2\int_{S^{2n+1}}\left[\langle\nabla_{\theta_0}v_r,\nabla_{\theta_0}\frac{dv_r}{dr}\rangle_{\theta_0}+(v_r-1)\frac{dv_r}{dr}\right]dV_{\theta_0}.
\end{split}
\end{equation*}
A simple calculation gives $\displaystyle\frac{d^2v_r}{dr^2}=-\frac{n+2}{n}v_r^{-1}\left(\frac{dv_r}{dr}\right)^2$. Thus we have
\begin{equation*}
\begin{split}
g''(r)
&=2(\epsilon_0-\epsilon)^2+2|p^-|^2
+2\int_{S^{2n+1}}\left[\left|\nabla_{\theta_0}\frac{dv_r}{dr}\right|_{\theta_0}+\left(\frac{dv_r}{dr}\right)^2\right]dV_{\theta_0}
\\
&\hspace{4mm}+2\int_{S^{2n+1}}\left[\langle\nabla_{\theta_0}v_r,\nabla_{\theta_0}\frac{d^2v_r}{dr^2}\rangle_{\theta_0}+(v_r-1)\frac{d^2v_r}{dr^2}\right]dV_{\theta_0}\\
&=2(\epsilon_0-\epsilon)^2+2|p^-|^2
+2\int_{S^{2n+1}}\left[\left|\nabla_{\theta_0}\frac{dv_r}{dr}\right|_{\theta_0}+\left(\frac{dv_r}{dr}\right)^2\right]dV_{\theta_0}
\\
&\hspace{4mm}-\frac{2(n+2)}{n}\int_{S^{2n+1}}\left[\langle\nabla_{\theta_0}v_r,\nabla_{\theta_0}\left(v_r^{-1}\left(\frac{dv_r}{dr}\right)^2\right)\rangle_{\theta_0}
+(v_r-1)v_r^{-1}\left(\frac{dv_r}{dr}\right)^2\right]dV_{\theta_0}\\
&=2(\epsilon_0-\epsilon)^2+2|p^-|^2
+2\int_{S^{2n+1}}\left[\left|\nabla_{\theta_0}\frac{dv_r}{dr}\right|_{\theta_0}+\left(\frac{dv_r}{dr}\right)^2\right]dV_{\theta_0}
\\
&\hspace{4mm}+\frac{2(n+2)}{n}\int_{S^{2n+1}}|\nabla_{\theta_0}v_r|^2_{\theta_0}v_r^{-2}\left(\frac{dv_r}{dr}\right)^2dV_{\theta_0}\\
&\hspace{4mm}-\frac{2(n+2)}{n}\int_{S^{2n+1}}\left[\frac{n}{n+1}v_r^{-(2+\frac{2}{n})}\frac{dv_r}{dr}
\langle\nabla_{\theta_0}v_r^{2+\frac{2}{n}},\nabla_{\theta_0}\frac{dv_r}{dr}\rangle_{\theta_0}
+(v_r-1)v_r^{-1}\left(\frac{dv_r}{dr}\right)^2\right]dV_{\theta_0}.
\end{split}
\end{equation*}
Now observe that $|v_r^{-(2+\frac{2}{n})}\nabla_{\theta_0}v_r^{2+\frac{2}{n}}|_{\theta_0}=o(1)$ and $(v_r-1)v_r^{-1}=o(1)$
if $\|v-1\|_{C^1_P(S^{2n+1})}$ is sufficiently small. Hence, using the H\"{o}lder's and Young's inequality, we get
\begin{equation*}
\begin{split}
g''(r)&\geq 2(\epsilon_0-\epsilon)^2+2|p^-|^2
+(2+o(1))\int_{S^{2n+1}}\left[\left|\nabla_{\theta_0}\frac{dv_r}{dr}\right|_{\theta_0}+\left(\frac{dv_r}{dr}\right)^2\right]dV_{\theta_0}\geq 0.
\end{split}
\end{equation*}
This shows that $g''(r)\geq 0$ for all $r\in [0,1]$. Thus we conclude that
$$g(r)\leq\max\{g(0),\,g(1)\}=\max\{r_0^2,\,\epsilon_0^2+|p^+|^2\}.$$
Now note that if $\eta r>0$, then $\eta>0$, hence $|p^+|<2\delta r_0$ by the definition of $\eta$. Thus by the estimate (\ref{7.48}), we have
$\epsilon^2\leq\displaystyle\frac{r_0^2}{4}$. This implies that $\epsilon_0^2\leq\displaystyle\frac{r_0^2}{9}$ and $4\delta^2\leq\displaystyle 1-\frac{2}{9}$, we have $g(1)\leq\epsilon_0^2+|p^+|^2\leq\displaystyle(\frac{1}{9}+4\delta^2)r_0^2\leq r_0^2.$ Therefore, we have
$g(\eta r)\leq \max\{g(0),\,g(1)\}\leq r_0^2$.

Therefore under smallness condition of $\|u-1\|_{C^1_P(S^{2n+1})}$, (which can be guaranteed by the construction of homotopies below), one has shown that
the homotopy $H_1: \overline{B_{r_0}(p_i)}\cap L_{\beta_i+\nu}\times[0,1]\rightarrow\overline{B_{r_0}(p_i)}$ given by $H_1(u,r)=u_r$ is well defined
and $H_1(\cdot, 1)$ maps the set $\{u\in B_{r_0}(p_i)\cap L_{\beta_i+\nu}: |p^+|<\delta r_0\}$ to the set
$B^+_{\delta r_0}$, where for $0<\rho<r_0$,
$$B^+_\rho:=\{u\in B_{r_0}(p_i): \epsilon=\epsilon_0, p^-=0, |p^+|<\rho, v=1\}$$
which is diffeomorphic to the unit ball of dimension $2n+1-ind(f,p_i)$.

Now we need to show that the energy level of $u_r$ is under control; that is, $E_f(u_r)\leq\beta_i+\nu$ if $\nu$ is sufficiently small.
To do this, we observe that
\begin{equation}\label{7.51}
\begin{split}
\frac{d}{dr}E_f(u_r)&=\eta\left(\frac{\partial E_f(u_r)}{\partial\epsilon_r}(\epsilon_0-\epsilon)-\frac{\partial E_f(u_r)}{\partial p_r}p^-\right.\\
&\hspace{4mm}\left.-\frac{n}{2n+2}\left\langle\frac{\partial E_f(u_r)}{\partial v_r},v_r^{-\frac{n+2}{n}}(v^{2+\frac{2}{n}}-1)\right\rangle\right)\\
&=\eta(1-r\eta)^{-1}\left(\frac{\partial E_f(u_r)}{\partial\epsilon_r}(\epsilon_0-\epsilon_r)-\frac{\partial E_f(u_r)}{\partial p_r}p_r^-\right.\\
&\hspace{4mm}\left.-\frac{n}{2n+2}\left\langle\frac{\partial E_f(u_r)}{\partial v_r},v_r^{-\frac{n+2}{n}}(v_r^{2+\frac{2}{n}}-1)\right\rangle\right)\\
&:=\eta(1-r\eta)^{-1}D:=\eta(1-r\eta)^{-1}(I-II-III).
\end{split}
\end{equation}
First we deal with the last term:
\begin{equation*}
\begin{split}
III&:=\frac{n}{2n+2}\left\langle\frac{\partial E_f(u_r)}{\partial v_r},v_r^{-\frac{n+2}{n}}(v_r^{2+\frac{2}{n}}-1)\right\rangle\\
&=\frac{n}{n+1}A_r^{-\frac{n}{n+1}}\left[\int_{S^{2n+1}}(2+\frac{2}{n})\langle\nabla_{\theta_0}v_r,
\nabla_{\theta_0}(v_r^{-\frac{n+2}{n}}(v_r^{2+\frac{2}{n}}-1))\rangle_{\theta_0}\right.\\
&\hspace{4mm}\left.
+\int_{S^{2n+1}}R_{\theta_0}v_rv_r^{-\frac{n+2}{n}}(v_r^{2+\frac{2}{n}}-1)dV_{\theta_0}\right]\\
&\hspace{4mm}-\frac{n}{n+1}A_r^{-\frac{n}{n+1}}E(v_r)A_r^{-1}\int_{S^{2n+1}}f\circ\phi_{-p_r,\epsilon_r}(v_r^{2+\frac{2}{n}}-1)dV_{\theta_0}\\
&:=\frac{n}{n+1}A_r^{-\frac{n}{n+1}}(III_1+III_2),
\end{split}
\end{equation*}
where $A_r=\displaystyle\int_{S^{2n+1}}f\circ\phi_{-p_r,\epsilon_r}v_r^{2+\frac{2}{n}}dV_{\theta_0}$. Since $\|v-1\|_{C^1_P(S^{2n+1})}=o(1)$ is sufficiently small, it yields $\|v_r-1\|_{C^1_P(S^{2n+1})}=o(1)$. We can estimate $III_1$ as follows:
\begin{equation*}
\begin{split}
III_1&=(2+\frac{2}{n})^2\int_{S^{2n+1}}|\nabla_{\theta_0}v_r|^2_{\theta_0}dV_{\theta_0}+R_{\theta_0}\int_{S^{2n+1}}(v_r^{\frac{2}{n}}-1)(v_r^{2+\frac{2}{n}}-1)dV_{\theta_0}\\
&\hspace{4mm}-\frac{(2n+2)(n+2)}{n^2}\int_{S^{2n+1}}|\nabla_{\theta_0}v_r|^2_{\theta_0}(1-v_r^{-(2+\frac{2}{n})})dV_{\theta_0}\\
&=(2+\frac{2}{n})^2\left[\int_{S^{2n+1}}|\nabla_{\theta_0}v_r|^2_{\theta_0}dV_{\theta_0}
-\frac{n}{2}\int_{S^{2n+1}}(v_r-1)^2dV_{\theta_0}\right]\\
&\hspace{4mm}+o(1)\|v_r-1\|_{S_1^2(S^{2n+1},\theta_0)}^2.
\end{split}
\end{equation*}
where in the first equality we have used the fact that
$\displaystyle\int_{S^{2n+1}}(v_r^{2+\frac{2}{n}}-1)dV_{\theta_0}=(1-r\eta)\int_{S^{2n+1}}(v^{2+\frac{2}{n}}-1)dV_{\theta_0}=0$. Observe that the following estimate holds true by the same argument of proving (\ref{6.31}):
\begin{equation}\label{7.49}
\int_{S^{2n+1}}|\nabla_{\theta_0}v_r|^2_{\theta_0}dV_{\theta_0}
-\frac{n}{2}\int_{S^{2n+1}}(v_r-1)^2dV_{\theta_0}\geq\left[\lambda_{2n+3}-\frac{n}{2}+o(1)\right]\|v_r-1\|^2_{S_1^2(S^{2n+1},\theta_0)},
\end{equation}
which will be used in controlling $D$. For $III_2$, we can rewrite it as
\begin{equation*}
III_2=-E(v_r)\left(1-A_r^{-1}\int_{S^{2n+1}}f\circ\phi_{-p_r,\epsilon_r}dV_{\theta_0}\right).
\end{equation*}
Note that
\begin{equation*}
\begin{split}
&\int_{S^{2n+1}}f\circ\phi_{-p_r,\epsilon_r}dV_{\theta_0}-f(p_r)\mbox{Vol}(S^{2n+1},\theta_0)\\
&=\left(A_r-f(p_r)\mbox{Vol}(S^{2n+1},\theta_0)\right)+\int_{S^{2n+1}}f\circ\phi_{-p_r,\epsilon_r}(1-v_r^{2+\frac{2}{n}})dV_{\theta_0}\\
&=\left(A_r-f(p_r)\mbox{Vol}(S^{2n+1},\theta_0)\right)+\int_{S^{2n+1}}(f\circ\phi_{-p_r,\epsilon_r}-f(p_r))(1-v_r^{2+\frac{2}{n}})dV_{\theta_0}\\
&=\left(A_r-f(p_r)\mbox{Vol}(S^{2n+1},\theta_0)\right)+o(1)\epsilon_r\|v_r-1\|_{S_1^2(S^{2n+1},\theta_0)}
\end{split}
\end{equation*}
where the third equality follows from $\displaystyle\int_{S^{2n+1}}(v_r^{2+\frac{2}{n}}-1)dV_{\theta_0}=0$ and the last equality follows
from (\ref{7.50}). Hence,
\begin{equation*}
\begin{split}
&1-A_r^{-1}\int_{S^{2n+1}}f\circ\phi_{-p_r,\epsilon_r}dV_{\theta_0}\\
&=A_r^{-1}\left[A_r-f(p_r)\mbox{Vol}(S^{2n+1},\theta_0)
-\left(\int_{S^{2n+1}}f\circ\phi_{-p_r,\epsilon_r}dV_{\theta_0}-f(p_r)\mbox{Vol}(S^{2n+1},\theta_0)\right)\right]\\
&=o(1)\epsilon_r\|v_r-1\|_{S_1^2(S^{2n+1},\theta_0)}.
\end{split}
\end{equation*}
Combining these estimates, we obtain
\begin{equation*}
\begin{split}
III&=2(2+\frac{2}{n})A_r^{-\frac{n}{n+1}}\left[\int_{S^{2n+1}}|\nabla_{\theta_0}v_r|^2_{\theta_0}dV_{\theta_0}
-\frac{n}{2}\int_{S^{2n+1}}(v_r-1)^2dV_{\theta_0}\right]\\
&\hspace{4mm}+o(1)(\epsilon_r+\|v_r-1\|_{S_1^2(S^{2n+1},\theta_0)})\|v_r-1\|_{S_1^2(S^{2n+1},\theta_0)}.
\end{split}
\end{equation*}
For $I$ and $II$, we can apply Lemma \ref{lem7.4} to get
\begin{equation*}
II=-\frac{n}{n+1}A_r^{-\frac{n}{n+1}}E(v_r)\mbox{Vol}(S^{2n+1},\theta_0)^{\frac{2n+1}{n+1}}\frac{df(p_r)\cdot p_r^-}{A_r}
+C\epsilon_r(\epsilon_r+\|v_r-1\|_{S_1^2(S^{2n+1},\theta_0)})|p_r^-|
\end{equation*}
and
\begin{equation*}
\begin{split}
I&=-\frac{n}{n+1}A_r^{-\frac{n}{n+1}}E(v_r)
\epsilon_r A_6\frac{\Delta_{\theta_0}f(p_r)}{A_r}(\epsilon_0-\epsilon_r)\\
&\hspace{4mm}+C\big(\epsilon_r^{2}+(\epsilon_r+|p_r-p_i|)\|v_r-1\|_{S_1^2(S^{2n+1},\theta_0)}\big)(\epsilon_r-\epsilon_0).
\end{split}
\end{equation*}

Note that in the local coordinates of $p_i$, $f(p_r)=f(p_i)+|p_r^+|^2-|p_r^-|^2$ and
$df(p_r)\cdot p_r^-=-2|p_r^-|^2$. Therefore, combining the estimates of $I$, $II$,  $III$ and (\ref{7.49}), we obtain
\begin{equation*}
\begin{split}
D&\leq A_r^{-\frac{n}{n+1}}\left\{-\frac{nA_6}{n+1}E(v_r)
\epsilon_r(\epsilon_0-\epsilon_r) \frac{\Delta_{\theta_0}f(p_r)}{A_r}
-\frac{2n}{n+1}E(v_r)\mbox{Vol}(S^{2n+1},\theta_0)^{\frac{2n+1}{n+1}}\frac{|p_r^-|^2}{A_r}\right.\\
&\hspace{4mm}+C|p_r^+|\|v-1\|_{S_1^2(S^{2n+1},\theta_0)}(\epsilon_r-\epsilon_0)
-2(2+\frac{2}{n})\left[\lambda_{2n+3}-\frac{n}{2}+o(1)\right]\|v_r-1\|^2_{S_1^2(S^{2n+1},\theta_0)}\\
&\hspace{4mm}\left.\vphantom{\frac{|p_r^-|^2}{A_r}}+o(1)\big(\epsilon_r(\epsilon_r-\epsilon_0)+\|v-1\|_{S_1^2(S^{2n+1},\theta_0)}^2+|p_r^-|^2\big)\right\}.
\end{split}
\end{equation*}
Now set $$d=\min\left\{\min_{|p-p_i|\leq r_0}\left(-\frac{nA_6}{n+1}E(v_r)\frac{\Delta_{\theta_0}f(p_r)}{A_r}\right),\frac{2n}{n+1}\frac{E(v_r)}{A_r},
2(2+\frac{2}{n})(\lambda_{2n+3}-\frac{n}{2})\right\}.$$
Since $\Delta_{\theta_0}f(p_i)<0$, by continuity, when $r_0$ is sufficiently small, $\Delta_{\theta_0}f(p_r)<0$ if $|p_r-p_i|<r_0$. Note also that $A_6>0$ by (\ref{7.38}), we have $d>0$. We can rewrite the above estimate as
\begin{equation}\label{7.52}
\begin{split}
D&\leq A_r^{-\frac{n}{n+1}}\left\{-\frac{d}{2}\big(\epsilon_r(\epsilon_r-\epsilon_0)+\|v-1\|_{S_1^2(S^{2n+1},\theta_0)}^2+|p_r^-|^2\big)\right.\\
&\hspace{4mm}\left.\vphantom{\frac{d}{2}}+C|p_r^+|\|v-1\|_{S_1^2(S^{2n+1},\theta_0)}(\epsilon_r-\epsilon_0)\right\}:=A_r^{-\frac{n}{n+1}}D_1.
\end{split}
\end{equation}
If $\eta>0$, then by the definition of $\eta$ we have $|p^+|<2\delta r_0$. By definition of $p_r$, $|p_r^+|=|p^+|$, hence,  we have $|p_r^+|<2\delta r_0$. Now if we choose $r_0$ sufficiently small such that $2C\delta r_0<d$, we have
\begin{equation*}
D_1\leq -\frac{d}{2}\left\{\big(\epsilon_r(\epsilon_r-\epsilon_0)+\|v-1\|_{S_1^2(S^{2n+1},\theta_0)}^2+|p_r^-|^2\big)-2\|v-1\|_{S_1^2(S^{2n+1},\theta_0)}(\epsilon_r-\epsilon_0)\right\}.
\end{equation*}
Since $\epsilon_r(\epsilon_r-\epsilon_0)\geq (\epsilon_r-\epsilon_0)^2$, it is easy to see that $D_1<0$ when $\|v-1\|_{S_1^2(S^{2n+1},\theta_0)}$ is sufficiently small, which implies that $D<0$ by (\ref{7.52}). Hence, by (\ref{7.51}), we have
$\displaystyle\frac{d}{dr}E_f(u_r)\leq 0$. Note that when $r=0$, $u_r=u=(\epsilon,p,v)$. Since $u\in B_{r_0}(p_i)\cap L_{\beta_i+\nu}$, we have $E_f(u_r)\leq E_f(u)\leq \beta_i+\nu$ for all $r\in [0,1]$.

Moreover, from (\ref{7.48}) and our choice of $\delta$, we have
$$H(\cdot,r)|_{\partial B_{r_0}(p_i)\cap L_{\beta_i+\nu}}=id,\hspace{2mm}\mbox{ for }0\leq r\leq 1.$$
Denote the vector field $X_1(u)$ as
$$X_1(u)=(0,p^+,0),$$
and $G_1(u,s)$ solves the flow equation
$$\frac{d}{ds}G_1(u,s)=X_1(G_1(u,s)),\hspace{2mm}0\leq s\leq \delta^{-1},$$
with initial data $G_1(u,0)=u$. Notice that $X_1$ is transversal to $\partial B_{r_0}(p_i)$ within $L_{\beta_i+\nu}$; in addition, for any
$u\in B_{r_0}(p_i)\cap L_{\beta_i+\nu}$ with $|p^+|\geq \delta r_0$, there holds $G_1(u,\delta^{-1})\not\in B_{r_0}(p_i)$ for some sufficiently small $\delta>0$, then there exists a first time $0\leq s_1\leq r(u)\leq\delta^{-1}$ such that $G_1(u,s_1)\not\in B_{r_0}(p_i)$, and the map $u\mapsto s_1(u)$ is continuous.
We extend this map to whole set $B_{r_0}(p_i)\cap L_{\beta_i+\nu}$ by letting $s_1(u)=\delta^{-1}$ whenever $G_1(u,s)\in B_{r_0}(p_i)$
for all $s\in [0,r_1]$. Setting $H_2(u,s)=G_1(u,\min\{s,s_1(u)\})=u_s$, with a uniform $C>0$, we obtain by
(\ref{7.36})  that
\begin{equation*}
\begin{split}
\frac{d E_f(u_s)}{ds}&=\frac{\partial E_f(u_s)}{\partial p}\cdot p^+\\
&\leq-\frac{n}{n+1}E(u_s)f(p)^{-\frac{2n+1}{n+1}}df(p)\cdot p^++ C\epsilon(\epsilon+\|v-1\|_{S_1^2(S^{2n+1},\theta_0)})|p^+|\\
&\leq-\frac{2n}{n+1}R_{\theta_0}\mbox{Vol}(S^{2n+1},\theta_0)f(p)^{-\frac{2n+1}{n+1}}|p^+|^2+Cr_0^3\leq -C r_0^2
\end{split}
\end{equation*}
if $|p^+|\geq\delta r_0$. Here we have used (\ref{6.97}) in the second inequality. Then, let $H$ be the composition of $H_1$ with $H_2$,
for sufficiently small $r_0>0$, it yields a homotopy $H:\overline{B_{r_0}(p_i)}\cap L_{\beta_i+\nu}\times [0,1]\rightarrow
\overline{B_{r_0}(p_i)}\cap L_{\beta_i+\nu}$ such that
$$B_{r_0}(p_i)\cap L_{\beta_i+\nu}\subset B_{\delta r_0}^+\cup(\partial B_{r_0}(p_i)\cap L_{\beta_i+\nu})$$
and
$$H(\cdot, r)|_{\partial B_{r_0}(p_i)\cap L_{\beta_i+\nu}}=id,\hspace{2mm}0\leq r\leq 1.$$
Composing $H$ with $U(T,\cdot)$ where $T=T(u_0)=\inf\{t\geq 0: E_f(U(t,u_0))\leq \beta_i-\nu\}$
for $u\in L_{\beta_i+\nu}$. From (\ref{7.46}) and (\ref{7.47}), since the transversal time of the annular region
$L_{\beta_i+\nu}\cap(B_{r_0}(p_i)\setminus B_{r_0/4}(p_i))$ is uniformly positive, then it follows that
$U(T,\partial B_{r_0}(p_i)\cap L_{\beta_i+\nu})\subset L_{\beta_i-\nu}$. Therefore, the proof can be followed as in part (iii).
\end{proof}

This completes the proof of Proposition \ref{prop7.1}.

\section{Concluding remarks}

We have proved that Theorem \ref{thm1.1} is true when $n\geq 2$. The natural question would be:
is Theorem \ref{thm1.1} true when $n=1$? We conjecture the answer is yes. In fact, we only used the assumption that $n\geq 2$
in section \ref{section6.2}. Especially, we need to used the assumption $n\geq 2$ in the proof of
Lemma \ref{lem6.2} and \ref{lem6.7}. So if one can prove the results in section \ref{section6.2}
for the case when $n=1$,
one would be able to prove Theorem \ref{thm1.1} by following the same arguments of the remaining part of this paper.

One would also like to study the largest possible number $\delta_n$ in the simple bubble condition
$$\max_{S^{2n+1}}f/\min_{S^{2n+1}}f<\delta_n$$
such that Theorem \ref{thm1.1} holds. In Theorem \ref{thm1.1}, we have $\delta_n=2^{\frac{1}{n}}$. Is it the best possible?

\bibliographystyle{amsplain}

\end{document}